\documentclass[10pt]{amsart}
\usepackage{geometry}                
\usepackage{graphicx}
\usepackage{amssymb}
\usepackage{epstopdf}
\usepackage{amsmath}
\usepackage{graphicx}  \usepackage{epstopdf}
 \usepackage[colorlinks=true]{hyperref}
\hypersetup{urlcolor=blue, citecolor=red}

\usepackage{placeins} 
\usepackage{flafter} 

\usepackage[usenames,dvipsnames]{xcolor}

\usepackage[normalem]{ulem}

\newtheorem{teorema}{Theorem}[section]
\newtheorem{theorem}{Theorem}
\newtheorem{corollary}[teorema]{Corollary}
\newtheorem{lemma}[teorema]{Lemma}
\newtheorem{proposition}[teorema]{Proposition}

\theoremstyle{definition}
\newtheorem{definition}[teorema]{Definition}
\newtheorem{remark}[teorema]{Remark}

\newtheorem{example}{Example}[section]
\newtheorem*{theorem*}{Theorem}

\usepackage{mathtools}
\usepackage{amssymb}
\usepackage{enumitem,mathrsfs}
\usepackage{stackrel}

\setlist{noitemsep,topsep=0pt,partopsep=0pt}

\newcommand{\ZZ}{\ensuremath{{\mathbb Z}}}
\newcommand{\CC}{\ensuremath{{\mathbb C}}}
\newcommand{\CW}{\ensuremath{{\widehat{\mathbb C}}}}
\newcommand{\RR}{\ensuremath{{\mathbb R}}}
\newcommand{\NN}{\ensuremath{{\mathbb N}}}

\newcommand{\HH}{\ensuremath{{\mathbb H}}}

\newcommand{\E}{\ensuremath{{\mathscr E}}}
\newcommand{\R}{\ensuremath{{\mathcal R}}}

\font\myfont=cmr10 at 12pt
\newcommand{\e}{{\text{\myfont e}}}

\newcommand{\abs}[1]{\left\lvert#1\right\rvert}

\renewcommand{\Re}[1]{{\mathfrak{Re}\left(#1\right)}}
\renewcommand{\Im}[1]{{\mathfrak{Im}\left(#1\right)}}

\newcommand{\del}[2]{\frac{\partial #1}{\partial #2}}
\newcommand{\ddel}[2]{\dfrac{\partial #1}{\partial #2}}

\usepackage{tikz,etoolbox}
\newcommand{\circled}[1]{ \text{\tikz[baseline=(char.base)]{
    \node[shape=circle,draw,inner sep=1pt] (char) {$#1$};}} }
\robustify{\circled}
\newcommand{\raiz}[1]{ \text{\tikz[baseline=(char.base)]{
    \node[shape=circle,draw,inner sep=0.001pt] (char) {
    \text{\tikz[baseline=(char.base)]{
    \node[shape=circle,draw,inner sep=1pt] (char) {$#1$};}}
    };}} }
\robustify{\raiz}

\usepackage{graphicx}
\makeatletter
\newcommand*\bigcdot{\mathpalette\bigcdot@{.5}}
\newcommand*\bigcdot@[2]{\mathbin{\vcenter{\hbox{\scalebox{#2}{$\m@th#1\bullet$}}}}}
\makeatother

\newlength\replength
\newcommand\repfrac{.33}

\newcommand\rulewidth{.6pt}
\newcommand\tdashfill[1][\repfrac]{\cleaders\hbox to \replength{%
  \smash{\rule[\arraystretch\ht\strutbox]{\repfrac\replength}{\rulewidth}}}\hfill}
\newcommand\tabdashline{%
  \makebox[0pt][r]{\makebox[\tabcolsep]{\tdashfill\hfil}}\tdashfill\hfil%
  \makebox[0pt][l]{\makebox[\tabcolsep]{\tdashfill\hfil}}%
  \\[-\arraystretch\dimexpr\ht\strutbox+\dp\strutbox\relax]%
}
\newcommand\tdotfill[1][\repfrac]{\cleaders\hbox to \replength{%
  \smash{\raisebox{\arraystretch\dimexpr\ht\strutbox-.1ex\relax}{.}}}\hfill}

\title[Geometry of transcendental singularities]
{Geometry of transcendental  singularities of complex analytic functions and vector fields}

\author[Alvaro Alvarez--Parrilla and 
Jes\'us Muci\~no--Raymundo]{}

\makeatletter
\@namedef{subjclassname@2020}{%
\textup{2020} Mathematics Subject Classification}
\makeatother



 \email{alvaro.uabc@gmail.com}
 \email{muciray@matmor.unam.mx}

\begin{document}
\maketitle

\centerline{\scshape Alvaro Alvarez--Parrilla
}
\medskip
{\footnotesize
 \centerline{Grupo Alximia SA de CV}
   \centerline{Ensenada, Baja California, CP 22800, M\'exico}
} 

\medskip

\centerline{\scshape Jes\'us Muci\~no--Raymundo}
\medskip
{\footnotesize
\centerline{Centro de Ciencias Matem\'aticas}
\centerline{Universidad Nacional Aut\'onoma de 
M\'exico, Morelia, M\'exico} 
}

\bigskip
\medskip

\centerline{
\dedicatory{
\it To Professor Alberto Verjovsky Sol\'a}
}

\begin{abstract}{
On Riemann surfaces $M$, there exists a canonical correspondence between a possibly
multivalued function $\Psi_X$ whose differential is single valued (\emph{i.e.} an additively automorphic singular complex
analytic function) and a vector field $X$.
From the point of view of vector fields, 
the singularities that we consider are zeros, poles, 
isolated essential singularities and 
accumulation points of the above. 
The theory of singularities of the inverse function 
$\Psi_X^{-1}$ is extended
from meromorphic functions to additively automorphic singular complex analytic functions. 
The main contribution is a complete characterization of when a singularity of $\Psi_X^{-1}$ 
is either algebraic, logarithmic or arises from a zero with nonzero residue of $X$. 
Relationships between analytical properties of $\Psi_X$, singularities of $\Psi_X^{-1}$ 
and singularities of $X$ are presented.  
Families and sporadic examples showing the geometrical richness of vector fields on the neighbourhoods of 
the singularities of $\Psi_X^{-1}$ are studied. 
As applications we have; 
a description of the maximal univalence regions for complex trajectory solutions of a vector field $X$, 
a geometric characterization of the incomplete real trajectories of a vector field $X$, and  
a description of the singularities of the vector field associated to the Riemann $\xi$ function.
}
\end{abstract}

\setcounter{tocdepth}{2} 
\tableofcontents
%
\theoremstyle{plain}
\newtheorem*{theo1}{Theorem \ref{singularidades-algebraicas-y-logaritmicas}}

\theoremstyle{plain}
\newtheorem*{theo2}{Theorem \ref{teo-logaritmicas-separada}}

\theoremstyle{plain}
\newtheorem*{theo3}{Theorem \ref{puntos-ideales-en-terminos-de-X}}

\theoremstyle{plain}
\newtheorem*{theo4}{Theorem \ref{familias-r-d}}

\theoremstyle{plain}
\newtheorem*{theo5}{Theorem \ref{familias-P-r}}

\theoremstyle{plain}
\newtheorem*{theo6}{Theorem \ref {teo:cubierta-universal-RX}}

\theoremstyle{plain}
\newtheorem*{theo7}{Theorem \ref{teo:tray-asintoticas-e-incompletas}}

\theoremstyle{plain}
\newtheorem*{theo8}{Theorem \ref{Teo:ValAsintotico-TrayIncompleta}}

\section{Introduction}

Essential singularities of meromorphic functions
$\Psi$ on $\CC$ are a natural source of 
intricate/complex
behaviour in analysis,
iteration of functions and 
differential equations, 
among others topics.
From a geometrical point of view, 
in 1914
F. Iversen \cite{Iversen} introduced the 
\emph{ideal points}
associated to a 
\emph{singularity of the inverse function
$\Psi^{-1}$} by defining neighbourhoods
$U_a(\rho)\subset \CC$,
where $a \in \CW_t$ is a 
singular value (\emph{i.e.}
a critical or an asymptotic value of $\Psi$).
Recently,
W. Bergweiler and A. Eremenko have contributed
in this direction, 
mostly applying their work to 
holomorphic dynamics,
see
\cite{BergweilerEremenko}, \cite{EremenkoReview}.
A great part of the complexity of an essential singularity
is that its description 
can require several ideal 
points, each with different behaviour.
For meromorphic functions $\Psi$, 
the ideal points 
are analytically classified as follows:

\noindent
$\bigcdot$
algebraic singularities of the 
inverse function $\Psi^{-1}$ and

\noindent
$\bigcdot$
transcendental singularities of the 
inverse function $\Psi^{-1}$.

\smallskip

We wish to extend this geometric perspective, 
more precisely the study via ideal points 
to not necessarily isolated essential 
singularities of the following:

\begin{enumerate}[label=\roman*),leftmargin=*]
\item
vector fields $X$
and, as a natural consequence,

\item 
certain multivalued functions
$\Psi_{X}$ associated to $X$.
\end{enumerate} 

\smallskip

Our framework is as follows.
Let $M$ be a connected, not necessarily 
compact, Riemann surface. 
By definition, a 
\emph{singular complex analytic function}
$\Psi_X : M \longrightarrow \CW_t$
can admit accumulations of zeros, poles
and/or essential singularities.
Throughout the work,
singular complex analytic means the analogous properties
for vector fields and 1--forms on $M$.
In addition, $\Psi_X$ is 
\emph{additively automorphic} when 
its differential 
$d\Psi_X$ is a singular complex analytic 1--form
on $M$.
Note that if $\Psi_X$ is additively automorphic,
then it can be single or multivalued.
The adjectives single--valued and multivalued 
shall be understood in the strict sense.
Thus, the concept of
\emph{additively automorphic 
singular complex analytic function $\Psi_X$} 
makes sense and includes the meromorphic case.

We recall the natural correspondence,
which will be used throughout  the work,
between a function, a vector field and an
associated Riemann surface. 
Let
$\Psi_{X}: M \longrightarrow \CW_t$ be
an additively automorphic singular complex analytic function.
The
\emph{singular complex analytic vector field $X$} 
on $M$
canonically associated to $\Psi_X$ 
is defined by 
$d\Psi_X (X) \equiv 1$,
see Diagram \ref{diagramacorresp} and 
\S\ref{generalidades-seccion-2}.
Conversely, 
given a complex analytic vector field $X$, the associated
$\Psi_X$ is a (generically multivalued)
additively automorphic singular complex analytic function.
The third element in the correspondence is
the Riemann surface $\R_X \subset M \times \CW_t$, 
roughly speaking the graph of $\Psi_{X}$.

The differential 
$d\Psi_X$ is the \emph{1--form of time of $X$}.
By definition, 
the \emph{residue of $X$ at a point} is
the residue of 
the 1--form of time $d\Psi_X$ at the point.
In particular, $\Psi_X$ is
single--valued if and only if $d\Psi_X$ has 
all its residues and periods\footnote{
As usual, the period of $d\Psi_X$ along $\beta$
is the integral of $d\Psi_X$ in $\beta$, 
where $[\beta]$ is in a basis of $H_1 (M,\ZZ)$ 
and does not enclose an isolated singularity or a conformal puncture of $M$.
}
equal to zero.
For technical reasons, throughout the entire work 
we require that the set of points where the 
1--form of time $d\Psi_X$
has nonzero residues, be numerable.

\smallskip
The classical theory of singularities of
the inverse function
$\Psi_X^{-1}$ fails for multivalued additively
automorphic singular complex analytic functions.
As a simple example, 
consider a zero of $X$ with nonzero residue
which gives origin to 
an essential singularity of  $\Psi_X^{-1}$ at
the singular value $\infty\in\CW_t$.
Thus, by the classical Casorati--Weierstrass theorem, 
the image of any neighbourhood of $\infty\in\CW_t$ is dense in $M$, 
\emph{i.e.} the neighbourhoods $U_\infty(\rho)$ are not 
useful in order to distinguish ideal points.

As a valuable central result, 
in \S\ref{caso-multivaluado-para-Psi}
we extend Iversen's theory to also 
hold for additively automorphic 
singular complex analytic functions
$\Psi_X$,
by introducing the 
\emph{fundamental domain $\Lambda$ of $\Psi_X$},
which is essentially
a maximal univalence region for $\Psi_X$.

The usefulness of vector fields $X$ in the study 
of functions $\Psi_X$
can be roughly stated as follows.
The vector field 
distinguishes the finite and 
infinite singular values $a\in\CW_t$
of $\Psi_X$  and its ideal points $U_a$ in
a clear geometric way.  
A natural/heuristic idea of this
is to exploit the 
phase portrait of the real part $\Re{X}$. 
This method allows us to describe 
the logarithmic singularities of $\Psi^{-1}_X$ in geometric terms.
Namely, the exponential tracts of $\Psi_X$ 
can be naturally classified as elliptic and hyperbolic
tracts, see
Figures \ref{2-sectores},
\ref{flores-elipticas-hiperbolicas}
and \ref{album-afin-AAP}. 
This leads us to the following:

\emph{Ansatz: The ideal points or singularities of $\Psi_X^{-1}$ can be understood as 
the points of the ideal boundary\,\footnote{
For the sake of simplicity we 
consider algebraic singularities also
as ideal points, even though they are not on the boundary \emph{per se}.} 
of $M$ 
minus the essential singularities and the multivalued
locus of $\Psi_X$. 
In other words, the ideal points are the branch 
points of the Riemann surface $\R_X$.}

Regarding the singularities of 
the inverse for single--valued $\Psi_X$
on $M=\CC$,
the cases of algebraic and logarithmic singularities are understood best. Recall 
the following well known classical result.

\begin{theorem*}[R.~Nevanlinna, \cite{Nevanlinna1} Ch. XI, \S1.3]
A transcendental singularity of $\Psi^{-1}$ 
over an isolated asymptotic value is logarithmic.
\end{theorem*}

\noindent 
We shall prove a stronger version of the above result
(the if and only if assertion and 
the extension to the multivalued case).
For this,
we require the following definitions
and methods suggested by the above ansatz.
Roughly speaking, a 
\emph{$\star$--transcendental singularity
of $\Psi_X^{-1}$}
arises from a pole of $d \Psi_X$ with
nonzero residue, see Definition 
\ref{definicion-singularidades-en-Lambda}.
Secondly, a singularity $U_a$ is
\emph{separate} if 
for a small enough $\rho>0$ the neighbourhood 
$U_{a}(\rho)$ does not intersect any other neighbourhood 
of another ideal point;
Definition \ref{separate} provides full details,
see
Examples \ref{no-finite-asymptotic-values}--\ref{Ejemplo-IntExpExp} 
in \S\ref{ejemplos-seccion}.
Regarding Nevanlinna's result,
for single--valued $\Psi_X$,
a non isolated singular value 
can support separate (including logarithmic) and 
non separate singularities.
Our result 
covers the single and multivalued cases.

\begin{theo2}[Separate singularities]
Let 
$\Psi_{X}: M \longrightarrow \CW_t$
be an
additively automorphic 
singular complex analytic function.
A singularity $U_a$ of $\Psi_X^{-1}$
is separate if and only if $U_a$ is 
one of the following:

\begin{enumerate}[label=\arabic*),leftmargin=*]
\item
algebraic, 

\item
$\star$--transcendental, 

\item 
logarithmic. 
\end{enumerate}
\end{theo2}

Noting that 
the geometry\footnote{By geometry, 
we understand
the geodesics described by $\Re{X}$, 
respect to the singular flat metric from $X$, and
the topology of the phase portrait of $\Re{X}$.} 
of transcendental separate singularities is independent
of the value of the corresponding residue,
it is natural to ask:
which new phenomena appear for 
multivalued $\Psi_X$?

\noindent
Using
Definition \ref{definicion-singularidades-en-Lambda},
the relationship between the singularities of 
$\Psi_X^{-1}$ and the singularities of $X$
is the statement of   
{\bf Theorems} \ref{singularidades-algebraicas-y-logaritmicas} 
and \ref{puntos-ideales-en-terminos-de-X}. 
A rough description 
of the relationship
is as follows:
\begin{equation*}
\begin{array}{rl}
\begin{array}{r}
\hbox{algebraic } \\
\hbox{singularity of }\Psi_X^{-1}
\end{array} & \left\{
\begin{array}{l}
\hbox{poles of $X$, $\mathcal{P}, $} \\
\hbox{zeros of $X$ with residue zero, $\mathcal{Z}_0$, } \\
\end{array} \right. 
\\
& \vspace{-.3cm} \\
\begin{array}{r}
\hbox{transcendental } \\
\hbox{singularity of }\Psi_X^{-1}
\end{array} & \left\{
\begin{array}{l}
\hbox{zeros of $X$ with nonzero residue,  $\mathcal{Z}_{R}$,} \\
\hbox{essential singularities of $X$ 
with residue zero, $\mathbb{E}_{0}$, } \\
\hbox{essential singularities of $X$ 
with nonzero residue, $\mathbb{E}_{R}$,} \\
\text{essential singularities of $X$ 
\emph{without residue}, $\mathbb{E}_{nR}$.}
\end{array} \right.
\end{array}
\end{equation*}

\noindent
The adjective \emph{without residue} means that the residue
of $d\Psi_X$ 
at the respective essential singularity 
is not well defined
(for example, if it is an accumulation of singularities with 
nonzero residue).
\emph{The above relationship is far from being a bijection;
an essential singularity of $X$  
gives rise to none, one or more than one transcendental singularity
of $\Psi_X^{-1}$.
}
Example \ref{sing-campo-no-implica-sing-psi}
provides a singularity of $X$ in $\mathbb{E}_{nR}$ that 
does not allow a singularity of the inverse $\Psi_X^{-1}$. 
In addition, 
Example \ref{una-sola-singularidad-de-la-inversa}
provides a singularity of $X$ in $\mathbb{E}_{0}$ with exactly
one singularity of $\Psi_X^{-1}$.
Theorem 4 below describes generic $X$, where 
an essential singularity  supports and even number
of singularities of $\Psi_X^{-1}$.
As a fortunate coincidence,

\centerline{$\mathcal{S}= \mathcal{P}\cup \mathcal{Z}_0 \cup
\mathcal{Z}_R \cup \mathbb{E}_0 
\cup \mathbb{E}_R \cup \mathbb{E}_{nR}$}

\noindent 
also refers in a uniform way to the singularities of 
$X$, $\omega_X$ and $\Psi_X$.
For example, 
$z_{\tt s} \in \mathcal{P}$ 
denotes a pole of $X$, simultaneously a zero
of $\omega_X$
and a critical point of $\Psi_X$.

In \S \ref{ejemplos-seccion},
we study finite dimensional holomorphic families of
additively automorphic  functions $\Psi_X$ with 
essential singularities.
The use of the fundamental domain $\Lambda$ 
technique allows us to
reduce their study to single--valued functions 
$\Psi_{X, \, \Lambda}$. 
As one of the contributions of this work,
the singularities of $\Psi_X^{-1}$
for the multivalued case  
are considered for the first time 
in the literature. 
In particular, 
vector fields with 
essential singularities that are accumulation points
of zeros with nonzero residue
are examined. 
 
We consider the families 

\centerline{
$\E(s, r,d)= \left\{
X(z) = \frac{Q(z)}{P(z)} \e^{E(z)} \del{}{z}
\ \Big\vert \
Q, \, P,  \, E \in \CC[z] \hbox{ of degree }
s, r, \, d \geq 1 
\right\}$.
}

\noindent 
{\bf Theorem} \ref{familias-s-r-d} describes 
the associated additively automorphic functions 

\centerline{$
\Psi_X (z)
= \int^z \frac{P(\zeta)}{Q(\zeta)} \e^{-E(\zeta)} d \zeta ,
\ \ \ 
\text{ on } 
M= \CW_z .
$}

\noindent  
All the singularities of $\Psi_X^{-1}$ are separate. 
The zeros and poles with zero residue of $d\Psi_X$ correspond to algebraic singularities of $\Psi_X^{-1}$.
The poles with nonzero residue of $d\Psi_X$ correspond to 
$\star$--transcendental singularities of $\Psi_X^{-1}$.
The essential singularity consists of $2d$ logarithmic singularities: 
$d$ elliptic tracts and $d$ hyperbolic tracts 
equidistributed about $\infty\in\CW_z$.

\noindent 
These functions are the simplest in several deep
subjects, 
determining functions with a finite 
number of singular values. 
They are related to 
the Schwartzian second order differential equation,
see
\cite{Hille}, \cite{Nevanlinna1} Ch. XI, 
and appear in the deformation of ramified coverings 
\cite{Taniguchi1} and \cite{Taniguchi2}. 
Also see our
previous work \cite{AlvarezMucinoII} and references therein.

As second kind of families, 
{\bf Theorem} \ref{familias-P-r}
studies the functions 

\centerline{
$\Psi_X(z) = R \big( \e^{2\pi iz/ T}\big)$,
}

\noindent 
where $R(w)$ are rational functions of 
degree ${\tt r} \geq 1$, $T \in \CC^*$.
A systematic description that depends
on the behaviour of $R$ is provided.
Note that this family is the simplest having periodic
functions and/or vector fields
where an accumulation of zeros and poles at the 
essential singularity $\infty$ appear.

\noindent 
In \S\ref{sporadic-examples},
the geometrical richness of the behaviour of the singularities of the inverse function $\Psi_X^{-1}$
that may appear, even in the single--valued case, is explored.
Examples of single--valued functions are considered from the 
perspective of vector fields;
also examples of vector fields $X$ which give rise to multivalued
additively automorphic $\Psi_X$ are fully explained.

In \S\ref{aplicaciones} we provide three applications.
In \S\ref{flujo-maximal}, we  obtain
a description of the 
maximal region for complex trajectory solutions of $X$, which \emph{a priori} are multivalued.

\begin{theo6}[Maximal univalence region for trajectory solutions]
Let $X$ be a singular complex analytic vector field
on $M$.
The maximal univalence region for a non stationary
complex solution $z(t)$ of $X$ is
\begin{equation*}
\mathscr{D}_X
=
\{ (z,\Psi_X(z) ) \ \vert \
z \in M \backslash \mathcal{S} \}.
\end{equation*}
 
\noindent 
Moreover, 
$\mathscr{D}_X$ is independent of 
the initial condition
$z_{\tt o} \in 
M \backslash \mathcal{S}$.
\end{theo6}

\hskip0.7cm
As a second application,
an \emph{incomplete trajectory} of $X$
is a solution of $\Re{X}$ with
a strict subset of $\RR$ as maximal
domain of existence.
Clearly, from the local analytic normal form of $X$, 
each pole $p$ of $X$  provides a 
finite number of
incomplete trajectories.
In Proposition
\ref{prop:infinidad-trayectorias-incompletas},
the following natural result
is presented.

\noindent 
\emph{Every 
nonrational, 
singular complex analytic vector field $X$ on a compact Riemann surface $M_\mathfrak{g}$, 
of genus $\mathfrak{g}$, 
has an infinite number of incomplete trajectories. }
At a pole or essential singularity of $X$,
the following clear mechanism occurs.

\begin{theo7}[Incomplete trajectories and finite singular values]
Let $X$ be a singular complex analytic vector field on
$M$.
The following statements are equivalent. 
\begin{enumerate}[label=\arabic*),leftmargin=*]
\item
There exists an
incomplete trajectory $z({\tt t} )$ of $X$ 
having $\alpha$ or $\omega$--limit at 
$z_{\tt s} \in M$.

\item
There exists a finite singular value 
$a \in\CC_t$ of $\Psi_X$, whose asymptotic path 
$\alpha_a ({\tt t} )$ is a trajectory of $\Re{X}$ ending at $z_{\tt s} \in M$.
\end{enumerate}
\end{theo7}

\noindent 
This raises a natural question:  
which neighbourhoods 
$U_a(\rho)$ of the singularities of 
$\Psi_X^{-1}$ contain incomplete trajectories,
and how many are there?

\noindent 
As example, 
the  neighbourhoods $U_\infty(\rho)$ of separate 
singularities over $\infty \in  \CW_t$ do not have
incomplete trajectories.
An analogous problem has been recently considered by
J. K. Langley 
\cite{Langley}, \cite{Langley-2} and
\cite{Langley-3}.
As an application 
of Theorem \ref{teo-logaritmicas-separada}, in \S\ref{localizingIncomplete} we prove a 
constructive description of how the incomplete 
trajectories of $X$ on a
Riemann surface $M$ 
arise in a vicinity of an essential singularity.

\begin{theo8}[Localizing incomplete trajectories]
Let $X$ be a singular complex analytic vector field on  
$M$ with an essential singularity at $z_{\tt s} \in M$.

\begin{enumerate}[label=\arabic*),leftmargin=*]
\item
Any neighbourhood $U_a(\rho)$, 
of an essential transcendental singularity $U_a$ of 
$\Psi_X^{-1}$ over a finite 
asymptotic value $a\in\CC_t$,  
contains an infinite number
of incomplete trajectories of $X$.

\item
If $\Psi_X$ has no finite asymptotic values
at $z_{\tt s}$,
then $X$ has an infinite number of poles accumulating at 
$z_{\tt s} \in M$.
\end{enumerate}
\end{theo8}

\noindent
In other words, any neighbourhood 
of an essential singularity of $X$ has an 
infinite number of incomplete trayectories.

As a third and final application, 
in \S\ref{xi-campo-Broughan-subseccion}, 
by recalling the work of K. Broughan \cite{Broughan}
on the Riemann $\xi$--vector field
$X_\xi(z)=\xi(z)\del{}{z}$,  
we show that it
is not holomorphically equivalent to a 
pullback of a periodic vector field with a finite 
number of distinct residues.
Furthermore, 
we show that 
$\Psi_{X_\xi}^{-1}$ 
has two logarithmic singularities 
over finite asymptotic values, 
whose hyperbolic tracts are the left and right half 
planes delimited by the critical strip. 
In addition,
$\Psi_{X_\xi}^{-1}$ has
an infinite number of $\star$--transcendental 
singularities over $\infty$ corresponding to the
zeros with nonzero residue in the critical strip;
see Proposition \ref{prop:Xi-Riemann}.

In \S \ref{Future-work-subseccion},
some possible avenues of further research are
presented.

Finally, we make a few comments from a panoramic viewpoint:

\noindent
$\bigcdot$
In Riemann surface theory, all the 
meromorphic functions
can be constructed by using  the elementary blocks 
$\{ z \mapsto z^d\}$, \emph{i.e.} the algebraic singularities of the inverse.

\noindent 
$\bigcdot$
Many singular complex analytic functions
can be constructed by using two new elementary blocks: 
hyperbolic and elliptic tracts,  \emph{i.e.} the logarithmic singularities of the inverse.

\noindent 
$\bigcdot$
In the general case of singular complex analytic functions, 
an infinite number of new blocks appear: those arising from the non separate 
singularities of the inverse.

\noindent
$\bigcdot$
Furthermore, for multivalued functions, 
the $\star$--transcendental singularities 
of the inverse 
complete the above elementary blocks.

\noindent 
$\bigcdot$
In any case, as the examples throughout the text show, 
clear patterns can be recognized
by using the above elementary building blocks. 


\section{General facts about functions and vector fields}
\label{generalidades-seccion-2}

\subsection{Functions and vector fields on Riemann surfaces}

Let $M$ be a Riemann surface, not necessarily compact, 
if we assume that ${\tt p}$
is a conformal puncture\footnote{ 
By definition,
$M \cup \{{\tt p}\}$ admits a holomorphic chart
$\phi_{\tt j}: V_{\tt j} \subset M
\longrightarrow 
D(0, 1) \subset \CC$ to the unitary disk
with $\phi_{\tt j}({\tt p})=0$,
compatible with the atlas of $M$.%
}
of $M$, then 
we consider ${\tt p}$ in $M$.
Thus, our Riemann surface $M$ includes their 
conformal punctures.

\begin{definition}
\label{definicion-SCA} 
On $M$, the adjective
\emph{singular complex analytic}
for   
functions, vector fields, 1--forms 
and quadratic differentials, 
means that 
they may have accumulation of zeros, 
poles and/or essential singularities.
\end{definition}

The singular complex analytic category 
includes holomorphic and meromorphic 
objects on compact Riemann surfaces, 
which are not transcendental meromorphic: 
{\it i.e.} 
singular complex analytic is a larger class.

\begin{definition}[\cite{Berenstein-Gay} p.\,579]
\label{aditivamente-automorfa}
A multivalued or single--valued analytic function 
$\Psi_X$ on $M$
is \emph{additively automorphic} when 
its differential $d\Psi_X$ is a single--valued 1--form. 
\end{definition}

\noindent
Of course any single--valued singular complex analytic function is additively automorphic; 
however, not all multivalued singular complex analytic functions are additively automorphic.

\smallskip
\noindent
{\bf Notation.} 
1. 
An \emph{additively automorphic
singular complex analytic function $\Psi_X$} on $M$ 
satisfies Definitions   
\ref{definicion-SCA} and  \ref{aditivamente-automorfa}.

\noindent 2. 
A \emph{single--valued additively automorphic
singular complex analytic function $\Psi_X$} on $M$
is strictly single--valued. 

\noindent 3. 
A \emph{multivalued additively automorphic
singular complex analytic function $\Psi_X$} on $M$
is strictly multivalued.

\smallskip

The advantage of the
subscript $X$ is explained below. 

\smallskip 

Throughout this work, 
we assume that all the 
vector fields $X$ 
are not identically zero 
and that the functions $\Psi_X$
are not identically constant. 
The formal expression 
of a vector field $X$ in holomorphic charts 
$\{
\phi_{\tt j} : V_{\tt j} \subset M \longrightarrow 
\CC_z \}$ 
must be
$\{ f_{\tt j}(z) \del{}{z} 
\ \vert \  
z \in \phi_{\tt j}( V_{\tt j})  \}$; 
as far as possible, 
we avoid this cumbersome notation.

\smallskip

\emph{From 
additively automorphic 
singular complex analytic functions 
to 
singular complex analytic vector fields.}
Let 

\centerline{ 
$\Psi_X : M \longrightarrow \CW_t$
}

\noindent
be an additively automorphic 
singular complex analytic function
(probably not well defined at every point 
since we are abusing notation).
Since
its differential is single--valued,
the canonical associated singular complex analytic
vector field is

\centerline{
$X(z) = \frac{1}{\Psi_{X}^\prime (z)} \del{}{z}
\ \ \ \hbox{ on } M$.
}

\emph{From 
singular complex analytic vector fields 
to 
additively automorphic 
singular complex analytic functions.}
Let 
\begin{equation}\label{campo-vectorial}
X(z)= f(z) \del{}{z}
\end{equation}
\noindent 
be a singular complex analytic 
vector field on $M$. 
By definition, the singular complex analytic
\emph{1--form of time of $X$} is 
\begin{equation}
\label{1-forma-de-tiempo}
\omega_X(z)= \frac{dz}{f(z)} . 
\end{equation}

\noindent 
We want to define $\Psi_X(z)= \int^z \omega_X$
with  single--valued  $\omega_X$.

\begin{remark}
The residue of $\omega_X$ at $z_0 \in M$,
\begin{equation}
\label{residuo-formula-general}
Res(\omega_X, z_0)=
\frac{1}{2 \pi i }\int_\gamma {\omega_X} \in \CC
\end{equation}

\noindent 
is well defined
if and only if $z_0$ is 

\begin{enumerate}[label=\roman*),leftmargin=*]
\item
a regular point of $X$
(as usual the counterclockwise path $\gamma$
encloses $z_0$ and the integral is zero),

\item
an isolated singularity of $X$, 
or

\item
a nonisolated singularity of $X$
which is at most an accumulation of
singular points with residue zero, 
\emph{e.g.} of poles of $X$
(in this case
the path $\gamma$ encloses $z_0$ and those infinite number
of singular points).
\end{enumerate}

\noindent 
By definition, 
the \emph{residue of $X$ at a point $z_0$}
is the residue of $\omega_X$, 
\emph{i.e.} 

\centerline{$Res(X, z_0) \doteq Res(\omega_X, z_0)$.}
\end{remark}

\noindent  
The \emph{singularities of $X$}, 
\begin{equation}
\label{singularidades-de-X}
\mathcal{S} = \{  z_{\tt s} \} =
\underbrace{\, \mathcal{Z}_0 \cup  \mathcal{Z}_{R} \, }_{\mathcal{Z}} 
\cup
\, \mathcal{P} 
\cup 
\underbrace{\, \mathbb{E}_0 \cup
\mathbb{E}_{R} \cup \mathbb{E}_{nR} \, }_{\mathbb{E}}     
\subset M, 
\end{equation}

\noindent 
are possibly infinite, of the following kinds:

\smallskip
\noindent 
The \emph{zeros of $X$},

\centerline{
$\mathcal{Z}=\{ q \} = 
\mathcal{Z}_0 \cup \mathcal{Z}_{R}$,}  

\noindent
here $\mathcal{Z}_0$ (resp. $\mathcal{Z}_{R}$)
denotes the zeros of $X$ with residue zero  
(resp. with nonzero residue).

\smallskip
\noindent 
The \emph{poles of $X$}, 

\smallskip 
\centerline{
$\mathcal{P}=\{ p \}$. }

\smallskip

\noindent 
The \emph{essential singularities of $X$},

\smallskip 

\centerline{
$\mathbb{E} =  \{ e \}
=
\mathbb{E}_0 \cup \mathbb{E}_{R} \cup \mathbb{E}_{nR}$, 
}

\smallskip 

\noindent
here $\mathbb{E}_0$ (resp. $\mathbb{E}_{R}$)
denotes the essential singularities of $X$ 
with residue zero  
(resp. with nonzero residue), and 
the points 
$\mathbb{E}_{nR}$ where the residue is not
well defined; these last
are accumulation of
points of $\mathcal{Z}_{R} \cup \mathbb{E}_{R}$.

\noindent
In addition, we introduce the following notations

\smallskip

\centerline{
$\mathcal{S}_0 = \mathcal{Z}_0 \cup \mathbb{E}_0$
\ and \
$\mathcal{S}_{R} = 
\mathcal{Z}_{R} \cup \mathbb{E}_{R}$.}

\noindent
Because of technical reasons, to be used in \S 
\ref{caso-multivaluado-para-Psi}, we require that 
$\mathcal{S}_R$ is at most a numerable set. 
Through all the work, 
$\mathcal{S}$ in \eqref{singularidades-de-X}
also refers to the singularities of $\Psi_X$  and $\omega_X$.
For example, 
in accordance with Equations 
\eqref{campo-vectorial} and
\eqref{1-forma-de-tiempo},
$z_{\tt s} \in \mathcal{Z}$ denotes
a zero of $X$ and simultaneously a pole of $\omega_X$.

\noindent 
Note that 

\centerline{
$M \backslash (\overline{\mathcal{S}_{R}}
\cup \mathbb{E}_{0} )  =
M \backslash (\mathbb{E} \cup  \mathcal{Z}_{R}).
$ }

\smallskip

\noindent
The \emph{additively automorphic singular complex 
analytic function associated to $X$} is 
\begin{equation}
\label{parametro-local}
\Psi_X (z)=\int^z_{z_{\tt o}} \omega_X:
M \backslash 
(\mathbb{E} \cup \mathcal{Z}_{R})
\longrightarrow \CW_t 
\end{equation}

\noindent 
and the initial point of integration is
a nonsingular point
$z_{\tt o}  \in M \backslash \mathcal{S}$; 
for simplicity, we omit it in some instances.

\begin{remark}\label{cuando-Psi-es-multivaluada}
1. 
In \eqref{parametro-local},
the integral function 
$\Psi_X(z)$ is single--valued if and only if 
both

\noindent i)
the residues 
$Res(\omega_X, z_{\tt s})$,
for $z_{\tt s} \in \mathcal{S}$, 
and 

\noindent ii)
the periods
$\int_\beta {\omega_X}$, 
where the class $[\beta]$ is in a basis of 
the fundamental group
$\pi_1 (M)$ and 
does not enclose an isolated singularity,

\noindent 
are both zero.

\noindent 
2. Assertion 1.i is equivalent to
$\mathcal{S}_{R}=
\mathcal{Z}_{R} \cup \mathbb{E}_{R}= \varnothing $.

\noindent 
3. 
The multivaluedness of the integral function
shall be studied in 
\S \ref{caso-multivaluado-para-Psi}.
\end{remark}

\begin{remark}
1. In both cases, single--valued or multivalued,
$\Psi_X$ is a \emph{global flow box} 
that rectifies the corresponding
singular complex analytic vector 
field $X$, thus 

\centerline{$\big( \Psi_X \big) _* X = \del{}{t}$.}

\noindent 
2.
In the language of quadratic differentials, 
$\Psi_X$ is
the global distinguished parameter of $X$,
and we exploit the global nature. 
Clearly, the poles of $X$ determine zeros of $\omega_X$
and critical points of $\Psi_X$ in $M$.
\end{remark}

\begin{proposition}[Dictionary between the 
singular analytic objects, 
\cite{MR},
\cite{AlvarezMucino} \S2] 
\label{basic-correspondence}
On a Riemann surface $M$,
there exists a canonical correspondence 
between the following objects.

\begin{enumerate}[label=\arabic*),leftmargin=*]

\item
A singular complex analytic vector field
$X = f(z) \del{}{z}$, as in 
\eqref{campo-vectorial}.

\item 
A singular complex analytic 1--form
$\omega_X = dz / f(z)$, as in 
\eqref{1-forma-de-tiempo}.

\item
An 
additively automorphic 
singular complex analytic function 
$
\Psi_X(z)=\int^z \omega_X $
as in \eqref{parametro-local}. 

\item
An orientable singular complex analytic quadratic differential 
$\mathcal{Q}_X=\omega_X \otimes \omega_X$, where the trajectories  
of $X$ coincide with horizontal trajectories of 
$\mathcal{Q}_X$.

\item
A singular flat metric 
$g_X = \Psi^*  \vert dt \vert $ on $M$, 
which is 
the pullback of the flat Riemannian metric
$\vert dt \vert = d{\tt t}^2 + d{\tt s}^2$, 
$t \doteq {\tt t} + i {\tt s} \in \CC$,
having suitable singularities at $\mathcal{S}$ 
and a unitary geodesic vector field $\Re{X}$.
By abuse of notation,
$(M ,g_X)$ denotes this singular noncompact
Riemannian manifold.

\item
A Riemann surface
$\big( \R_{X},\pi^*_2(\del{}{t}) \big)$
associated to an additively automorphic 
singular complex analytic function $\Psi_X$, 
where
\begin{equation}
\label{R_X-definicion}
\R_{X}= 
\big\{
(z,t) \ \vert \  
t=\Psi_{X}(z), \, 
z \in 
M \backslash 
(\mathbb{E} \cup \mathcal{Z}_R) 
\big\} \subset 
M \times\CW_{t}.
\end{equation}

\end{enumerate} 
 
\noindent 
Diagrammatically, 

\smallskip
\begin{picture}(200,128)(-70,-35)
\put(65,86){$X(z)=f(z)\, \del{}{z} $}
\put(4,53){$\omega_X(z)= \frac{dz}{f(z)}$}
\put(120,53){$\Psi_X (z)= \int\limits^z \omega_X$}
\put(60,-35){$\big((M , g_X ), \Re{X}\big)$ \ .}

\put(58,80){\vector(-1,-1){15}}
\put(44,66){\vector(1,1){15}}

\put(152,66){\vector(-1,1){15}}
\put(138,80){\vector(1,-1){15}}

\put(58,-24){\vector(-1,1){15}}
\put(44,-10){\vector(1,-1){15}}

\put(151,-10){\vector(-1,-1){15}}
\put(138,-23){\vector(1,1){15}}

\put(0,0){$\mathcal{Q}_X(z)= \frac{dz^2}{f(z)^2}
$}
 \put(128,0){$\big(\R_X, \pi^{*}_2 (\del{}{t})\big)$}

\put(36,37){\vector(0,-1){20}}
\put(36,22){\vector(0,1){20}}

\put(155,37){\vector(0,-1){20}}
\put(155,22){\vector(0,1){20}}

\put(-87,15){\vbox{\begin{equation}\label{diagramacorresp}
\end{equation}}}

\put(337,-38){$\qed$}

\end{picture}
\end{proposition}

\begin{remark} 
The correspondence \eqref{diagramacorresp}  
must be understood 
up to choice of initial point $z_{\tt o}$ for
the integral defining the global distinguished parameter. 
Thus,  
$\Psi_X$ and $\Psi_X + c$, 
for $c \in \CC$,   
are considered the same object.
\end{remark} 
 
\begin{example}[Abelian integrals] 
1.
Note that
nonadditively automorphic multivalued functions do 
not produce singular complex analytic vector fields. 
For instance, consider the 
nonadditively automorphic multivalued 
singular complex analtytic function

\centerline{$\Theta(z)=
{\displaystyle \int^{z} }
\dfrac{d\zeta}{\sqrt{P(\zeta)}}: 
\CW_z \longrightarrow \CW_t$,
\ \ \
where $P\in\CC[z]$, $\deg{P}\geq2$.}

\noindent 
Obviously, $\sqrt{P(z)}\del{}{z}$ is not a single--valued vector field on $\CC_z$. 
However, 
on the (hyper) elliptic Riemann surface 
$M=\{ w^2-P(z)= 0\}$,
the integrand $dz/\sqrt{P(z)}$ 
determines a holomorphic 1--form $\omega_X$, 
thus the Abelian integral

\centerline{$\Psi_X(\mathfrak{z})= 
{\displaystyle \int^\mathfrak{z} }
\omega_X: M \longrightarrow \CC$} 

\noindent 
is an 
additively automorphic singular complex analytic function on $M$. An associated meromorphic vector field
$X$ on $M$ is well defined. 

\noindent 
2. Let $\omega_X$ be a meromorphic 1--form on a compact 
Riemann surface $M$. The integral function 
$\Psi_X (z)= \int^z \omega_X : M \backslash \mathcal{Z}_R
\longrightarrow \CW_t$ is an
additively automorphic singular complex analytic function. 
\end{example}

\begin{remark}
Note that vector fields $X$ with 
$\mathbb{E}_{nR} \neq \varnothing$ are quite common,
for instance
see Figure \ref{album-afin-AAP} (d), (e) and (f) discussed  in Examples 
\ref{ejemplo-sin}, 
\ref{una-sola-singularidad-de-la-inversa} 
and \ref{ejemplo-tan}, where
$\mathbb{E}_{nR}$ is an accumulation of 
points of $\mathcal{Z}_{R}$ 
in the first two cases and
an accumulation of $\mathcal{Z}_{R}\cup\mathcal{P}$ in the third case.
See also Example \ref{sing-campo-no-implica-sing-psi}.
\end{remark}

\begin{figure}[htbp]
\begin{center}
\includegraphics[width=0.90\textwidth]{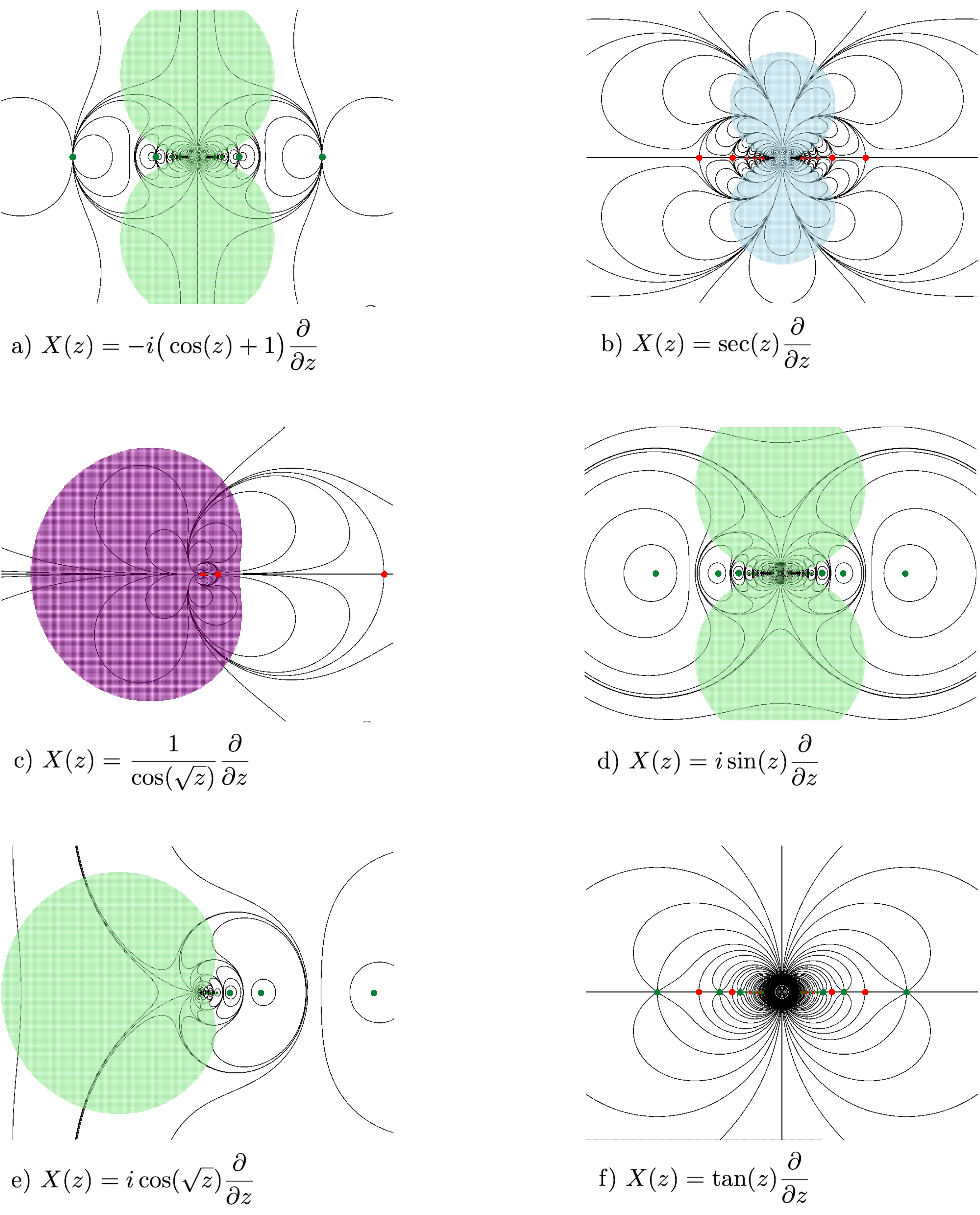}
\caption{
Geometry
of vector fields $X$ on $\CW_z$ 
with an essential singularity
at $\infty$,
with accumulations of poles and/or zeros.
(a) is 
described in
Example \ref{coszMasUno}, 
(b) in Example \ref{no-finite-asymptotic-values}, 
(c) in Example \ref{una-sola-singularidad-de-la-inversa},
(d) in Example \ref{ejemplo-sin}, 
(e) in Example \ref{cos-raiz}
and 
(f) in Example \ref{ejemplo-tan}.
The colouring scheme for neighbourhoods $U_a(\rho)$,
determining singularities of $\Psi_X^{-1}$,
is green for hyperbolic tracts and
blue for elliptic tracts
(to be described in 
Definition \ref{tract-eliptico-hiperbolico} 
and Figure \ref{2-sectores}). 
Moreover, purple region in (c) denotes a
connected component $U_\infty (\rho )$.
Green and red dots represent zeros and poles of $X(z)$ respectively.
It is remarkable,
that G. Gyllstr\"om \cite{Gyllstrom} 
describes intricate phase portraits of ordinary differential equations
one century ago.} 
\label{album-afin-AAP}
\end{center}
\end{figure}

\begin{lemma}\label{Lemma-RX}
With the notation as above.
\begin{enumerate}[label=\arabic*),leftmargin=*]
\item The following diagram of pairs,
(Riemann surface, vector field), commutes

\begin{center}
\begin{picture}(180,65)(0,20)

\put(-131,40){\vbox{\begin{equation}\label{diagramaRX}\end{equation}}}

\put(-40,75){$\big(
(M \backslash 
( \mathbb{E} \cup \mathcal{Z}_{R} )
,X\big) $}

\put(115,75){$\big(\R_X,\pi^*_2(\del{}{t})\big)$}

\put(108,78){\vector(-1,0){60}}
\put(75,85){$\pi_1$}

\put(133,65){\vector(0,-1){30}}
\put(138,47){$ \pi_2 $}

\put(38,65){\vector(2,-1){73}}
\put(55,39){$ \Psi_X $}

\put(115,20){$\big(\CW_t,\del{}{t}\big) $,}

\end{picture}
\end{center}
\noindent 
where
$\pi_1$ and $\pi_2$ are local isometric, possibly branched coverings over 
$(M, g_X)$ and 
$(\CW_t, \vert dt \vert)$ 
as singular Riemannian manifolds, 
respectively.

\item
Moreover, 
$\Psi_X$ is single--valued
if and only if 
the projection 
$\pi_1$ is a biholomorphism between

\centerline{ 
$\big( \R_{X},\pi^*_2(\del{}{t}) \big)$ 
\ and  \   
$\big(M\backslash(\mathbb{E}\cup\mathcal{Z}_{R}),X\big)$.
}

\item 
The (ideal) boundary of $\R_X$ is totally disconnected, separable and compact.
\end{enumerate}
\end{lemma}
\begin{proof}
A proof of (3) can be found in 
\cite{Ahlfors-Sario} Ch.\,I\,\S\,6, or
\cite{Richards} as proposition 3.
\end{proof}

We shall use the abbreviated form $\R_{X}$ instead of the 
cumbersome $\big(\R_{X},\pi^*_2(\del{}{t})\big)$.

\begin{example}
Let $X$ be a singular complex analytic vector field 
on $M= \CW_z$. 

\noindent 
1. By Lemma \ref{Lemma-RX}.2 above,
the Riemann surface
$\R_X$ is biholomorphic to 
$\CW_z \backslash \mathbb{E}$
if and only if every zero or essential singularity
of $X$ has zero residue
(in symbols $\mathcal{Z} = \mathcal{Z}_0$, 
$\mathbb{E} = \mathbb{E}_0$).

\noindent 
2. The Riemann surface
$\R_X$ is the
the universal cover of 
$\CW_z \backslash 
(\mathbb{E} \cup \mathcal{Z})$
if and only if every zero or essential singularity
of $X$ has nonzero residue
(in other words
$\mathcal{Z}_0 \cup \mathbb{E}_{0}  \cup \mathbb{E}_{nR}= \varnothing$).
\end{example}

\begin{definition}
\label{trayectoria-real}
A maximal real \emph{trajectory solution of $X$} is
$
z(  {\tt t} ): (a, b)
\subseteq \RR \longrightarrow M 
\backslash (\mathcal{P} \cup \mathbb{E})$, 
where $a, b \in \RR \cup \{ \mp\infty\}$,
satisfying that 

\centerline{
$\dfrac{dz( {\tt t} )}{d {\tt t} } = f(z(  {\tt t} )), 
\ \ \ 
z(0) = z_{\tt o } \in M \backslash (\mathcal{P} \cup \mathbb{E})$.
}

\noindent  
Equivalently, $z( {\tt t} )$ is a trajectory of the
associated real vector field $\Re{X}$.
\end{definition}

Abusing notation, 
the phase portrait of $X$ means the
portrait of the real vector field $\Re{X}$. 
Moreover, the trajectories 
$z( {\tt t} )$ of $\Re{X}$
coincide with the level sets

\centerline{$\{ \Im{\Psi_X (z)} = c  \}$,
\ \ \
for $c\in \RR$,
}

\noindent 
{\it i.e.} the horizontal 
trajectories of the orientable quadratic differential 
$\mathcal{Q}_X$. 
Whereas, the 
inversion $\Psi_X^{-1}(t)$ of the integral in
Equation \eqref{parametro-local}
provides the
non stationary complex trajectory solutions
of the vector field $X$.

There is a natural advantage of studying 
additively automorphic 
singular complex analytic functions
$\Psi_X$ 
via the associated vector fields $X$, as seen in  
\cite{MR}, \cite{AlvarezMucino}, 
\cite{AlvarezMucino2} and \cite{AlvarezMucinoII}.
Very particular families of vector fields with 
one  essential singularity 
are considered in 
\cite{AlvarezMucino}, \cite{AlvarezMucino2} 
and 
\cite{AlvarezMucinoII}.
Meromorphic vector fields on compact Riemann surfaces
are a current subject of study, see for example in
\cite{Kilmes-Rousseau},
\cite{Dias-Garijo} and references therein.

\subsection{Local theory of vector fields}
\label{comportamiento-local}

\begin{definition}(\cite{AlvarezMucino} \S5.)
\label{sectores-angulares-de-campos}
Let $(\CW, \del{}{z})$ be the holomorphic vector field
on the Riemann sphere with a double zero at $\infty$, 
and let $\overline{\HH}^2 = \{ \Im{z} \geq 0\} \cup 
\{\infty \} \subset \CW$.

\noindent 1)
A {\it hyperbolic sector} is the vector field germ 
$H=\big( (\overline{\HH}^2,0), \del{}{z} \big)$,
as in Figure \ref{3-sectores}.c. 

\noindent 2)
An {\it elliptic sector} is the vector field germ 
$E=\big( (\overline{\HH}^2, \infty), \del{}{z} \big)$,
equivalently
$\big( (\overline{\HH}^2, 0), -w^2\del{}{w} \big)$
when $\{z \mapsto \frac{1}{z}=w\}$,
Figure \ref{3-sectores}.a.

\noindent 3)
A (right) {\it parabolic sector} is the vector field germ 

\centerline{
$
P_+ =
\big( (\{ 0 \leq  \Im{z} \leq h \} \cap\{\Re{z} >0 \}, \infty), \del{}{z} \big)$,}

\noindent 
in addition 
the (left) parabolic sector $P_{-}$ occurs when $\Re{z}<0$;  
$h \in \RR^+$ is a parameter,
see Figure \ref{3-sectores}.b.
\end{definition}

The sectors are germs of flat Riemannian manifolds with boundary provided with a complex vector field;
in \cite{AlvarezMucino} \S5, 
we describe their properties.
Thus, we  say that 
$X$ has a 
hyperbolic, elliptic or parabolic 
when $\Re{X}$ has it.  

\begin{figure}[htbp]
\begin{center}
\includegraphics[width=0.5\textwidth]{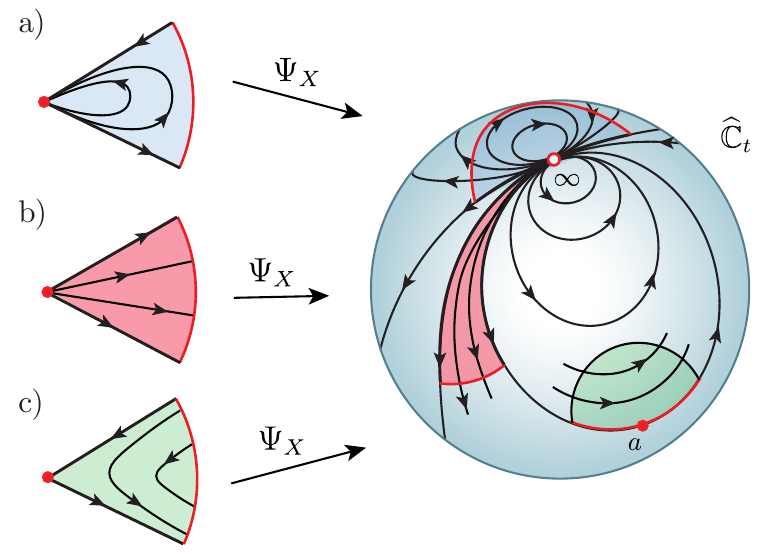}
\caption{
(a)
Elliptic $E$, (b) parabolic $P$, (c) hyperbolic $H$ sectors
of $\Re{X}$.
The left drawing sketches the sphere $(\CW_t, \del{}{t})$  
describing their embeddings under $\Psi_X$.
}
\label{3-sectores}
\end{center}
\end{figure}

The following result appears 
in the theory of quadratic differentials 
\cite{Jenkins}, 
\cite{Strebel},
\cite{Ahlfors}
and in complex differential equations
\cite{Hajek-1},  
\cite{Hajek-2},
\cite{Hajek-3},
\cite{Brickman-Thomas},
\cite{Needhan-King},
\cite{Garijo-Gasull-Jarque}.
See Figure \ref{2-forma-normal}.

\begin{proposition}[Local analytic normal forms at 
zeros and poles of $X$]
\label{la-topologia-dice-la-singularidad}
Let $\big( (\CC, z_0), X \big)$ be a germ of
a singular complex analytic vector field; in each item
the corresponding assertions are equivalent.

\smallskip 

\begin{enumerate}[label=\arabic*),leftmargin=*]

\item
i) 
$X$ is holomorphic and nonzero at $z_0$.

\noindent 
ii)
$\Re{X}$ is topologically equivalent to $\Re{\del{}{t}}$.

\noindent 
iii) Up to local biholomorphism $X$ is $\del{}{z}$.

\smallskip

\item 
i) $X$ has a zero at $z_0=q$ of multiplicity
$s \geq 1$.

\noindent 
ii) For multiplicity one
$\Re{X}$ is a
source, sink or center;
for multiplicity at least two
it
admits a decomposition with
$2s -2 \geq 2$ 
elliptic sectors and 
zero or one parabolic sectors.

\noindent 
iii) Up to local biholomorphism $X$ is 
$\frac{(z-q)^s}{\lambda (z-q)^{s-1} - (s-1)} \del{}{z}$,
$\lambda \in \CC$.

\smallskip

\item 
i) $X$ has
a pole at $z_0=p$ of multiplicity $-k \leq -1$.

\noindent 
ii) $\Re{X}$ admits a decomposition 
with
$2k +2$ hyperbolic sectors. 

\noindent 
iii) Up to local biholomorphism $X$ is 
$\frac{1}{(z-p)^k} \del{}{z}$.

\smallskip

\item 
i) $X$ has an essential singularity
at $z_0=e$. 

\noindent 
ii) $\Re{X}$ has any other topology different from 
(1)--(3).

\end{enumerate}
\end{proposition}

\begin{figure}[htbp]
\begin{center}
\includegraphics[width=0.75\textwidth]{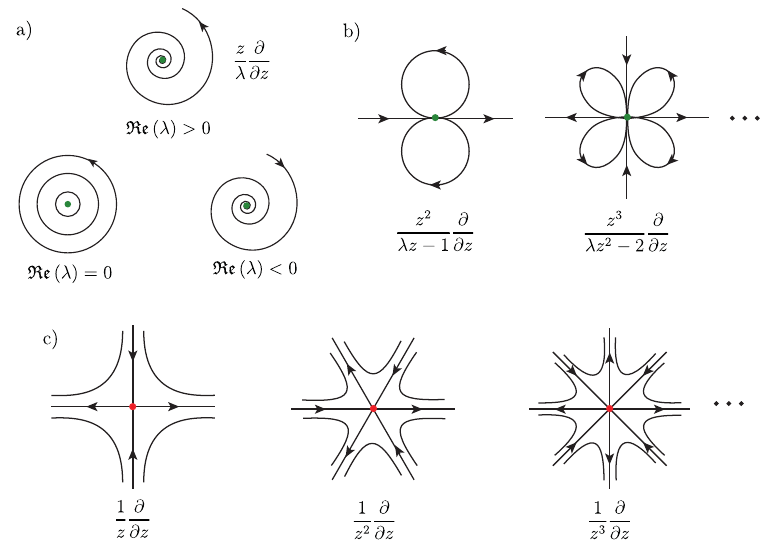}
\caption{
Local analytic normal forms: 
a) simple zeros, 
b) multiple zeros, 
and c) poles of $X$.
By simplicity, zeros and poles are at the origin.
}
\label{2-forma-normal}
\end{center}
\end{figure}

\begin{proof}
In assertions (1)--(4), 
$X$ is assumed to be holomorphic and nonzero
in a punctured disk $D(z_0, \rho) \backslash \{z_0\}$.
In (2), 
a parabolic sector appears
if and only if $Res(\omega_X, z_0) \in \CC \backslash \RR$,
for further details see \cite{AlvarezMucino} \S5.
\end{proof}

\section{Singularities of $\Psi^{-1}_X$: ideal points of $M \backslash \mathcal{S}$}\label{asymptoticvalues}

The work of F. Iversen \cite{Iversen} originates 
the study of transcendental singularities
of meromorphic functions, 
and modern expositions  can be found in
W. Bergweiler {\it et al.} \cite{BergweilerEremenko} 
and A. Eremenko \cite{EremenkoReview}. 
In this theory,  
the inverse function $\Psi_X^{-1}$ and 
the Riemann surface $\R_X$ 
play an essential role.

\begin{remark}
We consider three families functions on $M$:

\noindent $\bigcdot$
single--valued 
additively automorphic
singular complex analytic functions,

\noindent $\bigcdot$
multivalued additively automorphic 
singular complex analytic functions and

\noindent $\bigcdot$ 
nonadditively automorphic multivalued singular complex analytic functions.
 
\noindent  
The first two families are studied in \S 
\ref{caso-univaluado-para-Psi}--\ref{caso-widetilde-Psi}, 
below.
The third family does not appear when we deal
with vector fields, see comment on 
\S \ref{Future-work-subseccion}.
\end{remark}

\subsection{Single--valued additively automorphic $\Psi_X$}
\label{caso-univaluado-para-Psi}

In this section, we shall consider 
a singular complex analytic 1--form of time 
$\omega_X$ with
$\mathcal{S}_{R}=\varnothing$, hence
the domain is
$M \backslash \mathbb{E}$.
In other words,

\begin{equation}
\label{Psi-sec-3-1}
\Psi_X (z)=
\int^z_{z_{\tt o}} \omega_X:
M \backslash \mathbb{E}
\longrightarrow \CW_t, 
\ \ \ 
\mathcal{S}_{R}=\varnothing,
\end{equation}

\noindent 
is a
single--valued additively automorphic singular complex 
analytic function,
where the initial point of integration is
a nonsingular point
$z_{\tt o}  \in M \backslash \mathcal{S}$.
The integral function in
Equation \eqref{Psi-sec-3-1} is a particular case of
\eqref{parametro-local}.

\begin{definition}[\cite{Iversen}, 
\cite{BergweilerEremenko},
\cite{EremenkoReview}]
\label{eremenko1}
Take $a\in\CW_{t}$ and denote by 
$D(a,\rho) \subset \CW_t$
the disk of radius $\rho > 0$ (in the spherical metric) 
centered at $a$. 
For every 
$\rho > 0$, choose a component 
$U_a(\rho)\subset M$ of  
$\Psi_X^{-1}(D(a,\rho))$ in such a way that 
$\rho_{1} < \rho_{2}$ implies $U_{a}(\rho_1) \subset U_{a}(\rho_{2})$. 
Note that the function $U_{a} : \rho \to U_{a}(\rho)$ 
is completely determined by its germ at 0. 

\noindent
The two possibilities below can occur for the germ of $U_{a}$.
\begin{enumerate}[label=\arabic*),leftmargin=*] 
\item 
$\cap_{\rho>0} U_{a}(\rho)=
\{z_k \},
\, z_k\in M$. 
In this case, $a=\Psi_X(z_k )$. 

\noindent
Moreover, if $a\in\CC_{t}$ and $\Psi_X'( z_k )\neq0$, 
or 
$a = \infty$ and $ z_k $ is a simple pole of $\Psi_X$, 
then $z_k$ is called an \emph{ordinary point}. 

\noindent
On the other hand, 
if $a\in\CC_{t}$ and $\Psi_X'(z_k) = 0$, 
or 
if $a = \infty$ and $z_k$ is a multiple pole of $\Psi_X$, 
then $z_k$ is called a \emph{critical point} and $a$ 
is called a \emph{critical value} 
of $\Psi_X$. 
We also say that the critical point $z_k$ \emph{lies over $a$}.
In this case, 
$U_{a} : \rho \to U_{a}(\rho)$ defines an \emph{algebraic singularity of $\Psi_{X}^{-1}$}.

\item $\cap_{\rho>0}U_{a}(\rho) = \varnothing$. 
We then say that our choice $\rho \to
U_{a}(\rho)$ defines a \emph{transcendental 
singularity of $\Psi^{-1}_X$}
and that the transcendental singularity 
$U_{a}$ \emph{lies over $a$}. 
\end{enumerate}

\noindent 
In both cases, 
the open set $U_{a}(\rho) \subset M$ is called a 
\emph{neighbourhood of the  singularity $U_{a}$}. 
Therefore, 
when $\zeta_{m} \in M$, 
we say that $\zeta_{m} \to U_{a}$ 
if for every $\rho>0$ 
there exists $m_{0} \in \NN$ such that 
$\zeta_{m} \in U_{a}(\rho)$, for $m \geq m_{0}$.
\end{definition}

\begin{remark}
\label{remark-a-eremenko1}
The germ $U_a$ of Definition \ref{eremenko1} case (2) can be understood as follows.

\noindent
1.
A transcendental singularity of 
$\Psi_X^{-1}$, namely $U_{a}$,
is equivalent to the addition of an 
\emph{ideal point $U_a$} to 
$M \backslash \mathbb{E}$.

\noindent
2.
The addition of the ideal points $\{ U_a \}$, 
together with their corresponding neighbourhoods 
$\{ U_a(\rho) \} \subset M$,
provide a Hausdorff completion/compactification of 
$M \backslash \mathbb{E}$,
see \cite{Ahlfors-Sario} Ch.\,I\,\S\,6 for the general
construction.

\noindent 
In our framework, 
the families of functions 
$\Psi_X(z)= 
\int^z \big( P(\zeta)/Q(\zeta)\big) \e^{-E(\zeta)} d\zeta$, 
in Theorem \ref{familias-s-r-d}, provide 
prototypes of this kind of compactification,
even in the multivalued case. 

\noindent
3. 
In what follows, 
we shall interchangeably
refer to a \emph{transcendental singularity 
$U_a$ of $\Psi_X^{-1}$} or
an \emph{ideal point $U_a$ of $M \backslash \mathbb{E}$}.
\end{remark}

Let $z_{\tt s} \in \mathcal{S} \subset M$,
the expression $z$ tends to $ z_{\tt s} \in M$ 
makes sense.
Recalling \cite{EremenkoReview} p.~3,
the following concept is natural.

\begin{definition}\label{defasymptoticvaluepath}
1) 
Let $U_a$ be a transcendental singularity of $\Psi_X^{-1}$.
An \emph{asymptotic value $a\in\CW_t$ of $\Psi_X$}
means that there exists 
a $C^1$ 
\emph{asymptotic path}
$\alpha_a ( {\tt t} ):
[0, \infty) \longrightarrow M$,
$\alpha_a(0)=z_{\tt o}  \in M \backslash \mathcal{S}$,
tending to 
$ z_{\tt s} \in M$ with well defined slope, such that  
\begin{equation}\label{valasintPsi}
a=
\lim_{
{\tt t}  \to \infty
}
\Psi_X ( \alpha_a ( {\tt t} ) )=
\lim_{ {\tt t} \to \infty } 
\int_{\alpha_a ( {\tt t} )}
\omega_X 
\in\CW_t .
\end{equation}
\noindent

\noindent 
We shall not distinguish between 
individual members $\alpha_a$ of
the class of asymptotic
paths $[\alpha_a]$ giving rise to the 
same transcendental singularity $U_a$ over $a$  
of $\Psi^{-1}_X$. 

\noindent 2)
A pair $(\alpha_a, a)$ is a 
\emph{branch point of $\R_X$}.
\end{definition}

\begin{remark}
Because of Lemma \ref{Lemma-RX}.3, 
we will assume that
the asymptotic path in Definition
\ref{defasymptoticvaluepath} ends at the 
singular point $ z_{\tt s}$. 

\noindent 
1. There is a bijective correspondence between 
the following:

\noindent
i) 
classes $[\alpha_{a}( {\tt t} )]$ 
of asymptotic\footnote{
A slight abuse of notation is made here, when $U_a$ is algebraic
(\emph{i.e.} $z_{\tt s} \in \mathcal{P} \cup \mathcal{Z}_0$), 
the path $\alpha_a({\tt t})\to z_{\tt s}$ is not an asymptotic path, 
it is just a path arriving to the critical point $z_{\tt s}$.
}  paths $\alpha( {\tt t} )$,

\noindent
ii) asymptotic values $a\in\CW_t$ counted with multiplicity, 

\noindent
iii) transcendental singularities $U_a$ of $\Psi^{-1}_X$ and

\noindent
iv) branch points\footnote{
In the particular case of algebraic branch points arising from poles $p_k$
of $X$, we shall use the notation $(p_k,\widetilde{p}_k)$, 
instead of the more cumbersome $(\alpha_{\widetilde{p}_k},\widetilde{p}_k)$, since
in this case $\lim\limits_{{\tt t}\to\infty} \alpha_{\widetilde{p}_k}({\tt t}) = p_k$ and
$\lim\limits_{{\tt t}\to\infty} \Psi_X\big(\alpha_{\widetilde{p}_k}({\tt t})\big)=
\widetilde{p}_k$; see Example \ref{Ejemplo-ExpSin-z}. 
Similarly, we shall use $(q_k,\infty)$ for the branch points associated 
to the zeros $q_k$ of $X$.
} $(\alpha_a, a)$ of $\R_X$.

\noindent 2.
Certainly, the notation $U_a$ can be confusing for
singular values $a$ with multiplicity two or more;
in those cases we add a subscript
$ a_\sigma $, in order to distinguish them.
\end{remark}
  
\begin{definition}
\label{singular-values}
The \emph{singular values 
of $\Psi_X$} are the
critical values and asymptotic values, 
both counted with multiplicity.  
\end{definition}

If $a\in\CW_t$ 
is an asymptotic value of $\Psi_X$, then 
there is at least one transcendental singularity $U_a$ 
of $\Psi_X^{-1}$ over $a$.
Certainly, there can be 
finite or even infinite 
different transcendental singularities 
as well as critical and ordinary points over the same 
singular value $a$.

\begin{remark}[On the finitude of the set of asymptotic values]

\noindent 
1. The Denjoy--Carleman--Ahlfors theorem provides a sharp estimate 
for the number of asymptotic values 
when $M=\CW_z$.
If $\Psi_X$ is an entire function with
$d$ finite asymptotic values, then
the order of growth

\centerline{
$
\limsup_{\, r \to \infty} 
\dfrac{\log M(r)}{\log r} = d,
$
} 

\noindent 
where as usual $M(r) = \max_{\vert z\vert = r } \vert \Psi_X(z)\vert$.
Compare with \cite{Segal} \S5.2. 
In fact, 
the order of growth is a valuable local 
analytic invariant, 
see \cite{Nevanlinna1} for single--valued functions. 
In \cite{AlvarezMucino}, we consider
this invariant for vector fields, study 
some families and relate it to the number of
asymptotic values.  

\noindent 
2. On the other hand, 
there exist single--valued transcendental meromorphic 
functions on $\CC_z$ with an infinite set of 
asymptotic values. 
See W. Gross \cite{Gross} and 
A.~Eremenko \cite{EremenkoReview} \S4. 

\end{remark}

\begin{definition}
\label{direct-indirecta-logaritmica}
A transcendental singularity $U_a$ of $\Psi_X^{-1}$ over $a$ is 

\noindent
1)  
\emph{direct} 
if there exists $\rho > 0$ such that 
$\Psi_X(z)\neq a$ 
for $z \in U_a(\rho)$, this is also true for all smaller values of $\rho$,

\noindent
2) 
\emph{indirect} if it is not direct, 
{\it i.e.} for every $\rho > 0$, 
the function $\Psi_X$ takes the value $a$ in $U_a(\rho)$, 
in which case the function $\Psi_X$ takes the value $a$ infinitely often in $U_a(\rho)$,

\noindent 
3)  
\emph{logarithmic singularity 
over $a$} 
if 

\centerline{
$\Psi_X : U_a (\rho) \subset M \longrightarrow
D(a, \rho) \backslash\{a\} \subset \CW_t$} 

\noindent 
is a universal covering for small enough
$\rho $. 
\end{definition}

Naturally, logarithmic singularities are direct. 
We shall use ``nonlogarithmic'' without the ``direct'' adjective when referring to
direct nonlogarithmic as well as indirect singularities.

\begin{example}\label{singularidad-logaritmica}
The simplest case of direct singularities arises from 

\centerline{
$\Psi_X(z) = 
{\displaystyle \int^z } \e^{-\zeta} d\zeta: \CC_z
\longrightarrow 
\CC_t \backslash \{ 0 \}$.}

\noindent
There are logarithmic singularities over 
the asymptotic values $0, \infty\in\CW_t$ 
respectively.
For small enough $\rho>0$, 
the neighbourhoods $U_0(\rho)$ and $U_\infty(\rho)$ are 
\emph{exponential tracts}.
We illustrate this in 
Figure \ref{flores-elipticas-hiperbolicas}.a.
\end{example}

\subsection{Multivalued additively automorphic $\Psi_X$; the fundamental domain $\Lambda$}
\label{caso-multivaluado-para-Psi}
Consider a
\emph{multivalued additively automorphic}
singular complex 
analytic function  
\begin{equation}
\label{Psi-sec-3-2}
\Psi_X (z)=
\int^z_{z_{\tt o}} \omega_X :
M \backslash 
(\mathbb{E} \cup \mathcal{Z}_R)
\longrightarrow \CW_t, 
\quad
\mathcal{S}_{R}\neq\varnothing,
\end{equation}
\noindent 
where 
the initial point of integration is
a nonsingular point
$z_{\tt o}  \in M \backslash \mathcal{S}$.
The integral function in
Equation \eqref{Psi-sec-3-2} is a particular case of
\eqref{parametro-local}. 

One of the \emph{fundamental hurdles in studying 
multivalued additively automorphic functions
\eqref{Psi-sec-3-2}
\`a la Iversen}, 
Definition \ref{eremenko1},
is that the neighbourhoods $U_a(\rho) = \Psi_{X}^{-1} (D(a,\rho))$ 
are not useful for distinguishing the ideal points $U_a$.
For the sake of clarity, we describe the simplest object
where this occurs.

\begin{example}[Singular points with nonzero residue]
\label{example-nonzero-residue}
Let us consider the multivalued
additively automorphic
singular complex analytic function

\centerline{
$\Psi_{X}(z)=
\lambda
\log(z) + C =
\lambda
{\displaystyle \int^{z} } \dfrac{d\zeta}{\zeta}
:
\CW_z\backslash\{0,\infty\}\longrightarrow\CW_t,
\ \ \ \lambda\in\CC^*$.
}

\noindent 
The associated $\omega_X$ 
on $\CW_z$
has nonzero residues at 

\centerline{
$\mathcal{S}=\mathcal{Z}_{R}=\{0,  \, \infty\}$.}

\noindent
On the other hand, $\Psi_{X}^{-1}(t)=\exp(t/\lambda )$
is an entire function that
has an isolated essential singularity at $\infty\in\CW_{t}$.
As a consequence of Picard's theorem
applied to $\Psi_X^{-1}$, for any $\rho>0$ the neighbourhood 
$U_{\infty}(\rho)\doteq \Psi_{X}^{-1}\big(D(\infty,\rho)\big)$ 
is $\CW_z\backslash\{0,\infty\}$.

\noindent
In its original setting,
Iversen's theory of transcendental singularities 
does not make sense at 
$0,\infty\in\CW_z$, which are 
poles of  the associated $\omega_X$.

\noindent 
In other words, 
\emph{for every $\rho>0$ the neighbourhoods $U_\infty (\rho)$
are all the same,
and consequently the ideal points $U_\infty$, are not well defined.}
As will be seen in 
\S\ref{caso-widetilde-Psi}, this can be 
explained 
by considering the universal cover  $\CC$ of
$\CW_z \backslash  \mathcal{Z}_{R}$.

\noindent
The analogous behaviour of $\Psi_X^{-1}$
appears for many other families of 
functions, \emph{e.g.}  

\centerline{
$\Psi_X(z)= 
{\displaystyle  \int^z} \frac{P(\zeta)}{Q(\zeta)}\e^{-E(\zeta)} \, d\zeta$,}

\noindent 
assuming that their 1--forms of time have 
nonzero residues, see Theorem \ref{familias-s-r-d}. 
\end{example}

\subsubsection{Construction of a fundamental 
domain for $\Psi_X$}
\label{construccion-region-fundamental}
In order to extend 
Iversen's theory of singularities of the inverse 
function to multivalued 
additively automorphic singular complex analytic functions 
$\Psi_X$,
note that in Diagram \ref{diagramaRX},
the function
$\Psi_{X} = \pi_{1}^{-1}\circ\pi_{2}$
factors through $\R_{X}$.
Very roughly speaking,
for $\Psi_X$ as in Equation \eqref{Psi-sec-3-2},
we search for a 
maximal univalence domain $\Lambda$ for $\Psi_X$ 
(\emph{i.e.} where $\Psi_X$ is defined and single--valued).
Recalling Remark \ref{cuando-Psi-es-multivaluada}, 
we proceed as follows.

\smallskip

\noindent 
Let $\Psi_X$ be as in Equation \eqref{Psi-sec-3-2}.

\noindent
{\bf 1.} 
Assume first that $M=\CW_z$ or the disk $\Delta_z$.
(If $M=\CC_z$, 
by adding the conformal puncture $\infty$,
we obtain $\CW_z$.) 
Let $\mathcal{S}_R$ be 
the set of nonzero residue 
singular points of $\omega_X$;
assume by hypothesis that its
cardinality is $2 \leq \kappa \leq \infty$ 
(possibly infinite and numerable),
{\it i.e.}

\centerline{
$ \mathcal{S}_R
=\{ z_1, z_2,\ldots,z_\kappa \}$.}

\noindent
{\bf 2.}
Assume that we have a
collection of paths 
$\Gamma = \{ \gamma_k \}_{k=1}^{\kappa-1}$,
where $\gamma_k$ 
is the segment of $\Gamma$ 
between $z_k$ and $z_{k+1}$
satisfying the following:

\noindent i)
Each $\gamma_k\subset M$ is a 
continuous simple path with extreme points in 
$\mathcal{S}_R$
and avoids other singular 
points in $\mathcal{S}$. 

\noindent ii)
For $k \neq \ell$, 
the intersection $\gamma_k \cap \gamma_\ell$
is either one point $z_{\ell}$ 
(when $\ell = k+1$ or $k-1$) in 
$\mathcal{S}_R$
or is empty otherwise.

\noindent iii) 
The set

\centerline{$M \backslash 
\overline{\Gamma }$}

\noindent 
is an open connected Riemann surface, where 
$\overline{(\ \ )}$ means the closure in $M$.
Note that $\omega_X$ is still a singular complex analytic 1--form on 
$M \backslash\overline{\Gamma}$
with singular set $\mathcal{S}\backslash\overline{\mathcal{S}_R}$.

\noindent
{\bf 3.} 
As usual, if we
cut $M$ along $\gamma_k$, 
we obtain two boundary paths, say 
$\gamma_{k+ } $
and 
$\gamma_{k-} $, 
which are considered without their extreme points
$z_k$ and $z_{k+1}$.
We define

\centerline{
$\Lambda_0 \doteq 
\big( M\backslash \overline{\Gamma} \big) \,
\bigcup_{k=1}^\kappa 
\gamma_{k+}
$.
}

\noindent 
Simply stated, 
we add to the open surface 
$M  \backslash \overline{\Gamma}$  
only one boundary component $\gamma_{k+}$
for each path $\gamma_k$. 

\noindent 
{\bf 4.} 
In the case $M\neq\CW_z, \, \Delta_z$,
then $M$ is not simply connected and
we require an additional construction.
Let
$\{\widetilde{\gamma}_\ell\}_{\ell=1}^L
\subset M\backslash\mathcal{S}$ 
be representatives of the generators of 
the fundamental group $\pi_1( \Lambda_0 )$.
Note that $\widetilde{\gamma}_\ell$ are 
simple closed paths in 
$\Lambda_{0}$.
Hence, 
cutting $\Lambda_{0}$ along the paths 
$\{\widetilde{\gamma}_\ell\}_{\ell=1}^L$ and 
once again
adding only
one of their boundary components 
$\{\widetilde{\gamma}_{\ell+}\}_{\ell=1}^L$,
we obtain 

\centerline{
$\Lambda =
\big( \Lambda_{0} \backslash 
(\cup _{\ell=1}^L \widetilde{\gamma}_\ell ) \big)
\
\bigcup _{\ell=1}^L \widetilde{\gamma}_{\ell+}
$,
}

\noindent
\emph{a fundamental domain for $\Psi_X$}.

\begin{remark}\label{No-necesario-evitar-polos}
1.
Considering $\omega_X$, note that   
$\Lambda \cap \mathcal{S}$ contains its 

\noindent
$\bigcdot$
zeros $\mathcal{P}$ and

\noindent
$\bigcdot$
poles with residue zero $\mathcal{Z}_0$.

\noindent
Furthermore, $\overline{\mathcal{S}_R}$ is in the boundary of $\Lambda$.

\noindent 2.
By construction, $\Lambda$ is simply
connected, has nonempty boundary and
$\int_\beta \omega_X=0$ for any closed path 
$\beta$ in the locus where $\omega_X$ is holomorphic.
The restriction of $\Psi_X$ in 
Equations \eqref{parametro-local} and 
\eqref{Psi-sec-3-2}, 
\begin{equation}
\label{la-Psi-en-Lambda}
\Psi_{X, \, \Lambda}(z) =
\int^{z}_{z_{\tt o}} \omega_X :  
\Lambda
\backslash\mathbb{E}
\longrightarrow \CW_t
\end{equation}

\noindent 
is a single--valued 
singular complex analytic
function with singular set 
$\mathcal{S}\backslash \overline{\mathcal{S}_R}$
(note that $\mathbb{E} \cup \mathcal{Z}_R=\mathbb{E} \cup \mathcal{S}_R$).

\noindent 3.
In the construction of $\Gamma$ we have avoided 
the set of singular points, \emph{i.e.} we have
asked that $\Gamma\cap\mathcal{S}=\varnothing$.
This has been done for simplicity, 
however, 
note that $\Gamma\cap\mathcal{P}\neq\varnothing$
can be allowed (this is sometimes useful), 
since $\omega_X$ has zeros at $\mathcal{P}$ and
$\Psi_X$ is holomorphic on $\mathcal{P}$.
\end{remark}

\begin{definition}\label{region-fundamental}
A \emph{fundamental region} for a 
multivalued additively automorphic 
function $\Psi_X$ is 

\centerline{
$\Omega = 
\{ (z, \Psi_X(z)) \ \vert \   z \in 
\Lambda \backslash \mathbb{E} \} \subset 
M \times \CW_t.$
}
\end{definition}

\begin{remark} 1.
Obviously, a fundamental region
$\Omega$ depends on the choice of
$z_{\tt o}$, $\{ \gamma_k\}$  and 
$\{ \widetilde{\gamma}_\ell\}$.

\noindent
2.
The following diagram commutes
\begin{center}
\begin{picture}(180,65)(0,20)

\put(-131,40){\vbox{
\begin{equation}\label{diagrama-dominio-fundamental}\end{equation}
}}

\put(0,75){$\big(\Lambda \backslash \mathbb{E}, X \big) $}

\put(115,75){$\big( \Omega,\pi^*_2(\del{}{t})\big)$}

\put(108,78){\vector(-1,0){60}}
\put(75,85){$\pi_{1}|_{\Omega}$}

\put(133,65){\vector(0,-1){30}}
\put(138,47){$ \pi_{2} $ }

\put(38,65){\vector(2,-1){73}}
\put(50,39){$ \Psi_{X,\, \Lambda} $}

\put(115,20){$\big(\CW_t,\del{}{t}\big) $,}

\end{picture}
\end{center}
\noindent 
where 
$\pi_{1}|_{\Omega}$ and $\pi_2$ are local isometries.
The fundamental domain $\Lambda \backslash \mathbb{E}$
and the fundamental region $\Omega$
are biholomorphic under $\pi_{1}|_{\Omega}$.
Note that $ \Psi_{X,\, \Lambda}^{-1} =
\pi_{1}|_{\Omega}\circ\pi_{2}^{-1}$.
\end{remark}

Since $\Psi_X$ is single--valued on 
$\Lambda$,
we proceed to slightly modify all the concepts
in \S \ref{caso-univaluado-para-Psi}, 
by using $\Lambda$ instead of $M$. 
Let 
$\alpha_a({\tt t}): [0, \infty) \longrightarrow \Lambda$
be an asymptotic path, analogously
as in Equation \eqref{valasintPsi}
in Definition \ref{asymptoticvalues},
so $(\alpha_a,a)$ is the branch point in $\R_X$ 
corresponding to the path $\alpha_a({\tt t})\to z_{\tt s}$, and

\centerline{$
a =
\lim\limits_{{\tt t}\to\infty} 
\Psi_{X,\, \Lambda} \big( \alpha( {\tt t} ) \big) 
=
\lim_{ {\tt t} \to \infty } 
{\displaystyle \int_{\alpha_a ( {\tt t} )} }
\omega_X 
\in \CW_t .
$}

\begin{definition}[Extension to the additively automorphic case]
\label{Nueva-def-vecindades}
Let $\Psi_X$ be as in \eqref{Psi-sec-3-2} and 
the function  $\Psi_{X,\Lambda}$,  
which depends on the choice of $\Lambda$
be as in \eqref{la-Psi-en-Lambda}. 
Take $a \in \CW_t$
and denote by $D(a,\rho) \subset \CW_t$
the disk of radius $\rho > 0$ (in the spherical metric) 
centered at $a$. 
For every $\rho>0$, first choose 
a connected component 

\centerline{
$V \big( (\alpha_a,a),\rho \big) \subset \R_X$ 
\ of \
$ \pi_2 ^{-1}\big(D(a,\rho)\big)$,}

\noindent 
and,
using Diagram \ref{diagrama-dominio-fundamental},
$\pi_{1}|_{\Omega}$ is the restriction of $\pi_1$ to $\Omega$,
then let

\centerline{
$U_{a}(\rho) \doteq \pi_{1}|_{\Omega} \Big( V \big( (\alpha_a,a),\rho \big) \Big)$,
}

\noindent 
in such a way that $\rho_1<\rho_2$ implies $U_{a}(\rho_1) \subset U_{a}(\rho_2)$. 
The \emph{neighbourhoods $U_a(\rho)$}
determine
\emph{ideal points or singularities $U_a$ of 
$\Psi_{X,\Lambda}^{-1}$}. 

\noindent
With the above considerations, all 
the definitions and results presented in 
\S\ref{caso-univaluado-para-Psi} apply for 
$\Psi_{X,\, \Lambda}$.
\end{definition}

\begin{remark}[Some consequences of 
the multivalued nature of $\Psi_X$]
\label{aclaraciones-de-vecindad}
1. In our construction of the neighbourhood $U_a(\rho)$, 
there is a choice of one 
connected component of $\pi_2 ^{-1}\big(D(a,\rho)\big)$.
However, 
due to the choice of $\Gamma$,
the projection 
$U_a(\rho)= \pi_{1}|_{\Omega}
\Big( V \big( (\alpha_a,a),\rho \big) \Big) $
can have an arbitrary number of connected components.
For instance, in Examples \ref{Ejemplo-multivaluado} and 
\ref{ejemplo-3-puntos} 
(Figures \ref{flujo} and \ref{campos-Delta}.a),
if the paths $\gamma_k \subset\Gamma$ are chosen to lie on the
real axis, then at least one of the neighbourhoods $U_a(\rho)$ of 
the transcendental singularities corresponding to the essential singularity
would have two connected components.

\noindent
2. 
Note that, when $\mathcal{Z}_R\neq\varnothing$ a new type of
transcendental singularity of $\Psi_{X,\,\Lambda}^{-1}$ 
appears: ideal points $U_\infty$ arising from the
non zero residue poles of $\omega_X$, see example
below.
\end{remark}

The following definitions are natural.

\begin{definition}
\label{definicion-singularidades-en-Lambda}
Assume that
there exists an asymptotic path
$\alpha({\tt t})$
in $\Lambda$
tending to a singularity 
$z_{\tt s} \in \mathcal{S}$ of $\omega_X$ 
with asymptotic value $a$. 
The respective $U_{a}$ is:

\begin{enumerate}[label=\arabic*),leftmargin=*]

\item
An \emph{essential transcendental singularity of 
$\Psi_{X,\, \Lambda}^{-1}$}, 
when $z_{\tt s} \in \mathbb{E}$.

\begin{enumerate}[label=\roman*),leftmargin=*]

\item
A \emph{zero residue essential transcendental singularity of
$\Psi_{X,\, \Lambda}^{-1}$}, 
when $z_{\tt s} \in \mathbb{E}_{0}$.
\end{enumerate}

\item 
A \emph{nonzero residue
transcendental singularity of
$\Psi_{X,\, \Lambda}^{-1}$},
when $z_{\tt s} \in \mathcal{S}_{R}=
\mathcal{Z}_{R} \cup \mathbb{E}_{R} $.

\begin{enumerate}[label=\roman*),leftmargin=*]

\item
A \emph{$\star$--transcendental singularity of
$\Psi_{X,\, \Lambda}^{-1}$},
when $z_{\tt s} \in \mathcal{Z}_{R}$.

\item
A \emph{nonzero residue essential transcendental 
singularity of
$\Psi_{X,\, \Lambda}^{-1}$},
when $z_{\tt s} \in \mathbb{E}_{R}$.
\end{enumerate}
\end{enumerate}
\end{definition}

\begin{figure}[htbp]
\begin{center}
\includegraphics[width=0.55\textwidth]{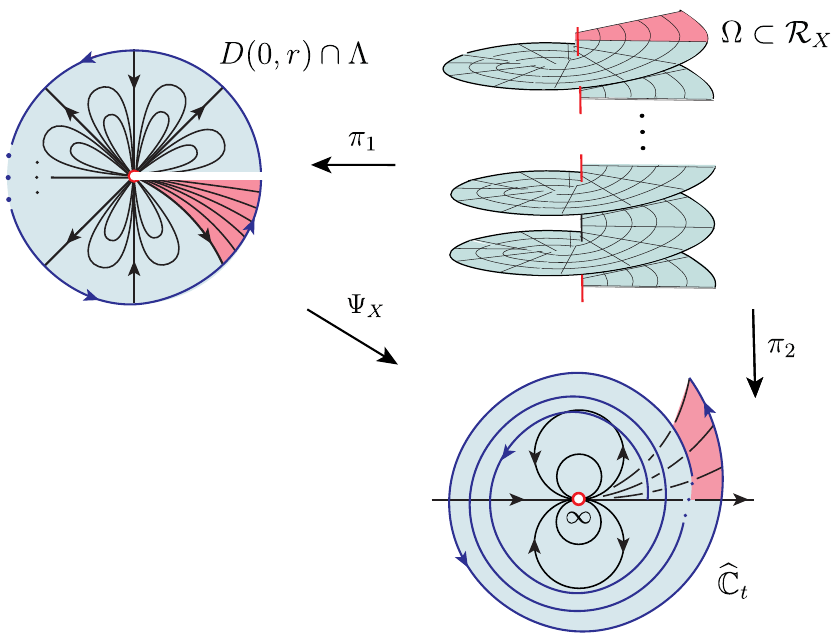}
\caption{
A pole of $\omega_X =\big((\lambda/z) - (s-1)/z^s) dz$
of order $s$ at least 2,  
and nonzero residue $\lambda$. 
The associated vector field  $X$ has 
$2s -2 \geq 2$ 
elliptic sectors and one parabolic sector. 
The behaviour of the function $\Psi_X$, Equation 
\eqref{nonzero-star--transcendental-singularity}, 
is called an 
\emph{$s$--fold unbranched holomorphic 
log--covering}.
The parabolic sector (and its corresponding images 
in $\Omega$ and $\CW_t$) has been coloured pink
for clarity.
This illustrates a  
$\star$--transcendental singularity of $\Psi_{X,\, \Lambda}^{-1}$.
}
\label{cubierta-incompleta}
\end{center}
\end{figure}

\begin{example}[Example \ref{example-nonzero-residue} revisited]
Let 
\begin{equation}
\label{nonzero-star--transcendental-singularity}
\Psi_X (z) =
\lambda \log(z) + 
\frac{1}{z^{s-1}},
\ \ \ \hbox{ on } M=\CW_z, 
\ s \geq 1, 
\ \lambda \in \CC^*, 
\end{equation}
\noindent 
be a multivalued additively automorphic function.
Let us consider 
$\Lambda =\CW_z \backslash (\RR^+ \cup \{ 0, \, \infty \})$. 
The fundamental region is

\centerline{
$\Omega = 
\left\{ \big(z, \int _{1}^{z} \omega_X  \big)
=
\big( z,\,  \lambda \log(z) + 
\frac{1}{z^{s-1}}  
\big)
\right\} \subset
\CW_z \times \CW_t .$}

\noindent
Note that $\Psi_X$
is the integral of the normal form of $\omega_X$ 
having a pole of multiplicity $s\geq 1$ at $z_{\tt s} =0$,
with nonzero residue.
Thus, for all paths $\alpha({\tt t}) \to 0$ in $\Lambda$, 
the asymptotic value  
of $\Psi_{X,\, \Lambda}$ is $\infty$ 
and  the corresponding 
transcendental singularity
$U_\infty$ 
of $\Psi_{X,\, \Lambda}^{-1}$
lies over $\infty$.
The neighbourhoods $U_\infty(\rho) \subset \Lambda$
contain  $D(0, r(\rho, \lambda))\cap \Lambda$, 
for suitable radius $r(\rho, \lambda)$ which tend to 0 when 
$\rho \to 0$; hence
Definition \ref{Nueva-def-vecindades}
is satisfied. 
According to Definitions \ref{direct-indirecta-logaritmica} and
\ref{definicion-singularidades-en-Lambda},
it is a direct singularity, 
which is not logarithmic;
thus $U_\infty$ is a 
\emph{$\star$--transcendental singularity of $\Psi_{X,\, \Lambda}^{-1}$}.
Figure \ref{cubierta-incompleta} illustrates
the generic behaviour of $\Psi_X$, where
the parabolic sector depending on
$\lambda$ appears; for
an accurate explanation see \cite{AlvarezMucino} \S5.

\noindent 
Note that for the singular point  $z_{\tt s} =\infty$, the study is completely analogous.
\end{example}

We obtain the following normal forms summary for poles and 
zeros of $\omega_X$:

\begin{center}
\begin{tabular}{|c|c|c|c|c|}
\hline 
&&&&\vspace{-.4cm}
\\
Singularity 
of $\Psi_X^{-1}$
& $\Psi_X(z)$ & $\omega_X(z)$ & $X(z)$ & parameters
\\
\hline
& & & & \\[-8pt]
algebraic 
for $\lambda=0 $ 
& & & & $s\geq1$,\\[-10pt]
&
$\lambda \log(z)$ 
$+\ \ \frac{1}{z^{s-1}}$ 
& 
$\left( \dfrac{\lambda}{z} - \dfrac{s-1}{z^{s}} \right) dz$
&
$\dfrac{z^{s}}{\lambda z^{s-1} - (s-1)} \ddel{}{z}$ & 
  \\
\\[-19pt]
\tabdashline & & & & \\[-9pt] 
$\star$--transcendental  
for $\lambda \neq 0$
& & & & residue $\lambda\in\CC$ \\
\hline 
&&&& \vspace{-.3cm}
\\ 
algebraic 
& $\dfrac{z^{k+1}}{k+1}$ & $z^k \, dz$ & $\dfrac{1}{z^k} \ddel{}{z}$ &
$k\geq1$ \\[10pt]
\hline
\end{tabular}
\end{center}

\begin{remark}
\label{remark-def-sing-inv}
1.\
The transcendental  singularities of $\Psi_X^{-1}$ that appear in the classical theory of 
single--valued functions $\Psi_X$,
as in Definition \ref{eremenko1}, 
are all zero residue
essential transcendental singularities, as in 
Definition \ref{definicion-singularidades-en-Lambda}.1.i.

\noindent
2.\
In Definition \ref{definicion-singularidades-en-Lambda}.2.i, 
since $z_{\tt s} \in \mathcal{Z}_{R}$,
the asymptotic value is necessarily
$a = \infty\in\CW_t$.

\noindent
3.\
Note that nonzero residue essential transcendental singularities
Definition \ref{definicion-singularidades-en-Lambda}.2.ii, 
are also essential transcendental singularities
Definition \ref{definicion-singularidades-en-Lambda}.1.
\end{remark}

An accurate description of the singular values 
for $\Psi_{X, \, \Lambda}$ 
is required.
In Definition \ref{Nueva-def-vecindades} 
there is a choice of 
fundamental domain $\Lambda$,
equivalently of a fundamental region 
$\Omega\subset\R_X$.
This actually makes a difference
on what is considered a singular value
of $\Psi_{X,\Lambda}$.

Let $\Psi_X$ be a multivalued 
additively automorphic singular complex 
analytic function, 
Equation \eqref{Psi-sec-3-2}, 
and
assume that $a\in\CW_t$ is a singular value 
of $\Psi_{X,\, \Lambda}$
with asymptotic path 
$\alpha_a({\tt t})\subset\Lambda$.

\noindent 
Now, consider a path or class  
$\varrho \in \pi_1 
(M \backslash \overline{\mathcal{S}_R} ) $,
that starts at the 
nonsingular point
$z_{\tt o} \in M\backslash\mathcal{S}$, defining 

\centerline{
$\Xi_\varrho \doteq 
{\displaystyle \int_\varrho} \omega_X $.}

\noindent
Then

\centerline{
$\lim\limits_{{\tt t}\to\infty} 
\left( 
\int_{\varrho} \omega_X 
+
\int_{\alpha_a(\tt t)} \omega_X 
\right)= 
\Xi_\varrho +a$.
}

\noindent
In other words, 
the linear combinations 
$\Xi_\varrho  \in \CC^*$
of
residues and periods of $\omega_X$,
determine an infinite collection
$\{a + \Xi_\varrho \} \subset \CC$ consisting of: 

\noindent
i)
a singular value $a\in\CW_t$ 
of $\Psi_{X,\, \Lambda}$
and 

\noindent
ii) an infinite number
of \emph{fake} singular values, one for each possible 
nonzero linear combination $\Xi_\varrho$.

\noindent
Of course, 
the only true singular value 
of $\Psi_{X,\, \Lambda}$
is $a$, since the paths $\varrho$ 
concatenated with $\alpha_a$ 
do not lie in $\Lambda$ unless $\varrho$ is 
homotopic to the identity.

\begin{proposition}[Configurations of singular values 
amongst fundamental regions]
\label{teo-conf-valores}
Let $\Psi_X: 
M \backslash 
(\mathbb{E} \cup \mathcal{Z}_{R})
\longrightarrow \CW_t$ 
be a multivalued additively automorphic singular 
complex analytic function as in \eqref{Psi-sec-3-2}. 

\noindent
1) 
Given any two fundamental regions $\Omega_1$
and $\Omega_2$, 
the singular values
$\{a_j \}$ and
$\{\widetilde{a}_j \}$ 
of $\Psi_{X,\, \Lambda_1}$
and 
$\Psi_{X,\, \Lambda_2}$ respectively,
satisfy 

\centerline{
$a_j = \widetilde{a}_j + 
\Xi _{\varrho} 
$,\quad  
for $j=1,\ldots,m$
or 
infinite and numerable  $\{ j\}$,
}

\noindent
where $\Xi_{\varrho}\in\CC$ 
is a fixed linear combination of the 
residues and periods of 
$\omega_X$, 
that depends  
on the choice of the two 
fundamental regions 
$\Omega_1$, $\Omega_2$ and
on the initial point of integration 
$z_{\tt o}$ for $\Psi_X$.

\noindent
2) 
The qualitative behaviour
of the ideal points $U_a$ associated to 
$\Psi_{X,\, \Lambda_1}$
and $U_{a+\Xi_{\varrho } }$ associated to 
$\Psi_{X,\, \Lambda_2}$
is independent of the 
choice of fundamental regions $\Omega_1$ or $\Omega_2$.
\end{proposition}

\begin{proof}
For (1), 
given two different fundamental regions, 
say $\Omega_1$ and $\Omega_2$,
the corresponding $\Lambda_1=\pi_1(\Omega_1)$ and $\Lambda_2=\pi_1(\Omega_2)$
are simply connected subsets of the universal cover.
There exists an element $\varrho$ 
of the fundamental group 
$\pi_1 (M \backslash 
\overline{\mathcal{S}_R} )$
such that $\Lambda_2=\varrho (\Lambda_1)$
as a cover transformation.
Therefore, given a singular value $a\in\CW_t$ of 
$\Psi_{X,\, \Lambda_1}$,
the value $a+\Xi_{\varrho} \in\CW_t$ is the corresponding 
singular value of 
$\Psi_{X,\, \Lambda_2}$.

For (2), note that 
$\Psi_{X,\, \Lambda_1}$ and $\Psi_{X,\, \Lambda_2}$
differ only by the value
$\Xi_{\varrho}\in\CC_t$. 
However, since $\Omega_1$
and $\Omega_2$ are 
copies of each other
(up to cutting and pasting and
using the flat metric $g_X$ on $\R_X$ 
arising from $\pi_2 ^* (\del{}{t})$). 
Then,
the branch points associated to $U_a$ in 
$(\Omega_1, \pi_2 ^* (\del{}{t}))$
and $U_{a+\Xi_{\varrho}}$ in 
$(\Omega_2, \pi_2 ^* (\del{}{t}))$ are 
related by the cover transformation $\varrho$.
Hence the ideal points arising from
either $\Lambda_1$ or $\Lambda_2$ are qualitatively
the same.
\end{proof}


\subsection{A model for $\R_X$; the universal cover of $M \backslash \overline{\mathcal{S}_R}$}
\label{caso-widetilde-Psi}
Once again, we consider a
multivalued additively automorphic singular complex 
analytic function as in \eqref{Psi-sec-3-2}, namely

\centerline{
$\Psi_X (z)=
{\displaystyle \int^z_{z_{\tt o}} } 
\omega_X:
M \backslash (\mathbb{E} \cup \mathcal{Z}_{R}) 
\longrightarrow \CW_t, 
\quad
\mathcal{S}_{R}\neq\varnothing,$
}

\noindent 
where 
the initial point of integration is
a nonsingular point
$z_{\tt o}  \in M \backslash \mathcal{S}$.
Let  

\centerline{
$\pi:
\mathfrak{M} \longrightarrow 
M \backslash 
\overline{\mathcal{S}_R} $
}

\noindent 
be the universal cover 
of 
$M \backslash \overline{\mathcal{S}_R}$.
The analytic extension 
of $\Psi_X$ to $\mathfrak{M}$,
namely

\centerline{
$\widetilde{\Psi_X}(z)=
{\displaystyle \int^z }  \pi^* \omega_X: 
\mathfrak{M} \backslash \pi^{-1} (\mathbb{E}_0)
\longrightarrow \CW,$
}

\noindent 
is a single--valued additively automorphic
singular complex analytic function,
thus \S\ref{caso-univaluado-para-Psi} applies.
As a mater of record,

\centerline{
$\widetilde{X}= \pi^* X $ 
\ on \ 
$\mathfrak{M}$
}

\noindent 
denotes the singular complex analytic vector field
associated to $\widetilde{\Psi_X}$.
Moreover, by Lemma \ref{Lemma-RX}.2, we have that

\centerline{
$\R_X = 
\cup_\varrho \Omega_\varrho
\cong \mathfrak{M}, 
\ \ \   
\varrho \in 
\pi_1  (
M \backslash \overline{\mathcal{S}_R} )$.
}

\noindent 
In fact, the surface 
$\R_X\subset M\times\CW_t$ 
in \eqref{R_X-definicion},
can be reconstructed 
by using copies of the fundamental region
$\Omega$ by
the analytical continuation of $\Psi_X$ across the
$\gamma_k$ as in the construction of $\Lambda$.
Note that the 
$\Omega_\varrho$ are isometric copies of $\Omega$,
using the flat metric $g_X$ on $\R_X$ 
arising from $\pi_2 ^* (\del{}{t})$.

\begin{remark}
\label{la-esfera-no-aparece}
1. 
Even though in $\mathfrak{M}$ the corresponding 
1--form of time 
$\widetilde{\omega}_X \doteq d\widetilde{\Psi_X}$ always has 
zero residues, we shall still add the adjective \emph{nonzero residue} 
when naming those transcendental singularities 
$\widetilde{U}_a$
of $\widetilde{\Psi_X}^{-1}$
whose corresponding singularity 
$U_a \doteq \pi(\widetilde{U}_a)$ 
is a nonzero residue transcendental
singularity of $\Psi_{X, \, \Lambda}^{-1}$.
See Definition \ref{definicion-singularidades-en-Lambda}.2.

\noindent 2. Assuming that 
$\overline{\mathcal{S}_R} \neq \varnothing$,
by simple inspection,
we obtain that $\mathfrak{M}$ is biholomorphic
to $\Delta$ or $\CC$, the case $\CW$ does not appear. 
\end{remark}

A direct application of Proposition \ref{teo-conf-valores}
yields the result below.

\begin{corollary}
\label{cubierta-universal-singularidades}
Let $\Psi_X: 
M \backslash (\mathbb{E} \cup \mathcal{Z}_{R})
\longrightarrow \CW_t$ 
be a multivalued additively automorphic singular 
complex analytic function, 
as in \eqref{Psi-sec-3-2},
with fundamental domain $\Lambda$, 
and let
$\widetilde{\Psi_X}: 
\mathfrak{M} \backslash \pi^{-1}(\mathbb{E}_0)
\longrightarrow \CW$
be its extension to the universal cover $\mathfrak{M}$.

\begin{enumerate}[label=\arabic*),leftmargin=*]
\item 
For each singular value $a\in\CW_t$ of 
$\Psi_{X,\, \Lambda}$,
there are an infinite number of 
singular values 
$\{a+\Xi_\varrho 
\ \vert \  
\varrho \in \pi_1(M \backslash 
\overline{\mathcal{S}_R} ) \} \subset\CC_t$ of 
$\widetilde{\Psi_X}$.
In case that $a=\infty$,  
the singular value $\infty$ of 
$\widetilde{\Psi_X}$
has infinite multiplicity.

\item 
The function $\widetilde{\Psi_X}$ has an infinite 
number of ideal points $\widetilde{U}_{a+\Xi_\varrho }$,
each of which has the same qualitative behaviour, 
on each copy of $\Lambda$, 
as that of the ideal point $U_a$ of 
$\Psi_{X,\, \Lambda}$.

\item
When $M$ is compact and 
$\mathfrak{M}$ is biholomorphic to $\Delta$, the 
nonzero residue 
transcendental singularities of 
$\widetilde{\Psi_X}$, say $\{\widetilde{U}_a \}$, 
are a dense subset of $\partial\Delta$.

\item
When $\mathfrak{M}$ is biholomorphic to $\CC$,
consider its compactification $\CW_z$.

\begin{enumerate}[label=\roman*),leftmargin=*]
\item
If  $\mathcal{S}= \mathcal{S}_{R}= \mathcal{Z}_{R}$,
then $\infty \in \CW_z$ is a simple pole of 
$\widetilde{\Psi_X}$.

\item
If  $\mathcal{S}= \mathcal{S}_{R}\neq \mathcal{Z}_{R}$,
then $\infty$
is an isolated essential singularity of $\widetilde{\Psi_X}$.

\item
Otherwise, 
$\infty$
is an nonisolated essential singularity of $\widetilde{\Psi_X}$.
\end{enumerate}
\end{enumerate}
\end{corollary}

\begin{proof}
Assertion (3) is true by simple inspection.

For assertion (4),
assume that $\Psi_X: \CW_z \longrightarrow \CW_t$ and 
$\mathcal{S}_{R} =\{0,\infty\}$, 
thus the lift to the 
universal cover is

\centerline{$
\widetilde{\Psi_X}: \CC_z \longrightarrow \CW_t,
$}

\noindent which
has a singularity at $\infty \in \CW_z$.

\noindent 
The case (i), where 
$\mathcal{S} = \mathcal{Z}_{R}$ is equal to 
two simple poles of $\omega_X$ at $0, \, \infty$,
determines a 
simple pole of $\widetilde{\Psi_X}$ at $\infty$,
recalling the normal form of $\omega_X$
in Proposition \ref{2-forma-normal}.

\noindent 
By an analogous argument in case (ii), 
since 
$\mathcal{S}= \mathcal{S}_{R} \neq \mathcal{Z}_{R}$, 
it follows that
$\widetilde{\Psi_X}$ has an 
isolated singularity at $\infty$. 
Once again, by Proposition \ref{2-forma-normal}
it is an essential singularity of
$\widetilde{\Psi_X}$.

\noindent 
For case (iii),
if $\mathcal{S}_{R}$ is equal to 
two points $0, \, \infty$,
and $\mathcal{S}_{R} \neq \mathcal{S}$, 
then the singularity of  
$\widetilde{\Psi_X}$ at $\infty$ has an accumulation 
of the zero residue singularities,
in complete detail, points in 
$\mathcal{Z}_{0} \cup \mathcal{P}$.
\end{proof}

The noncompact case  
for $M$ in assertion (3) 
of Corollary \ref{cubierta-universal-singularidades}
is left to the reader. 

\subsection{Equivalence relation on the singularities of $\Psi_{X,\, \Lambda} ^{-1}$} 
\label{Relacion-equivalencia}
Because of the biholomorphism between  $\R_X$ and the universal cover $\mathfrak{M}$
of $M\backslash\overline{\mathcal{S}_R}$,
and
Proposition \ref{teo-conf-valores}.2, 
it is natural
to define an equivalence relation
for different choices of $\Lambda$.

\begin{definition}\label{clases-equivalencia-singularidades}
Consider two singularities
$U_{a}\subset\Lambda_1$ 
of $\Psi_{X,\Lambda_1}^{-1}$ over $a\in\CW_t$,
and 
$U_{\widetilde{a}}\subset\Lambda_2$ 
of $\Psi_{X,\Lambda_2}^{-1}$ over $\widetilde{a}\in\CW_t$.
They 
are in the same \emph{equivalence class $[[U_a]]$}
if there exists a cover transformation 
$\varrho \in \pi_1 (M \backslash \overline{\mathcal{S}_R})$
such that the following occurs:

\begin{enumerate}[label=\alph*),leftmargin=*]

\item \label{condicion1}
$\Omega_2 = \varrho (\Omega_1) \subset \R_X 
$,

\item \label{condicion2}
$a=\widetilde{a}+\Xi_{\varrho}$, where 
$\Xi_{\varrho}
=\int_{\varrho } \omega_X$,

\item \label{condicion3}
for each $\rho>0$, there exist $\rho_1, \rho_2 >0$ such that

$\varrho \Big( \pi_1^{-1} \big( U_a(\rho_1) \big) \Big) \subset 
\pi_1^{-1} \big( U_{\widetilde{a}}(\rho) \big) 
\subset \Omega_2
\subset \R_X $

\noindent
and

$\varrho^{-1} \Big( \pi_1^{-1} \big( U_{\widetilde{a}}(\rho_2) \big) \Big) 
\subset  \pi_1^{-1} \big(  U_{a}(\rho) \big) 
\subset \Omega_1
\subset \R_X 
$. 
\end{enumerate}

\noindent
There exists an 
\emph{equivalence class $[\cdot]$ of 
singular values} 
induced by the equivalence class $[[\cdot]]$. 

\noindent
We shall say \emph{$[[U_a]]$ is a singularity class of $\Psi_X^{-1}$ over 
the singular value class $[a]$}.
\end{definition}

The equivalence relation is well defined; 
we leave the proof for the interested reader.

\begin{remark}\label{remark-abuso-notacion-multivaluadas}
1. Clearly, 
condition (b) in Definition 
\ref{clases-equivalencia-singularidades}
is necessary but not sufficient for the equivalence relation 
on the singular values.

\noindent
2. A convenient abuse of
notation is to say

\centerline{
``the singularity $ U_a $ of $\Psi_X^{-1}$ over the 
singular value $a$'',
}

\noindent
when in reality we should say

\centerline{
``the singularity class $[[U_a]]$ of $\Psi_X^{-1}$ over the 
singular value class $[a]$''.
}
\end{remark}

\section{Singularities of $\Psi_X^{-1}$ from the perspective of vector fields}
\label{perspectiva-de-campos}

Because of the correspondence between singular complex analytic vector fields 
$X$ and additively automorphic singular complex analytic functions $\Psi_X$, 
given by
Proposition \ref{basic-correspondence};
the study of the singularities of $\Psi_X^{-1}$,
both in the single--valued case and in the multivalued additively automorphic 
case, benefits from the perspective of vector fields.

\begin{figure}[h]
\begin{center}
\includegraphics[width=0.55\textwidth]{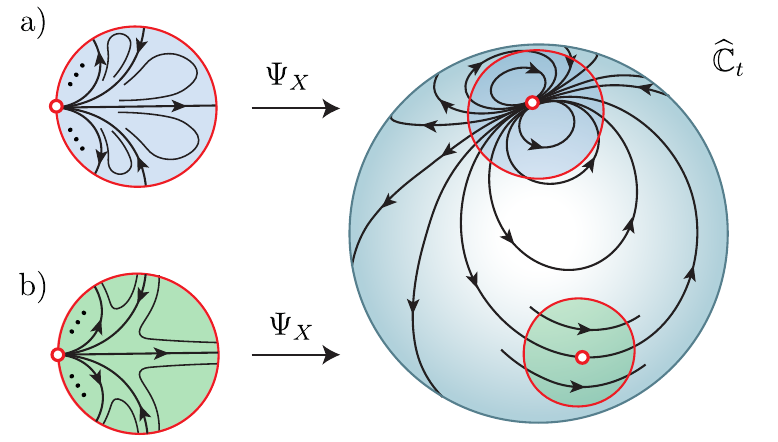}
\caption{
(a) Elliptic tracts
arise from the asymptotic value $\infty\in\CW_t$. 
(b) Hyperbolic  tracts
arise from finite asymptotic 
values $a \in\CW_t$.
The asymptotic values $a$, $\infty$ are
represented by small red circles.
}
\label{2-sectores}
\end{center}
\end{figure}

\begin{example}[Example \ref{singularidad-logaritmica} revisited]
\label{singularidad-logaritmica-revisada}
The distinguished parameter 

\centerline{
$\Psi_X(z) = 
{\displaystyle \int^z }  \e^{-\zeta} d\zeta: \CC_z
\longrightarrow 
\CC_t \backslash \{ 0 \}$,}

\noindent 
with two logarithmic singularities 
over the asymptotic values $0$ and $\infty$,
has

\centerline{$X(z) = \e^z \del{}{z}$}

\noindent 
as its associated vector field.
By considering the
phase portrait of the associated
vector field,  the exponential tracts 

$U_0(\rho)=\{ \Re{z} > \log ( 1/ \abs{\rho} ) \}$  
over the asymptotic value $0$, and 

$U_\infty(\rho)=\{ \Re{z} < \log(1 /\abs{\rho} ) \}$
over the asymptotic value $\infty$

\noindent
can be clearly distinguished.
See Figures \ref{2-sectores} and 
\ref{flores-elipticas-hiperbolicas}.a.
\end{example}

As an advantage of the existence of a vector field
$X$ associated to a function $\Psi_X$, we can
refine exponential tracts.

\begin{definition}
\label{tract-eliptico-hiperbolico}
1) 
The pairs

\smallskip 

\centerline{
$\mathscr{U}_{H}=\big( \{ \Re{z} >0 \}, \e^{z} \del{}{z} \big)$, 
\qquad \
$\mathscr{U}_{E}=\big( \{ \Re{z} <0 \}, \e^{z} \del{}{z} \big)$
}

\smallskip 

\noindent
are the \emph{hyperbolic tract over 0}
and 
\emph{elliptic tract over $\infty$} of 
$X(z)=\e^z \del{}{z}$ 
respectively.
See Figure \ref{2-sectores}.

\noindent
2) 
The pair
$\big( U_a (\rho), X  \big)$ is a
\emph{hyperbolic tract over the asymptotic value $a$
of $X$,}
or
\emph{elliptic tract over the asymptotic value $a=\infty$
of $X$, 
}
if there is a biholomorphism
$\Upsilon:(U_a (\rho),X) \subset M
\longrightarrow
\mathscr{U}_{H}$, 
or to $\mathscr{U}_{E}$, 
respectively.
\end{definition}

\begin{figure}[htbp]
\begin{center}
\includegraphics[width=0.80\textwidth]{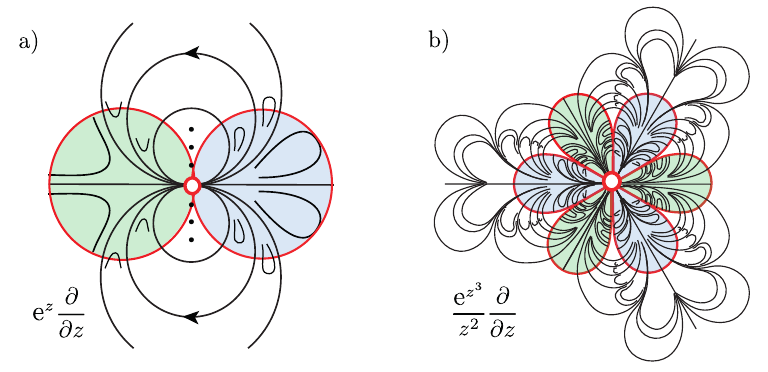}
\caption{
Geometry of exponential vector fields. 
(a) For $X(z)= \e^{z}\del{}{z}$
of Example \ref{singularidad-logaritmica-revisada}, the
essential singularity at $\infty\in\CW_z$
gives rise to two logarithmic singularities;
one hyperbolic tract over 
the singular value $0$,
and one elliptic tract over $\infty\in\CW_t$.
(b) For $X(z)=(\e^{z^3}/z^2) \del{}{z}$
of Example \ref{ejemplo-a=b}.1,  
the essential singularity at $\infty\in\CW_z$
gives rise to six logarithmic singularities.
There are
three hyperbolic tracts with finite asymptotic value $0$, 
with multiplicity 3, and
three elliptic tracts with asymptotic value $\infty$, 
accurately denoted
$\infty_1, \infty_2, \infty_3$.
The colouring 
scheme for petal regions is as follows: 
green for hyperbolic tracts and
blue for elliptic tracts, it will be used consistently throughout. 
}
\label{flores-elipticas-hiperbolicas}
\end{center}
\end{figure}

Certainly, the notion of biholomorphism is rigid. 
It is suitable for our present work since
we gain flexibility of this notion by applying it
to open Jordan domains of 
$(M, X)$
and under
variations of the radius $\rho$.

Let us recall the following theorem, 
cited in the Introduction in a brief version,
due to R. Nevanlinna,
that applies to single--valued functions.

\begin{theorem*}[Nevanlinna's isolated singular values, 
\cite{Nevanlinna1}  Ch.~XI \S1.3,
\cite{Zheng} Thm.~6.2.2]
Let $\Psi_{X}: \CC_z \longrightarrow \CW_t$
be a single--valued meromorphic function, 
and let $a$
be an isolated singular value for $\Psi_X$.
If $U_a$ is a singularity of $\Psi^{-1}_X$ over $a$, then $U_a$ is algebraic or logarithmic.
\end{theorem*}

As an immediate consequence,
direct nonlogarithmic and indirect singularities of (single--valued) $\Psi^{-1}_X$ over $a$
imply that the singular value $a$ is nonisolated, {\it i.e.} $a$ is an accumulation point of singular values of $\Psi_X$.
There are, however, 
logarithmic singularities of $\Psi^{-1}_X$ over nonisolated asymptotic values $a$;
see for instance Example \ref{Ejemplo-ExpSin-z} and its corresponding Figure \ref{ExpSin-z}.
However, 
for multivalued additively automorphic functions $\Psi_X$
we also have to consider
the nonzero residue transcendental singularities of $\Psi_X^{-1}$, 
see Definition \ref{definicion-singularidades-en-Lambda}.2.
As seen in Remark \ref{remark-def-sing-inv}.2, 
the $\star$--transcendental singularities of $\Psi_{X,\,\Lambda}^{-1}$ is 
not algebraic or logarithmic even though the asymptotic value $\infty$ is 
isolated (in fact it is a direct nonlogarithmic singularity
over the isolated singular value $\infty\in\CW_t$).

\noindent
The above discussion shows that when working with
multivalued
additively automorphic singular complex analytic functions,
it is not enough to just consider the singular values of the ideal points $U_a$; 
one must also examine the neighbourhoods $U_a(\rho)$.
For this,
we introduce the following concept,
understood as in Remark 
\ref{remark-abuso-notacion-multivaluadas}.2,
{\it i.e.} with a choice of a fundamental domain $\Lambda$.

\begin{definition}\label{separate}
For $\Psi_X$ an additively automorphic singular
complex analytic function,
let $U_{a}, \ U_{b}$
be singularities of $\Psi^{-1}_X$ over the singular values 
$a$ and $b$, respectively.
A singularity \emph{$U_a$ is separate} if there exists $\rho>0$ such that

\centerline{$U_{a}(\rho)\cap U_{b}(\rho)
=\varnothing$,  \ for all \  $U_b \neq U_a$.}
\end{definition}

In the above definition the case $a=b$ is possible,
see Example \ref{ejemplo-a=b}.
In words, the ideal point $U_a$ is separate if for small enough $\rho>0$ the neighbourhood 
$U_{a}(\rho)$ does not intersect 
any neighbourhood of another ideal point.
Similarly, an ideal point is \emph{nonseparate} if and only if any neighbourhood $U_{a}(\rho)$ 
always intersects a neighbourhood of another ideal point.

\begin{example}
\label{ejemplo-a=b}
1. \emph{Separate singularities, case $a=b$.}
Consider 

\centerline{$X(z)=\dfrac{\e^{z^3} }{ z^2} \ddel{}{z}$. }

\noindent 
The corresponding $\Psi_X^{-1}$ has six logarithmic singularities arising from 
the essential singularity of $X$ at $\infty\in\CW_z$ and an algebraic singularity
arising from the pole of $X$ at the origin.
All the singularities of $\Psi_X^{-1}$ are separate.
The asymptotic values are

\centerline{
$a_1=a_2=a_3=0, 
\quad\text{ and }\quad
\infty_1 =\infty_2 =\infty_3 = \infty$,}

\noindent 
\emph{i.e.} there are two asymptotic values, 
each of multiplicity 3.  
As can be seen in Figure
\ref{flores-elipticas-hiperbolicas}.b, 
there are six singularities of 
$\Psi_X^{-1}$, 
corresponding to three hyperbolic tracts over $0$
and three elliptic tracts over $\infty$. 
Full details appear in example 4.2 of \cite{AlvarezMucinoII}.

\noindent 
Thus the singular values 
$a$ and $b$ can be the same  
and yet $U_a$ and $U_b$ 
can be different singularities of $\Psi^{-1}_X$.

\noindent 
2. \emph{Nonseparate singularities, case $a \neq b$.} 
In Example \ref{Ejemplo-ExpSin-z}, 

\centerline{
$\Psi_X (z) = \e^{\sin(z)-z}$,}

\noindent 
is considered. 
Among other things, it is shown that 
for any given $\rho>0$, each neighbourhood 
$U_{\infty,-}(\rho)$ and $U_{0,+}(\rho)$, 
with asymptotic values $\infty$ and $0$ respectively,
contains: 
\begin{itemize}[label=$\bigcdot$,leftmargin=*]
\item 
an infinite number of neighbourhoods $U_{a_{k\pm}}(\rho)$, 
with asymptotic value $0$ or $\infty$, for $k$ odd or even, respectively, 

\item 
an infinite number of critical points. 
\end{itemize}

\noindent 
Thus both $U_{\infty,-}(\rho)$ and $U_{0,+}(\rho)$ 
are nonseparate, 
Figure \ref{ExpSin-z} illustrates this fact.
\end{example}

\begin{remark}
The notion of separate is of a topological nature. 
Thus, even when dealing with multivalued
additively automorphic singular analytic functions $\Psi_X$, 
for small enough $\rho>0$, the 
neighbourhoods $\{U_a(\rho)\}$ are well defined.
One just needs to recall that as soon as a choice of fundamental domain $\Lambda$ has been made,
all happens inside the chosen $\Lambda$.
\end{remark}

With the notion of separate singularity of $\Psi_X^{-1}$, 
we can improve 
Nevanlinna's isolated singular values theorem.

\begin{theorem}[Separate singularities]
\label{teo-logaritmicas-separada}
Let 
$\Psi_{X}: M \longrightarrow \CW_t$
be a  
additively automorphic singular complex 
analytic function,
as in \eqref{parametro-local}.
A singularity $U_a$ of $\Psi_X^{-1}$
is separate if and only if $U_a$ is one of the following:
\begin{enumerate}[label=\arabic*),leftmargin=*]
\item
algebraic, 

\item
$\star$--transcendental, 

\item
logarithmic.
\end{enumerate}
\end{theorem}

\begin{proof}
$(\Leftarrow)$ For cases (1) and (3) note that 

\centerline{
$\Psi_{X,\,\Lambda}:U_a(\rho)\longrightarrow D(a,\rho)\backslash\{a\}$
}

\noindent
is an unbranched holomorphic covering for sufficiently small $\rho>0$.
In case (1), 
the covering 
is of finite degree and in case (2) it
is the universal covering, in accordance with Definition
\ref{direct-indirecta-logaritmica}.
In either case,
$U_a$ is separate.

For case (2), recall that Equation
\eqref{nonzero-star--transcendental-singularity}
provides a local normal form as 

\centerline{
$\Psi_{X,\,\Lambda}(z) = \lambda \log(z) + 1/z^{s-1}$, 
\ \ \
for $s\geq 1$.
}

\noindent
Moreover, the asymptotic value is $a=\infty$ and 
the neighbourhoods of $U_\infty$ are 
(up to biholomorphism)  of the form
$U_\infty(\rho)
\cong
\big(D(0,R)\backslash [0,R] \big)\cup [0,R]_+$, 
recalling the construction of $\Lambda$ in \S\ref{construccion-region-fundamental}.3. 
In fact, 
$$
\begin{array}{rcl}
\Psi_{X,\,\Lambda}:U_\infty(\rho)
\cong
\big(D(0,R)\backslash [0,R] \big) \cup [0,R]_+
& \longrightarrow & D(\infty,\rho)\backslash\{\infty\}
\\
z & \longmapsto & \lambda \log(z) + {1}/{z^{s-1}}
\end{array}
$$

\noindent
topologically 
is an $s$--fold unbranched holomorphic log--covering 
\emph{i.e.} $2s-2$ elliptic sectors followed by a parabolic sector determining $\lambda$; 
see Figure \ref{cubierta-incompleta}.
Clearly, $U_a$ is separate.

\medskip
\noindent
$(\Rightarrow)$
Now, we assume that $U_a$ is separate.
Thus 
given $U_b\neq U_a$,
there exists $\rho>0$ such that
$U_a(\rho) \cap U_b(\rho) = \varnothing$ in $\Lambda$.
In particular, this implies that
$U_a(\rho)$ does not contain any singular points
other than $z_{\tt s}=\lim\limits_{{\tt t}\to\infty} \alpha_a({\tt t})$, 
where $\alpha_a$ is the asymptotic path corresponding to $U_a$.

\noindent
Recalling Definition \ref{Nueva-def-vecindades},
$\Psi_{X,\,\Lambda} = \pi_2 \circ \pi_1^{-1}\vert_\Lambda$.
In fact 
\begin{equation}\label{pi2-es-cubriente}
\pi_2: V\big((\alpha_a,a),\rho\big) \longrightarrow D(a,\rho)\backslash\{a\}
\end{equation}
is an unbranched holomorphic covering, and 

\centerline{$\pi_1\vert_\Omega:
\Omega\longrightarrow\Lambda\backslash\mathbb{E}_0 \subset M$ 
}

\noindent
is a biholomorphism.
It follows that for any neighbourhood $U_a(\rho)\subset \Lambda$ 
of a singularity $U_a$ of $\Psi_X^{-1}$ one has  
the diagram
\begin{center}
\begin{picture}(180,65)(0,20)

\put(-131,40){\vbox{
\begin{equation}\label{factorizando-Psi}\end{equation}
}}

\put(-25,75){$M \supset \Lambda \supset U_a(\rho) $}

\put(115,75){$V\big((\alpha_a,a),\rho\big)\cap\Omega\subset\R_X$}

\put(108,78){\vector(-1,0){60}}
\put(75,85){$\pi_1\vert_\Omega$}

\put(133,65){\vector(0,-1){30}}
\put(138,47){$ \pi_{2} $ }

\put(38,65){\vector(2,-1){73}}
\put(50,39){$ \Psi_{X,\, \Lambda} $}

\put(115,20){$D(a,\rho)\subset\CW_t$,}

\end{picture}
\end{center}
\noindent
where $V\big( (\alpha_a,a), \rho \big)$ is the component of 
$\pi_2^{-1}\big( D(a,\rho) \big)$ such that 

\centerline{
$U_a(\rho) \doteq \pi_{1}|_{\Omega} \Big( V\big( (\alpha_a,a), \rho \big) \Big)$.
}

\noindent
Thus, in order to specify the neighbourhood
$U_a(\rho)$, we first choose a connected component 
$V\big((\alpha_a,a),\rho\big)\subset\R_X$ of $\pi_2^{-1}\big(D(a,\rho)\big)$.

\noindent 
Since \eqref{pi2-es-cubriente} is an unbranched holomorphic covering, 
it follows that the closure of $V\big((\alpha_a,a),\rho\big)$ 
in $\R_X$ is  topologically 
a disk or a punctured disk.

\noindent 
Having identified $V\big((\alpha_a,a),\rho\big)$, we now intersect with $\Omega$.
Once again, recalling the construction of $\Lambda$, particularly $\Gamma$, 
in \S\ref{construccion-region-fundamental}.3,
three cases appear:

\noindent 
(A) $\partial\Omega \cap V\big((\alpha_a,a),\rho\big) = \varnothing$, or

\noindent 
(B) $\partial\Omega \cap V\big((\alpha_a,a),\rho\big) = 
\widehat{\gamma}_1
$
for a simple path
$\widehat{\gamma}_1$ 
that has as one of its
extrema the branch point $(\alpha_a,a)$.

\noindent 
(C) $\partial\Omega \cap V\big((\alpha_a,a),\rho\big) = 
\widehat{\gamma}_1 \cup \widehat{\gamma}_2
$,

\noindent 
for some 
simple paths $
\widehat{\gamma}_1 \cup \widehat{\gamma}_2
\subset\R_X$ that have as 
common extrema the branch point $(\alpha_a,a)$.

\noindent
Note that $\pi_1(
\widehat{\gamma}_j
)=\gamma_j \subset \Gamma\subset M$.

\smallskip
In case (A), since $\pi_{1}\vert_\Omega$ is a biholomorphism
it follows immediately 
that 

\centerline{
$\Psi_{X,\,\Lambda}=\pi_2 \circ \pi_{1}^{-1}\vert_{\Lambda}:
U_a(\rho) \longrightarrow D(a,\rho)\backslash\{a\}$ 
}

\noindent
is an unbranched holomorphic covering.
Thus by 
\cite{Zheng} theorem 6.1.1, 
either:

\noindent
(A.i) there exists a biholomorphism $\Phi$ of $U_a(\rho)$ onto 
$\Delta^*\doteq\{z\ \vert\ 0<\abs{z}<1\}$ such that $\Psi_{X,\,\Lambda} =\Phi^k$ 
for some natural number $k$, or

\noindent
(A.ii) there exists a biholomorphism $\Phi$ of $U_a(\rho)$ onto the left half plane 
$\mathbb{H}_{-} = \{z \ \vert\  \Re{z} < 0\}$ such that $\Psi_{X,\,\Lambda} = \exp \circ\, \Phi$.

\noindent
For  
(A.i), 
$U_a$ is algebraic and for 
(A.ii), 
$U_a$ is logarithmic.

\medskip
Let us now examine cases (B) and (C).
By Definition \ref{eremenko1}, for $0<\rho^\prime<\rho$ the neighbourhoods 
satisfy
$\overline{V\big((\alpha_a,a),\rho^\prime \big)} \subset V\big((\alpha_a,a),\rho \big)$, 
where the closure is in $\R_X$.

\noindent 
Since \eqref{pi2-es-cubriente} is an unbranched holomorphic covering,
the closure of 
$V\big((\alpha_a,a),\rho\big)$ in $\R_X$
is topologically a disk or 
a punctured disk.

\begin{lemma}\label{deformacion-gammas}
Let $0<\rho^\prime<\rho$.
The paths 
$\widehat{\gamma}_1$ and $\widehat{\gamma}_2$
can be deformed to
$\hat{\widehat{\gamma}}_1$ and $\hat{\widehat{\gamma}}_2$,
within $V\big((\alpha_a,a),\rho \big)$ so that:

\noindent
(a) 
$\hat{\widehat{\gamma}}_1$ and $\hat{\widehat{\gamma}}_2$
do not intersect $V\big((\alpha_a,a),\rho^\prime \big)$, when 
$V\big((\alpha_a,a),\rho\big)$ is topologically a disk,

\noindent
(b) 
$\hat{\widehat{\gamma}}_1$ and $\hat{\widehat{\gamma}}_2$
coincide inside 
$V\big((\alpha_a,a),\rho^\prime \big)$, when 
$V\big((\alpha_a,a),\rho\big)$ is topologically
a punctured disk.
\end{lemma}

Figure \ref{deformacion-curvas-3} illustrates the lemma.
Note that the paths $\widehat{\gamma}_1$ and $\widehat{\gamma}_2$ 
do not change outside of $V\big((\alpha_a,a),\rho \big)$, hence do not affect other 
singularities of $\Psi_{X,\,\Lambda}^{-1}$.

\begin{proof}
Follows immediately from the fact that $U_a$ is separate, and hence we can deform 
$\gamma_1$ and $\gamma_2$ 
in the open set $U_a(\rho) \backslash \overline{U_a(\rho^\prime)}$, 
leaving the extrema at $\partial U_a(\rho)$ and 
the branch point $(\alpha_a,a)$ fixed.
\end{proof}

\noindent
Case (i) tells us that 
$\partial \Omega \cap V\big((\alpha_a,a),\rho^\prime \big) = \varnothing$
and we have reduced to case (A) above, 
so $U_a$ is an algebraic or logarithmic singularity of $\Psi_{X,\,\Lambda}^{-1}$.

\noindent
For case (ii), up to biholomorphism

\centerline{
$V\big((\alpha_a,a),\rho^\prime \big) \cap \Omega
=
\big(D(0,R)\backslash [0,R] \big) \cup [0,R]_+ \subset\R_X$,
}

\noindent
note that $[0,R]$ projects by $\pi_1$ to a trajectory
of $\Re{X}$. 
By simple inspection we can recognize that 
Figure \ref{cubierta-incompleta} describes $\Psi_X$.
Thus the singularity $U_a$ is a
$\star$--transcendental singularity of $\Psi_{X,\,\Lambda}^{-1}$.
This completes the proof of the theorem.
\end{proof}

A list of the simplest singular behaviours 
is provided by the theorem below. 

\begin{theorem}[Topological behaviour of $\Re{X}$ and 
the singularities of $\Psi^{-1}_X$]
\label{singularidades-algebraicas-y-logaritmicas}
Let $X$ be a singular complex analytic vector field
and $\Psi_{X}$ the corresponding
additively automorphic singular complex analytic 
function,
as in \eqref{parametro-local}. 
Considering 
the phase portrait 
of $\Re{X}$ on the neighbourhood $U_a(\rho)$ for small enough $\rho>0$,
the name of the singularity of $X$, 
the type of singularity of $\Psi_X^{-1}$, 
and 
the residue of the 1--form of time $\omega_X$
at $z_{\tt s} \in \mathcal{S}$ with asymptotic value $a \in \CW_t$ 
as in
Definition \ref{defasymptoticvaluepath},
a partial correspondence is 
\begin{center}
\begin{tabular}{|c|l|l|c|}
\hline
&&& \vspace{-.3cm}
\\
$\big(U_a(\rho),\Re{X} \big)$ consists of & Name
of the singularity of $X$
& Type of  the singularity of $\Psi^{-1}_X$ &  Value of $Res(\omega_X,z_{\tt s})$
\\
&&& \vspace{-.3cm}
\\
\hline
\hline
$(2k+2)$ hyperbolic sectors & pole $p\in\mathcal{P}$ of  & 
algebraic singularity $U_a$  & $0 $
\\
 & multiplicity $-k\leq -1$ 
& over $a\in\CC_t$ &
\\
\hline
$(2s-2)$ elliptic sectors & \hfill $q\in\mathcal{Z}_0$ & 
algebraic singularity $U_\infty$ & $ 0$
\\
& zero $q\in\mathcal{Z}$ of & over $\infty$ & 
\\
\tabdashline &  multiplicity $s\geq 2$ & \tabdashline
and & & $\star$--transcendental singularity
 & \\
parabolic sectors& \hfill $q\in\mathcal{Z}_{R}$ & $U_\infty$ over $\infty$ & $\CC^*$
\\
\hline
source, sink or center & simple zero $q\in\mathcal{Z}_R$ & 
$\star$--transcendental singularity 
& 
\\
 &  & $U_\infty$ over $\infty$
& $\CC^*$
\\
\hline
hyperbolic tract & isolated essential & logarithmic transcendental & $\CC$
\\
& singularity $e\in\mathbb{E}$ & singularity $U_a$ over $a\in\CC_t$ & 
\\
\hline
elliptic tract & isolated essential & logarithmic transcendental & $\CC$
\\
 & singularity $e\in\mathbb{E}$ & singularity $U_\infty$ over $\infty$ &
\\
\hline
&essential singularity & non separate essential &
\\
?&$e\in\mathbb{E}$& transcendental singularity $U_a$ & $\CC$ 
\\
& &  over $a \in \CW_t$&
\\
\hline
\end{tabular}
\end{center}
\hfill $\Box$
\end{theorem}

\begin{figure}[h]
\begin{center}
\scalebox{0.7}{\includegraphics{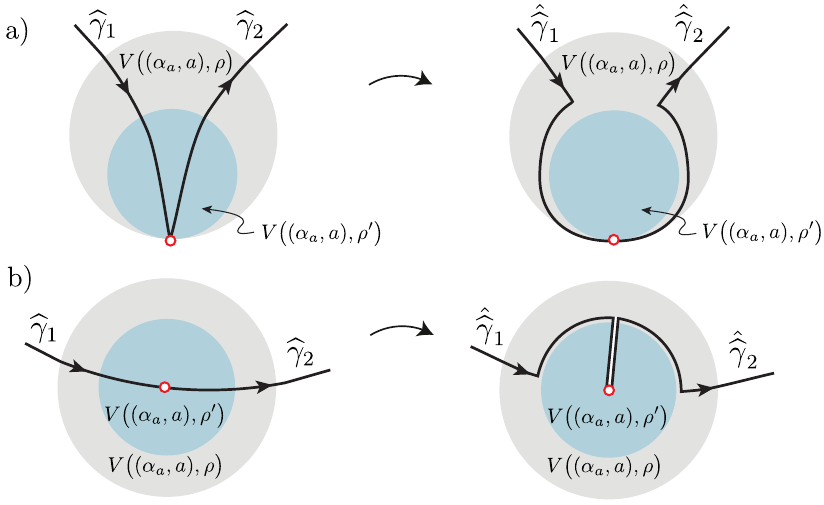}}\\
\caption{
Deformation of
the paths $\widehat{\gamma}_1$ and $\widehat{\gamma}_2$ 
inside the conformal disk 
$V\big((\alpha_a,a),\rho \big) \subset \R_X $,
as in Lemma \ref{deformacion-gammas}.
The red circle represents the branch point $(\alpha_a,a)$.
}
\label{deformacion-curvas-3}
\end{center}
\end{figure}

\begin{remark}\label{corolario-directa-nolog}
1. The above result emphasizes 
the dichotomy between 
finite and infinite singular values of $\Psi_X$.

\noindent 2.
A singular value $a \in \CW_t$ can 
admit several ideal points $\{ U_a \}$ over it.

\noindent
3.
In the table of Theorem 
\ref{singularidades-algebraicas-y-logaritmicas}, the 
question mark in the last row 
means that many other topologies occur.
For instance, 
the last row contains direct and non direct singularities.
\end{remark}

By Lemma \ref{Lemma-RX}.3  each 
asymptotic path of $\Psi_X$ can be realized as a
trajectory $z({\tt t} )$ of $\Re{\e^{i \theta}X}$
with $\alpha$ or $\omega$--limit $z_{\tt s}$, 
for some $\theta$, the converse is obvious.

\begin{definition}
A singularity $z_{\tt s} \in M$ of $X$ 
is \emph{reachable} when there exists an asymptotic path
of $\Psi_X$ with limit $z_{\tt s}$.
\end{definition}

\begin{example}[A non reachable singularity]
\label{sing-campo-no-implica-sing-psi}
Note that, not all singularities of $X$ 
have an associated singularity of $\Psi_X^{-1}$.
In our framework 
the choice of 
a singular point
$z_{\tt s} \in \mathcal{S}$ of $X$
does not
imply the existence of an asymptotic path and its
asymptotic value.  
For instance, consider the singular complex analytic vector field 

\centerline{
$X(z)=\frac{\cosh{z}}{\cos{z}} \del{}{z}$.
}

\noindent
It has simple zeros at $i(4k\pm 1)\frac{\pi}{2}$ and 
simple poles at $(4k\pm 1)\frac{\pi}{2}$, for $k\in\ZZ$.
Thus $z_{\tt s}=\infty\in\mathbb{E}_{nR}$ is an accumulation point
of $\mathcal{Z}_{R}\cup\mathcal{P}$;
see Figure \ref{cosh-cos}.
However,
since there is no asymptotic path tending to $z_{\tt s}=\infty$,
there is no singularity of $\Psi_X^{-1}$ associated 
to $z_{\tt s}=\infty$;
Elliptic functions $\wp(z)$ in $\CC_z$ 
provide analogous examples.
\end{example}

\begin{figure}[htbp]
\begin{center}
\includegraphics[width=0.65\textwidth]{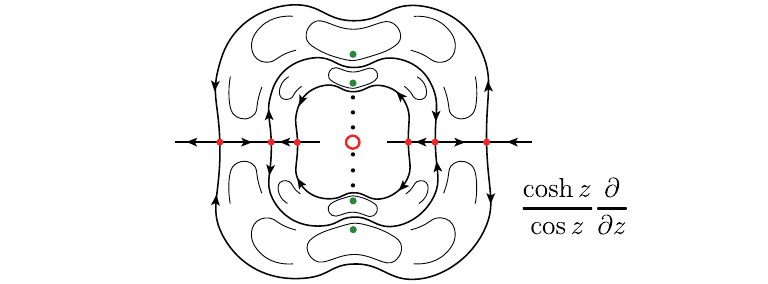}
\caption{
The singularity at $\infty$ of the vector field 
$X(z)= \big( \cosh z/\cos z \big)\del{}{z}$, 
the small red circle, 
does not have an associated singularity of $\Psi_X^{-1}$.
}
\label{cosh-cos}
\end{center}
\end{figure}

\begin{theorem}[Ideal points in terms of singularities of $X$]
\label{puntos-ideales-en-terminos-de-X}
Let $\Psi_X$ be an
additively automorphic 
singular complex analytic function,
as in \eqref{parametro-local}, and $X$ its 
corresponding singular complex analytic vector field.
A reachable singularity $z_{\tt s}\in\mathcal{S}$ of $X$ 
determines at least
one singularity of $\Psi_X^{-1}$.

\begin{enumerate}[label=\arabic*),leftmargin=*] 
\item If the singularity 
$z_{\tt s} \in \mathcal{S}_0$ of $X$ has {\bf residue zero},
then one of the following cases occurs.
\begin{enumerate}[label=\alph*),leftmargin=*]
\item 
A zero of $X$ of order 2 determines: 
a simple pole of $\Psi_X$, 
a nonsingular point of $\Psi_X$, and 
an ordinary point of $\Psi_X^{-1}$.

\item 
A pole or a zero (of order greater than 2) of $X$
determines: 
a critical point of $\Psi_X$, and 
an algebraic singularity of $\Psi_X^{-1}$.

\item
An essential singularity of $X$ determines: 
an essential singularity of $\Psi_X$, 
and at least one zero residue
essential transcendental singularities $\{ U_{a_\iota} \}$ of 
$\Psi_X^{-1}$
over $\{ a_\iota \}\subset\CW_t$.
\end{enumerate}
\item 
If the singularity $z_{\tt s} \in \mathcal{S}_R$ of $X$ has 
{\bf nonzero residue}, then it
can be understood within equivalent perspectives 
given by $\Lambda$ or $\mathfrak{M}$:

\begin{enumerate}[label=\Alph*),leftmargin=*]
\item 
In the context of 
a fundamental domain
$\Lambda$ with $\Psi_{X,\, \Lambda}$ 
as in \S\ref{caso-multivaluado-para-Psi}.

\begin{enumerate}[label=\alph*),leftmargin=*]
\item
A zero of $X$ determines a  
$\star$--transcendental singularity $U_\infty$ of 
the inverse of 
$\Psi_{X,\, \Lambda}^{-1}$
over $\infty\in\CW_t$.

\item 
An essential singularity of $X$
determines at least one 
nonzero residue essential
transcendental singularities $\{ U_{a_\iota} \}$ of
$\Psi_{X,\, \Lambda}^{-1}$
over 
$\{ a_\iota \}\subset\CW_t$
and $\{U_\infty\}$ over $\infty\in\CW_t$. 
\end{enumerate}

\item
In the context of 
the universal cover
$\mathfrak{M}$,
with $\widetilde{\Psi_X}$ the analytic extension of $\Psi_X$
as in \S\ref{caso-widetilde-Psi}.
\begin{enumerate}[label=\alph*),leftmargin=*]
\item When $\mathfrak{M}=\Delta$, each singularity $z_{\tt s} \in \mathcal{S}_{R}$ of $X$ 
determines an infinite number of
nonzero residue transcendental singularities of 
$\widetilde{\Psi_X}^{-1}$, 
$\{ \widetilde{U}_a \} \subset \partial\Delta$ located on the boundary
$\partial\Delta$ of $\mathfrak{M}$.
Moreover, when $M$ is compact, the set of ideal points
$\{ \widetilde{U}_a \}$ is dense in $\partial\Delta$.

\begin{enumerate}[label=\roman*),leftmargin=*]
\item 
A zero of $X$  determines only the asymptotic value 
$a=\infty$ and an infinite number of ideal points
$\widetilde{U}_\infty$.

\item
An essential 
singularity of $X$ determines an infinite number of 
nonzero residue essen\-tial
transcendental singularities of 
$\widetilde{\Psi_X}^{-1}$ over $\{a_\iota+\Xi\}\subset\CW_t$,
namely $\widetilde{U}_{a_\iota+\Xi}$, where
$\{\Xi\}$ is the set of linear combinations of
residues and periods of 
$\omega_X$, 
and $\{ a_\iota \}$ are the asymptotic values as in
(2.A.b) above.
\end{enumerate}

\item When $\mathfrak{M}=\CC$, we 
consider its compactification $\CW_z$.

\begin{enumerate}[label=\roman*),leftmargin=*]
\item
If  $\mathcal{S}= \mathcal{S}_{R}= \mathcal{Z}_{R}$, 
then $\infty \in \CW_z$ is a
simple zero of $\widetilde{X}$.

\item
If  $\mathcal{S}= \mathcal{S}_{R}\neq \mathcal{Z}_{R}$,
then $\infty$
is an isolated essential singularity of $\widetilde{X}$.

\item
Otherwise, 
$\infty$
is an nonisolated essential singularity of $\widetilde{X}$.
\end{enumerate}

\end{enumerate}

\end{enumerate}

\end{enumerate}

In (1.c), (2.A.b) and (2.B.a.ii) the number of 
asymptotic values depends on the order of growth of $X$.
\hfill\qed
\end{theorem}

As an illustrative family of 
Theorem \ref{puntos-ideales-en-terminos-de-X} 
it is natural to consider.

\begin{example}[Rational vector fields on $\CW$]
\label{Ejemplo-racional}
Let $q_1, \ldots , q_s$
be $s \geq 3$ distinct
points in $\CC$ and 
let $r_1, \ldots , r_s \in \CC^*$, such that 
$\sum r_k= 0$. We have the vector field

\centerline{
$X(z) =
\Big(  
\sum_{k=1}^s \dfrac{r_{k}}{z-q_k}
\Big)^{-1} \del{}{z}.$
}

\noindent
In this case the singular set is 
$\mathcal{S} = \mathcal{Z}_R \cup \mathcal{P}$ and
$\mathcal{Z}_R=\{ q_k\}_{k=1}^s$.
Consider $z_{\tt o} \in \CW \backslash \mathcal{S}$, 
its global distinguished parameter 

\centerline{
$\Psi_X(z)= 
{\displaystyle  \int_{z_{\tt o}}^z}
\Big(\sum_{k=1}^s \dfrac{r_{k} d\zeta }{\zeta -q_k} \Big)
=
\sum_{k=1}^s r_k \log(z-q_k)
+ C
$
}

\noindent 
is a multivalued additively automorphic 
singular complex analytic function.
We construct a fundamental 
region $\Lambda$ for $\Psi_X$,
as in
\S\ref{construccion-region-fundamental}.
Let $\{ \gamma_k\}$ be segments between two 
zeros  $\{ q_k\}_{k=1}^s$ of $X$, 
$\Gamma$ is the union
of these segments, 
thus 

\centerline{
$\Lambda =
\big( \CW_z \backslash  \Gamma  \big)
\cup \gamma_{k+}$
}

\noindent
is a fundamental region. 
Each simple zero $q_k$ of $X$ has 
a $\star$--transcendental singularity 
of $\Psi_{X,\,\Lambda}^{-1}$
over the asymptotic value $\infty$, 
as in Theorem \ref{singularidades-algebraicas-y-logaritmicas}.
Thus,
$\infty \in \CW_t$ is an asymptotic value
with multiplicity $s$. 
Considering the universal cover 
$\Delta$ of $\CW_z \backslash 
\mathcal{Z}_{R}$, 
Corollary
\ref{cubierta-universal-singularidades}.3 
and Theorem \ref{puntos-ideales-en-terminos-de-X}.2.B.a.i applies.
\end{example}


\section{Holomorphic families and sporadic examples}
\label{ejemplos-seccion}


\subsection{Exponential families}
\label{familias-exponenciales}
The family of entire functions with 
at most a finite number of logarithmic singularities is a cornerstone of the 
theory of entire functions.
A first analytic characterization due
to R. Nevanlinna is the following.

\begin{theorem*}[\cite{Nevanlinna1} Ch.~XI]
Entire functions $\Psi_X$ 
with  degree ${\tt p}-2$ polynomials 
as Schwarzian derivatives are 
precisely functions that have ${\tt p}$ 
logarithmic singularities.
\end{theorem*}

Also recall 
the pioneering work of E. Hille \cite{Hille}
and M. Taniguchi \cite{Taniguchi1}, \cite{Taniguchi2};
see R. Devaney \cite{Devaney} \S10 for a modern study.
For the  relations with the theory of the
linear differential equation 
$y^{\prime \prime} - P(z)y =0$, see 
\cite{Steinmetz} pp.\,156--157. 
We consider the family 
\begin{equation*}
\E(s, r,d)= \left\{
X(z) = \frac{Q(z)}{P(z)} \e^{E(z)} \del{}{z}
\ \Big\vert \
Q, \, P,  \, E \in \CC[z] \hbox{ of degree }
s, r, \, d \geq 1 
\right\}.
\end{equation*}

\noindent
Each $\E(s,r,d)$ is a 
holomorphic family of 
complex dimension $s+r+d+1$. 
Note that the functions $\Psi_X \in \E(0, r, d)$ 
are in the Speiser class, 
\emph{i.e.}
entire functions with a finite number
of critical and asymptotic values.
The vector fields 

\centerline{
$X(z)=\e^{az^2+bz+c}\del{}{z} \in \E(0,0, 2)$ 
\ and  \
$X(z)=\e^{z^d}\del{}{z} \in \E(0,0,d)$, for $d \geq 3$, 
}

\noindent 
were studied
by K. Hockett {\it et al.} \cite{HockettRamamurti}
using real vector field methods. 
In \cite{AlvarezMucino2} and \cite{AlvarezMucinoII}, the
families $\E(0,r,d)$ are examined and described using 
combinatorial methods.
Examples of phase portraits of $\Re{X}$,
for $X$ in $\E(0,0,1)$ and $\E(0,2,3)$ can be found in 
Figure \ref{flores-elipticas-hiperbolicas} 
and 
\cite{AlvarezMucino2}, \cite{AlvarezMucinoII},
for $X$ in $\E(1,0,1)$ in Figure \ref{flujo}, and for $X$ in $\E(2,0,1)$ in 
Figure \ref{campos-Delta}. 

Recall our convention from Equation 
\eqref{residuo-formula-general}, 
that the residue $Res(X,z_0)$
of a vector field $X$ at $z_0$ is 
the residue of the 1--form $\omega_X$ at $z_0$.
Let 
\begin{equation}
\label{campo-s-r-d}
X(z)=
\lambda 
\dfrac{\prod\limits_{j=1}^{m_0 + m_R} (z-q_j)^{\mu_j} }
{\prod\limits_{\iota=1}^n (z-p_\iota)^{\nu_\iota} }
\e^{E(z)} \del{}{z} 
\, \in\E(s,r,d), \ \lambda\in\CC^*,
\end{equation}

\noindent
where
$m_0$ denotes the number of zeros with zero residue, 
\\
$m_R$ denotes the number of zeros with nonzero residue,
\\
$n$ denotes the number of poles; 
$r=\sum_{\iota=1}^{n}\nu_\iota$
and
$s=\sum_{j=1}^{m_0+m_R}\mu_j$.

\begin{theorem}[The families $\E(s,r,d)$]
\label{familias-s-r-d}
Let 
\begin{equation*}
\Psi_{X} (z)= \int^z \frac{P(\zeta)}{Q(\zeta)} 
\e^{- E(\zeta)} d\zeta
\end{equation*}
be the additively automorphic singular complex analytic function 
arising from $X \in \E(s,r,d)$.

\begin{enumerate}[label=\arabic*),leftmargin=*]

\item
The function
$\Psi_X$ has $n+ m_0$ critical values ($n$ of them are finite) and 
$2d + m_R$ asymptotic values
(counted with multiplicity); 
$d$ over points in $\CC_t$ 
and $d+m_R$ over $\infty\in\CW_t$.

\item 
All the singularities of $\Psi_X^{-1}$ are separate
(algebraic, logarithmic or $\star$--transcenden\-tal). 

\item 
There is a hyperbolic tract over
each finite asymptotic value and   
an elliptic tract over each
infinite asymptotic value 
corresponding to the essential transcendental singularities
of $\Psi_X^{-1}$.

\item 
There is an $\mu_j$--unbranched holomorphic log--covering
for each nonzero residue zero $q_j$ of $X$.

\item 
The isolated essential singularity at $\infty\in\CW_z$ 
is the $\alpha$ or $\omega$--limit point of an infinite 
number of incomplete trajectories.
\end{enumerate}
\end{theorem}
\begin{proof}
Step 1.
We shall apply a rational approximation argument
to $X$ in \eqref{campo-s-r-d}, as in 
\cite{Nevanlinna1} ch. XI \S3.4  
and 
\cite{AlvarezMucinoII} \S4.3. 
Recall Euler's formula for the exponential, thus 
\begin{eqnarray*}
\label{aproximacion-racional}
X_{\tt n}(z) 
\doteq 
\frac{Q(z)}{P(z)
\left(1- \frac{E(z)}{{\tt n}} \right)^{\tt n} } \del{}{z} 
& _{\stackrel{ \xrightarrow{\hspace{30pt}} }{ 
{\tt n} \to\infty  }} & 
\frac{Q(z)}{P(z)} \e^{E(z)} \del{}{z} 
\doteq
X(z),
\\
\Psi_{X_{\tt n}}(z) 
\doteq
\int_{z_{\tt o}}^z \frac{P(\zeta)}{Q(\zeta)} 
\left(1- \frac{E(\zeta)}{{\tt n}} \right)^{\tt n} d\zeta
& _{\stackrel{ \xrightarrow{\hspace{30pt}} }{ {\tt n} \to\infty }} & 
\int_{z_{\tt o}}^z \frac{P(\zeta)}{Q(\zeta)} \e^{-E(\zeta)} d\zeta
\doteq 
\Psi_X(z) ,
\end{eqnarray*}
with the convergence being uniform on compact sets.

\noindent
In accordance with Equation \eqref{campo-s-r-d},
$X$ has the following features:
\begin{itemize}[label=$\bigcdot$,leftmargin=*]
 
\item
$n$ poles at the roots $\{ p_\iota \}_{\iota=1}^{n}$ of $P(z)$ with multiplicity $\{\nu_\iota\}$,
where $r=\sum_{\iota=1}^{n}\nu_\iota$;

\item
$m_0$ zeros with zero residue, at roots $\{ q_j \}_{j=1}^{m_0}$ of $Q(z)$ with multiplicity $\{\mu_j\}$,
where $s_0=\sum_{j=1}^{m_0}\mu_j$,

\item
$m_R$ zeros with nonzero residue, at roots $\{ q_j \}_{j=m_0 +1}^{m_0 + m_R}$ of $Q(z)$ 
with multiplicity $\{\mu_j\}$,
where $s_R=\sum_{j=m_0 + 1}^{m_0 + m_R} \mu_j$, and

\item
an isolated essential singularity at $\infty\in\CW_{z}$, 
the residue $Res(X,\infty)\in\CC$ 
may or not be zero.
\end{itemize}

\smallskip

Since the convergence $X_{\tt n}\to X$ is uniform, we may assume that 
for sufficiently large ${\tt n}$, 
the succession $\{ X_{\tt n} \}$, 
shares the following features with $X$:

\begin{itemize}[label=$\bigcdot$,leftmargin=*]

\item
The $n$ poles $\{ p_\iota \}_{\iota=1}^{n}$, arising from the factor $P(z)$, and 
the $m_0 + m_R$ zeros $\{ q_j \}_{\ell=1}^{m_0 + m_R}$, arising from the factor $Q(z)$,
are fixed (they do not depend on ${\tt n}$);
these poles and zeros coincide with those of
$\omega_X$.

\item
However the $d$ poles $\{ \widehat{e}_\sigma({\tt n}) \}_{\sigma=1}^{d}$ of $X_{\tt n}$,
each of multiplicity ${\tt n}$, arising from the factor 
$\left(1- {E(z)}/{{\tt n}} \right)^{\tt n}$, tend towards $\infty\in\CW_z$ as ${\tt n}\to\infty$. 

\item 
The phase portrait of $\Re{X_{\tt n}}$ at 
$\widehat{e}_\sigma({\tt n})$ consists of $2{\tt n} + 2$ hyperbolic sectors.

\item
$X_{\tt n}$ has a zero
of multiplicity $r-s+d{\tt n}+2$ at $\infty\in\CW_{z}$, 
the residue $Res(X_{\tt n},\infty)$ of this zero may or not be zero, 
however if $Res(X, \infty)\neq 0$
we may assume that $Res(X_{\tt n},\infty) \neq 0$.

\item
the phase portrait of $\Re{X_{\tt n}}$ at $\infty\in\CW_z$ consists of $2(r-s+d{\tt n})+2$ elliptic sectors
if $Res(X_{\tt n},\infty) = 0$ or \\
$2(r-s+d{\tt n})+2$ elliptic sectors and a parabolic sector
if $Res(X_{\tt n},\infty) \neq 0$.
\end{itemize}

\noindent
In the limit, when ${\tt n}\to\infty$, the zero at $\infty\in\CW_z$ and the poles 
$\{\widehat{e}_\sigma({\tt n}) \}_{\sigma=1}^{d}$ coalesce, forming an essential singularity of
$X$ at $\infty\in\CW_z$.
A careful examination of the phase portraits of $\Re{X_{\tt n}}$ as ${\tt n}\to \infty$ shows that
$\Re{X}$ has 
\begin{itemize}[label=$\bigcdot$,leftmargin=*]
\item
$d$ elliptic tracts and 

\item
$d$ hyperbolic tracts
\end{itemize}
angularly equidistributed about $\infty\in\CW_z$, see 
Figure 2 in \cite{AlvarezMucinoII}.

\smallskip

Step 2. 
Now let us consider the succession of additively automorphic functions $\{ \Psi_{X_{\tt n}} \}$ 
and its limit function
$\Psi_X$. 
We shall choose a fundamental domain $\Lambda$ (recall \S\ref{construccion-region-fundamental}).
Let $\Gamma=\cup_{\kappa=1}^{m_R} \gamma_\kappa$ be a simple path that passes through the 
$m_R$ poles $\{ q_\ell \}_{\ell=m_0 +1}^{m_0 + m_R}$ with non zero residue and let it end at
the pole at $\infty$ (which may or not have zero residue), \emph{i.e.} $\gamma_{m_R}$ has 
$\infty\in\CW_z$ as one of its extrema.
Furthermore, for large enough $N>>0$ we may assume that $\Gamma$ avoids all the singularities 
of $\omega_{X_{\tt n}}$ for all ${\tt n}>N$.
In this way, the fundamental domain 
$\Lambda=(\CW_z\backslash\Gamma) \cup _{\kappa=1}^{m_R} \gamma_{\kappa +}$ 
can be used with $\Psi_X$ and $\Psi_{X_{\tt n}}$ for all ${\tt n}>N$, so that we obtain the single--valued
functions $\Psi_{X_{\tt n},\,\Lambda}$ that converge uniformly on compact sets of $\Lambda$ to
$\Psi_{X,\,\Lambda}$.

\noindent
The succession $\{ \Psi_{X_{\tt n},\,\Lambda} \}$ and the function $\Psi_{X,\,\Lambda}$ 
have the following common properties.

\begin{itemize}[label=$\bigcdot$,leftmargin=*]

\item 
$n+m_0$ critical values corresponding to the poles and zeros with zero residue of $X$
arising from the factors $P(z)$ and $Q(z)$ respectively,

\item 
$m_R$ $\star$--transcendental singularities of $\Psi_{X_{\tt n},\,\Lambda}^{-1}$ 
arising from the 
zeros with nonzero residue of $X$,

\end{itemize}

However, the succession $\{ \Psi_{X_{\tt n},\,\Lambda} \}$ has:

\begin{itemize}[label=$\bigcdot$,leftmargin=*]

\item $d$ 
finite critical values corresponding to the $d$ zeros of $\omega_{X_{\tt n}}$ arising from the
factor $\left(1- {E(z)}/{{\tt n}} \right)^{\tt n}$,

\item 
the point $\infty\in\CW_z$ is a critical point of $\Psi_{X_{\tt n},\,\Lambda}$ with critical value $\infty$
or a $\star$--transcendental singularity of $\Psi_{X_{\tt n},\,\Lambda}^{-1}$ 
with asymptotic value $\infty$, depending on whether 
$Res(X_{\tt n},\infty)$ is zero or nonzero.
\end{itemize}

\smallskip

Step 3. Identification of the singularities. 
Clearly, $\Psi_X$ has a finite number of singular values, hence the singularities of $\Psi_X^{-1}$ are 
separate. 
The poles and zeros with zero residue of $X$ correspond to algebraic singularities of $\Psi_X^{-1}$.
The zeros with nonzero residue of $X$ correspond to $\star$--transcendental singularities
of $\Psi_X^{-1}$ (and hence an $\mu_j$--unbranched holomorphic log--covering at each 
nonzero residue zero of $X$).
Moreover, the limit function $\Psi_{X}$ has an essential singularity at $\infty\in\CW_z$ and
the phase portrait of $X$ shows $d$ elliptic tracts and $d$ hyperbolic tracts equidistributed about 
$\infty\in\CW_z$. 
Each elliptic tract accepts a class of asymptotic paths with asymptotic value 
$\infty\in\CW_t$ for $\Psi_{X,\,\Lambda}$;
each hyperbolic tract accepts a class of asymptotic paths with finite asymptotic value 
$a_\sigma\in\CC_t$ for $\Psi_{X,\,\Lambda}$.

\noindent
Finally, each hyperbolic tract provides an infinite number of incomplete trajectories with
the essential singularity $\infty\in\CW_z$ being their $\alpha$ or $\omega$ limit point.
\end{proof}

\begin{example}[$X\in\E(1,0,d)$ using rational approximation]
\label{aproximacion-E10d}
Let

\centerline{
$
X(z)= z \e^{z^d} \del{}{z}\in\E(1,0,d), 
\ \ \
\text{ for }d\geq1.$
}

\noindent
The corresponding distinguished parameter is 

\centerline{
$\Psi_X(z)=
{\displaystyle \int^z} \frac{\e^{-\zeta^d}}{\zeta} d\zeta
=
-\frac{1}{d} \Gamma \left(0,z^d\right),$}

\noindent
where $\Gamma (a,z)\doteq\int _z^{\infty } \zeta^{a-1} e^{-\zeta} d\zeta$ is the incomplete Gamma function.

\noindent
Euler's formula provides the approximation of $X$ by the vector fields

\centerline{
$X_{\tt n}(z)=\dfrac{ z }{ \big( 1 - \frac{z^d}{{\tt n}} \big)^{\tt n} }\del{}{z},
\ \ \ 
\text{ for } {\tt n} \geq 1,$
}

\noindent 
so

\centerline{
$\Psi_{X_{\tt n}} (z)=
{\displaystyle \int^{z}_{0} } 
\frac{\big( 1 - \frac{\zeta^d}{{\tt n}} \big)^{\tt n}}{\zeta} \, d\zeta
=
-\frac{1}{d (n+1)}
\left(1-\frac{z^d}{n}\right)^{n+1} \, _2F_1\left(1,n+1;n+2;1-\frac{z^d}{n}\right),$
}

\noindent
where $ _2F_1$ is the classical Gauss's hypergeometric function, 
see for instance \cite{Oberhettinger} ch.~15.

\noindent
The zeros of $X_{\tt n}$ are $0$, of order $1$, and $\infty\in\CW_z$ of order ${\tt n}d+1$;
with residue $1$ and $-1$ respectively.

\noindent
The poles of $X_{\tt n}$ are
$\{ \widehat{e}_\sigma({\tt n}) \doteq e^{\frac{2 i \pi \sigma}{d}} {\tt n}^{1/d} \}_{\sigma=1}^{d}$, 
of order $-{\tt n}$. 
Of course the poles of $X_{\tt n}$ are the critical points of $\Psi_{X_{\tt n}}$.

\noindent
Choosing $\Lambda=(\CW_z\backslash [-\infty,0] )\cup (-\infty,0)_+$, we can 
compute the critical values of $\Psi_{X_{\tt n},\,\Lambda}$

\centerline{
$\widetilde{e}_\sigma({\tt n}) \doteq \Psi_{X_{\tt n},\,\Lambda} (e^{\frac{2 i \pi \sigma}{d}} {\tt n}^{1/d})
= 0$, 
\ \ \
for $\sigma=1,\ldots,d$.
}

\noindent
Moreover, the finite asymptotic values $a_{\sigma}$ of 
$\Psi_{X,\,\Lambda}(z)=\int_{0}^{z} \zeta^r \e^{-\zeta^d} d\zeta$, are given by 

\centerline{
$a_{\sigma} = \lim\limits_{{\tt t}\to\infty} -
\frac{1}{d}
\Gamma(0,\alpha_\sigma^d({\tt t}) ) = 0 \in\CC_{t},
\ \ \
\text{ for } \sigma=1,\ldots,d.$
}

\noindent
We conclude that the critical values $\widetilde{e}_\sigma ({\tt n})$ converge, to the finite asymptotic values.

\noindent
Furthermore, travelling along the asymptotic paths 
$\alpha_{\sigma}({\tt t})={\tt t} \e^{2 i \pi (\sigma -d) / d} \e^{i\pi/d}$,
that arrive at $\infty\in\CW_{z}$ with angle $\frac{2\pi}{d}(\sigma-d)+\frac{\pi}{d}$, for
$\sigma=d+1,\ldots,2d$,
we see
that $\Psi_{X_{\tt n},\,\Lambda}(z)$ converges to $\infty\in\CW_t$.
Thus there are $d$ 
(classes of) asymptotic paths that give rise to the asymptotic value $\infty\in\CW_t$.

\noindent
Using the techniques\footnote{
The images were obtained using simple code written in the Julia\texttrademark\ 1.9.3 language 
which is particularly well suited for numerical computation. 
The code is freely available upon request.
} presented in \cite{Alvarez-Mucino-Solorza-Yee}, 
we visualize the phase portraits of $\Re{X_{\tt n}}$ and $\Re{X}$ 
for $s=1$ and $d=5$.
The poles $\{ e^{\frac{2 i \pi \sigma}{d}} {\tt n}^{1/d} \}_{\sigma=1}^d$ 
are portrayed as green dots.
Note that at $\infty\in\CW_{z}$ there is a zero of $X_{\tt n}$ of order 
exactly $d{\tt n}+1$.
See Figure \ref{figEjemploE105},
for the case $d=5$.
\end{example}
\begin{figure}[htbp]
\begin{center}
\includegraphics[height=0.9\textheight]{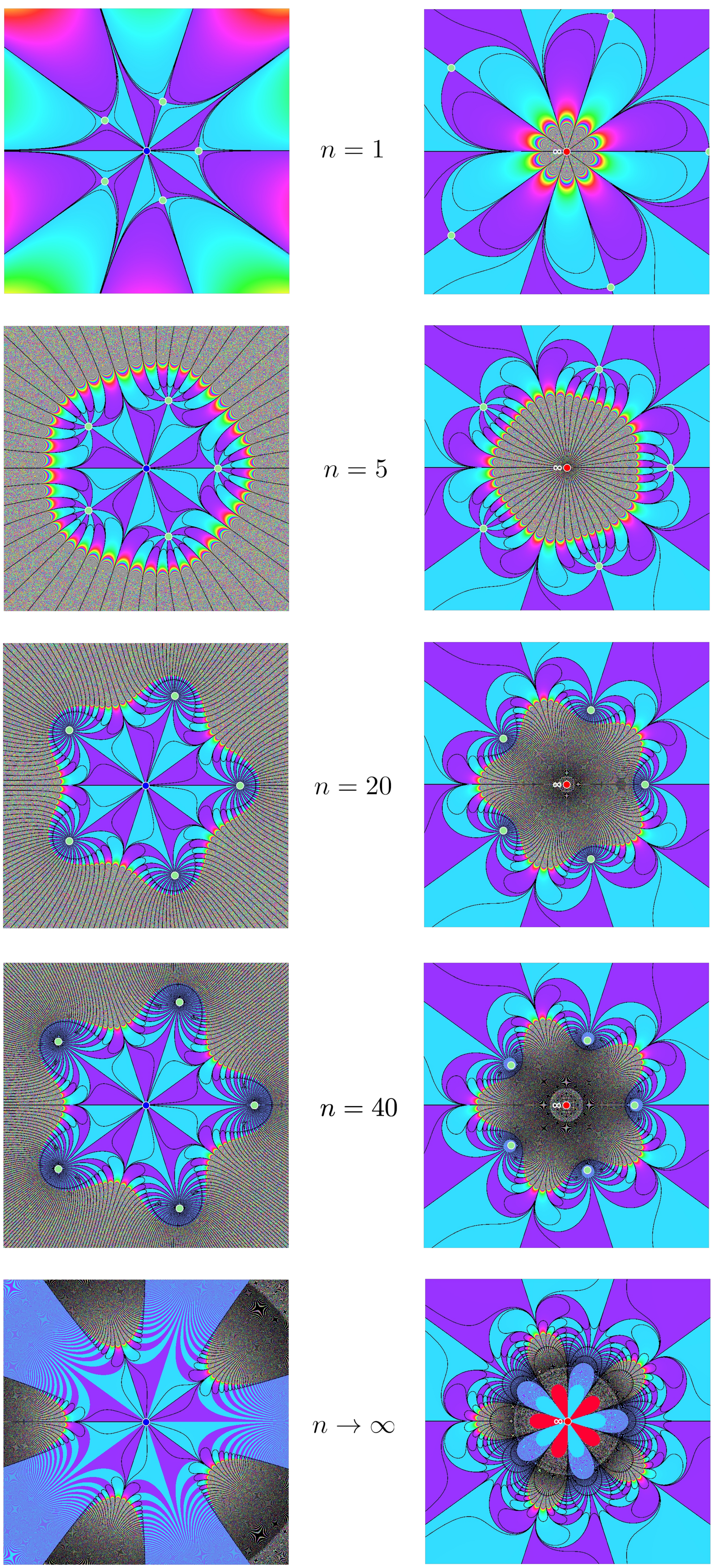}
\caption{
Phase portraits of $\Re{X_{\tt n}}$ for ${\tt n}=1, 5, 20, 40$ converging to $\Re{X}$ 
with $X\in\E(1,0,5)$ as in Example \ref{aproximacion-E10d}. 
Left hand side portrays a neighbourhood of the origin, and the right hand side a neighbourhood
of $\infty\in\CW_{z}$.
Note that by approaching $\infty\in\CW_{z}$ along paths that avoid the poles 
$\{ e^{\frac{2 i \pi \sigma}{d}} {\tt n}^{1/d} \}_{\sigma=1}^d$ (green dots), the value of 
$\Psi_{\tt n}(z)$ converges to $\infty\in\CW_t$. 
Images are high resolution, zooming is suggested particularly for high values of $n$.
}
\label{figEjemploE105}
\end{center}
\end{figure}

\begin{example}[$X\in\E(1,0,1)$]
\label{Ejemplo-multivaluado}
Let us consider the vector field

\centerline{
$X(z)=z \e^z \del{}{z}\in\E(1,0,1),$
}

\noindent
whose phase portrait of $\Re{X}$ on $\CW_z$
is sketched in Figure \ref{flujo}.
Its global distinguished parameter 

\centerline{
$\Psi_X(z)=
{\displaystyle \int_1^z} 
\frac{\e^{-\zeta} }{\zeta}d\zeta$
}

\noindent 
is  a multivalued additively automorphic 
singular complex analytic function.
In this case $\mathcal{S}_{R}=\{0,\infty\}$.

From the perspective of a fundamental region
\S\ref{construccion-region-fundamental}, 
we have that

\centerline{
$\Lambda=(\CW_z \backslash \gamma) \cup \gamma_+$,
}

\noindent
where $\gamma $ is a path as in Figure \ref{flujo}.
The fundamental 
region is

\centerline{$\Omega= \left\{\big(z,\Psi_X(z)\big) \ \vert\ 
z\in\Lambda \right\} \subset \R_X.$}

\begin{figure}[h]
\begin{center}
\includegraphics[width=0.75\textwidth]{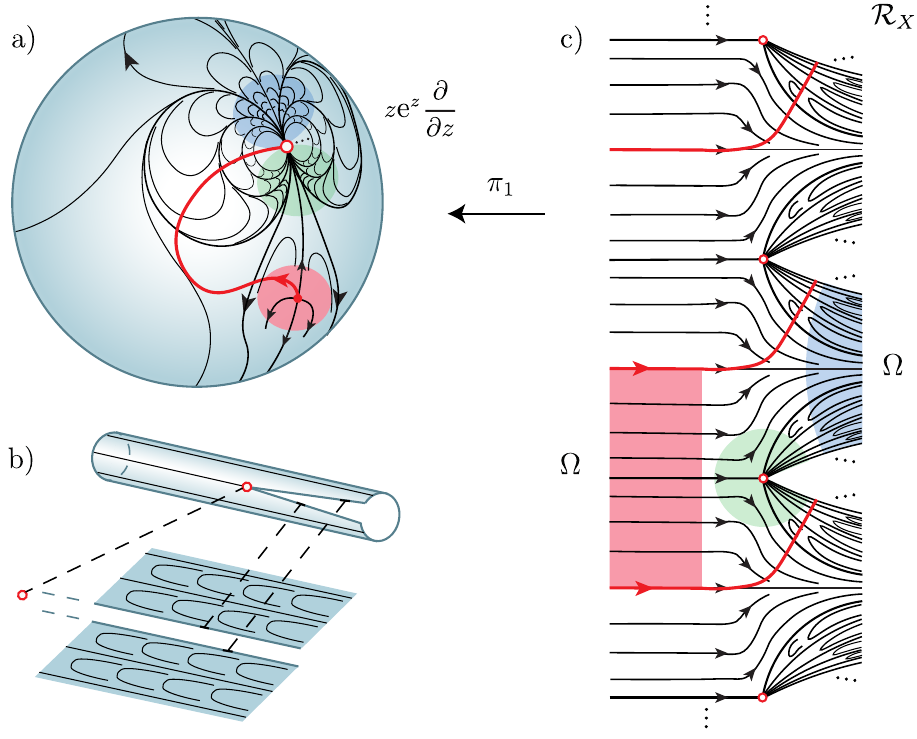}
\caption{
Let $X(z)=z \e^z \del{}{z}$, 
the 
singularities with nonzero residue
are $\mathcal{S}= \mathcal{S}_{R} =\{0, \, \infty\}$. 
(a) The essential singularity at $\infty$
is represented by a small red circle.
Here $\Gamma \subset \CW_z$ (in red)
is a path from $0$ to $\infty$.
(b) The flat metric $(\CW_z, g_X)$
is obtained from an infinite number of 
Reeb components
as in $\e^z\del{}{z}$
and a cylinder.  
(c) The Riemann surface $\R_X$ and a
fundamental region
$\Omega$ in the universal cover of
$\CW_z \backslash \mathcal{S}$ are sketched.  
Three singularities $U_{a_j}$ of $\Psi_{X,\,\Lambda}^{-1}$ 
whose neighbourhoods are 
hyperbolic and elliptic tracts
(coloured green and blue respectively),
and a $\star$--transcendental 
singularity coloured pink, arising from the
source at $z=0$. 
The colouring scheme is applied both in $\Lambda$ 
and $\Omega$.
}
\label{flujo}
\end{center}
\end{figure}
\noindent
Once again,  Figure \ref{flujo} shows a sketch of 
$\Lambda$, $\Omega$
and the Riemann surface $\R_X$.
According with Theorem 
\ref{singularidades-algebraicas-y-logaritmicas},
we have three singularities of $\Psi_X^{-1}$.

\noindent
$\bigcdot$ The simple zero $z=0\in\Lambda$ has 
associated a $\star$--transcendental singularity
of $\Psi_{X,\, \Lambda}^{-1}$ over the asymptotic value 
$\infty\in\CW_t$, 
its neihgborhood 
$U_\infty(\rho)\doteq
\Psi_{X,\,\Lambda}^{-1}\big(D(\infty,\rho)\big)$ is
coloured pink in $\Omega$, see
Figure \ref{flujo}.

\noindent
The essential singularity $\infty\in\Lambda$ has 
associated two nonzero residue essential transcendental 
singularities
of $\Psi_ {X,\, \Lambda} ^{-1}$; 

\noindent 
$\bigcdot$ over the asymptotic value
$0$, 
the neighbourhood 
$U_0(\rho)\doteq
\Psi_ {X,\, \Lambda}^{-1}\big(D(0,\rho)\big)$
is a hyperbolic tract, 
coloured green in Figure \ref{flujo}, and 

\noindent 
$\bigcdot$ over $\infty$, 
the neighbourhood
$U_\infty(\rho)\doteq
\Psi_ {X,\, \Lambda} ^{-1}\big(D(\infty,\rho)\big)$
is an elliptic tract,
coloured blue in Figure \ref{flujo}.

\noindent
The last two singularities are logarithmic.

From the perspective of the universal cover
$\mathfrak{M}$ of
$\CW_z\backslash\{0,\infty\}$: 
Corollary \ref{cubierta-universal-singularidades}.4.ii
applies.

\end{example}

\subsection{Families of periodic vector fields}
\label{familias-periodicas}

On $\CW_z$ there exists a correspondence between 

\noindent 
$\bigcdot$
singular complex analytic vector fields  
$X$ on $\CW_z$ of period $T \in \CC^*$ 
with $\omega_X$ having zero residues, and

\noindent 
$\bigcdot$
singular complex analytic functions $\Psi_X$
of period $T$.

Moreover, in such a case, 

\centerline{
\label{condicion-periodica}
$\Psi_X (z)= h \big( \e^{2\pi i z/T} \big)$
}

\noindent 
is single--valued, where 
$h$ is a suitable singular complex analytic function.  

\begin{theorem}[Families of periodic vector fields with single--valued $\Psi_X$]
\label{familias-P-r}
Let $X$ be a
singular complex analytic vector field 
on $\CW_z$ arising from a function
$\Psi_X$ in the family
$$
\mathscr{P}_{\tt r}\doteq
\left\{ 
\Psi_X(z)=R\big( \e^{2\pi i z/T} \big)
\ \big\vert\ 
R:\CW\longrightarrow\CW_t \text{ rational of degree } 
{\tt r} \geq 1 
\right\}.
$$
The following assertions hold.

\begin{enumerate}[label=\arabic*),leftmargin=*]

\item
$X$ is periodic of period $T \in \CC^*$ with
a unique essential singularity at $\infty\in\CW_z$.

\item
$\Psi_X$ has 
two asymptotic values
$a_1\doteq R(0)$ and 
$a_2\doteq R(\infty)$,
counted with multiplicity.

\item
Each of the two transcendental singularities of 
$\Psi^{-1}_X$ is logarithmic. 
The corresponding exponential tracts are
\begin{enumerate}[label=\roman*),leftmargin=*]
\item 
hyperbolic tracts when the asymptotic value is finite, and
\item 
elliptic tracts when the asymptotic value is $\infty$.
\end{enumerate}

\item 
If the critical point set 
$\mathcal{C}_R \subset \CW_z$ of $R$ 
satisfies that 
$\mathcal{C}_R \backslash \{0, \infty\}\neq\varnothing$,
then $X$ has an infinite number of poles accumulating at $\infty\in\CW_z$.

\item If $\infty\in\CW_t$ is not an asymptotic value, 
then $X$ has an infinite number of zeros of multiplicty 
2 and residue zero accumulating at $\infty\in\CW_z$.

\item 
The behaviour in (4) or (5)
depends on the
configuration of the two asymptotic values and infinity:

\begin{enumerate}[label=\roman*),leftmargin=*]
\item 
(Generic case.) Three distinct points
$\{a_1,a_2,\infty\} \subset\CW_t$.

\item 
Two distinct points 
$\{a_1=a_2,\, \infty\}$.

\item 
Two distinct points 
$\{a_1,\, a_2=\infty\}$ 
\, or \,
$\{a_1= \infty, \, a_2 \}$.
	
\item 
One distinct point 
$\{a_1=a_2=\infty\}$.
\end{enumerate}

\noindent
It provides a complete 
decomposition of the family $\mathscr{P}_{\tt r}$ into four subfamilies.
\end{enumerate}
\end{theorem}

As usual, generic means an open and dense set in 
the space of parameters of $\mathscr{P}_{\tt r}$.

\begin{proof}
The space of rational functions 
$R(w)$ of degree ${\tt r} \geq 1$
is an open Zariski set in 
$\mathbb{CP}^{2{\tt r} +1}$, 
hence
$\mathscr{P}_{\tt r}$ inherits this open complex manifold structure. 

Without loss of generality, assume that the period is 
$T=2\pi i$.  
Under pullback we have a diagram 
\begin{equation}
\label{caja-de-flujo-periodicos}
(\CW_z, X) 
\stackrel{\e^z}{\longrightarrow}
\big( \CW_w, R^* \del{}{t} \big)
\stackrel{R}{\longrightarrow}
\big( \CW_t, \del{}{t} \big) \, . 
\end{equation} 

\noindent
Here $R^* \del{}{t}$ is a rational vector field 
with

\noindent 
$\bigcdot$
zeros of order $\geq 2$ and residue zero, 
at the poles of $R$, 
and 

\noindent 
$\bigcdot$ 
poles at the critical points of $R$ 
in $\CC^*_w$ with
finite critical values. 

\noindent
From the above observations,
the statements (4) and (5) follow.

Statement (1) follows from the periodicity and essential singularity 
of $\e^z$.

Statement (2) follows from noting that the asymptotic values of $\e^{z}$ are precisely 
$0$ and $\infty$, thus the asymptotic values of $\Psi_X$ are
$a_1\doteq R(0)$ and $a_2\doteq R(\infty)$.

Note that $\Psi_{X}$ is the universal cover of a
neighbourhood of the transcendental singularities 
$U_a$ of $\Psi^{-1}_X$,
hence for the asymptotic values $a=a_1, a_2$ 
and $\rho>0$ sufficiently small,
we have

\centerline{
$U_a(\rho)=\Psi_X^{-1}\big(D(a,\rho)\big)=\log\Big(R^{-1}\big(D(a,\rho)\big) \Big)$.
}

\noindent
Thus, statements (3.i) and (3.ii) follow
from Theorem \ref{singularidades-algebraicas-y-logaritmicas}.

For statement (6), 
in accordance with Diagram \ref{caja-de-flujo-periodicos}, the
behaviour of $R$ provides a sharp description
of the zeros and poles of $X$,
as well as the exponential tracts of $\Psi_X$.
A systematic description of the different subfamilies in 
$\mathscr{P}_{\tt r}$
is given by the configuration of the two asymptotic values and infinity.

\smallskip

\emph{
i) Generic case.
A three distinct point $\{a_1,a_2,\infty\}$
configuration.
}

\noindent 
Clearly, the above condition defines a generic
set in $\mathscr{P}_{\tt r}$. 
Moreover, $X$ has an infinite number of zeros of multiplicity at 
least 2 and residue zero
accumulating at $\infty\in\CW_z$, corresponding to
assertion (5).
In addition, if the critical point set of $R$ is different from $0$ or $\infty$, 
then $X$ has an infinite number of poles accumulating at $\infty\in\CW_z$; as in assertion (4).
Finally, the neighbourhoods $U_{a_1}(\rho)$ and $U_{a_2}(\rho)$ 
of the singularities of $\Psi^{-1}_X$ will be hyperbolic tracts.
See Example \ref{coszMasUno}.

\smallskip  

\emph{ii)
A two  point $\{a_1=a_2,\, \infty\}$
configuration.}

\noindent 
Since $a\doteq a_1=a_2\neq\infty$, then
$\Psi_X$ has one finite asymptotic value $a\in\CC_t$ 
of multiplicity 2, 
{\it i.e.} two logarithmic branch points over the 
same finite asymptotic value $a$. 

\noindent
By necessity,
$\Psi_X$ has at least another branch 
point over $b\in\CW\backslash\{a_1\}$, 
which can not 
be transcendental. 
Thus, $b$ must be a critical value.

\noindent
If $b =\infty$, then $X$ has an infinite number of zeros of multiplicity 
at least 2 and residue zero
accumulating at $\infty\in\CW_z$.

\noindent
If $b \neq \infty$, then $X$ also has an infinite number of poles 
accumulating at $\infty\in\CW_z$.

\noindent
Finally, the two neighbourhoods $U_{a_1}(\rho)$ and $U_{a_2}(\rho)$ 
(over the same asymptotic value $a=a_1=a_2$)
of the singularities of $\Psi^{-1}_X$ will be hyperbolic tracts.
Let
\begin{equation}\label{funcion-racional}
R(w)=
\frac{c_r w^r + c_{r-1} w^{r-1} + \cdots + c_1 w + c_0}
{b_s w^s + b_{s-1} w^{s-1} + \cdots +b_1 w + b_0},
\ \ \ {\tt r}= \max \{r,\, s \}, 
\end{equation}
\noindent 
be a rational function.
A straightforward calculation shows that either

\centerline{
$r=s \ \hbox{ and }  \
\dfrac{c_r}{b_r}=\dfrac{c_0}{b_0},
\ \hbox{ so } \  a = R(\infty)=R(0)= \dfrac{c_0}{b_0} \in\CC^*_t,$

}

\noindent
or

\centerline{
$s > r 
\ \hbox{ and } \
a = R(\infty)=R(0)=0, 
\text{ in particular }c_0=0 \text{ in }
\eqref{funcion-racional}.$
}

\noindent
See Example \ref{cscz}.

\smallskip

\emph{
iii)
A two 
point $\{a_1,\, a_2=\infty\}$ 
\, or \,
$\{a_1= \infty, \, a_2 \}$, configuration.}

\noindent
The vector field $X$ will not have any zeros. 
If $\mathcal{C}_{R}\backslash\{0,\infty\} \neq \varnothing$, 
then $X$ has an infinite number of poles accumulating at $\infty\in\CW_z$. 
One of the neighbourhoods of the singularities of $\Psi^{-1}_X$ will
be a hyperbolic tract and the other will be an elliptic tract.
In particular, Equation \eqref{funcion-racional} requires
\begin{equation}\label{caso3}
s<r
\ \hbox{ and } \
a_1=R(0)=\dfrac{c_0}{b_0}\in\CC_t, 
\ a_2=R(\infty)=\infty.
\end{equation}
\noindent
The other option is given by considering the rational function
$\widehat{R}(w)=R(1/w)$ with $R$ as in \eqref{caso3}, 
so $a_1=\widehat{R}(0)=\infty$ and $a_2=\widehat{R}(\infty)\in\CC_t$.
See Example \ref{cscz}.
\smallskip

\emph{ iv)
A one  point $\{a_1=a_2=\infty\}$
configuration.}

\noindent 
Note that, $X$ will have no zeros,
assertion (5) of the 
Theorem  \ref{familias-P-r} does not occur. 
If $\mathcal{C}_{R}\backslash\{0,\infty\} \neq \varnothing$, 
then $X$ has an infinite number of poles accumulating at $\infty\in\CW_z$. 
The two neighbourhoods $U_{a_1}(\rho)$ and $U_{a_2}(\rho)$ 
of the singularities of $\Psi^{-1}_X$ will be elliptic tracts.
In this case

\centerline{
$s<r
\ \hbox{ and } \
R(0)=R(\infty)=\infty\in\CW_t,
\text{ in particular }b_0=0 \text{ in }
\eqref{funcion-racional}.$
}

\noindent
See Example \ref{no-finite-asymptotic-values}.
\end{proof}

\begin{example}[Two 
logarithmic singularities over finite asymptotic values]
\label{coszMasUno}
The vector field 

\centerline{$X(z) = -i \big(\cos(z)+1\big)\del{}{z}=
-2i \cos^2 (\frac{z}{2}) \del{}{z}$ }

\noindent
is such that 

\centerline{$\Psi_X(z)=i\tan(\frac{z}{2})
= \frac{\e^{iz}-1}{\e^{iz}+1}
$,}

\noindent 
so it falls under the hypothesis of 
Theorem \ref{familias-P-r}, 
Case 6.i. 
Thus, $R(w)=(w-1)/(w+1)$ 
and $-1,1\in\CC_t$ are the finite
asymptotic values of $\Psi_X$.
There are two logarithmic singularities of $\Psi^{-1}_X$
over $-1$, $1\in\CC_t$,
whose neighbourhoods are hyperbolic tracts.
In this case, 
$X$ has an infinite number of double zeros
and no poles.
See Figure \ref{album-afin-AAP}.a.
\end{example}

\begin{example}[]
\label{cscz}
1. The pair 

\centerline{$
X(z)=-2 i \dfrac{\sin^2(z)}{\cos(z)} \del{}{z}$
\ and \ 
$ \Psi_X(z) =\dfrac{1}{2 i}\dfrac{1}{\sin(z)} $ }

\noindent 
falls under the hypothesis of 
Theorem \ref{familias-P-r}, Case 6.ii. 

\noindent 2.
Let $P\in\CC[z]$ be a non constant polynomial, 
the pair

\centerline{
$ X(z)= \dfrac{1}{\e^{z}P^\prime ( \e^{z} )} \del{}{z}$
\ and \
$\Psi_X(z)=P(\e^{z}) $ }

\noindent
falls under the hypothesis of Theorem \ref{familias-P-r}, Case 6.iii.

\noindent In both cases, 
details are left to the interested reader.
\end{example}

\begin{example}[Two 
logarithmic singularities over $\infty$]
\label{no-finite-asymptotic-values}
The vector field 

\centerline{$X(z) = \sec(z)\del{}{z}$ }

\noindent
is such that 

\centerline{$\Psi_X(z)=\sin(z)=
\dfrac{\e^{iz}-\e^{-iz}}{2i}
$,} 

\noindent 
so it falls under the hypothesis of 
Theorem \ref{familias-P-r}, Case 6.iv.
Since $R(w)=(w-w^{-1})/2i$ takes 
$0,\infty\mapsto\infty$, 
thus $\infty\in\CW_t$ is an
asymptotic value of multiplicity 2 and
$\Psi_X$ has no finite asymptotic values.
There are two logarithmic singularities of $\Psi^{-1}_X$
over $\infty\in\CW_t$,
whose neighbourhoods are elliptic tracts. 
Since 
$\int_{\pi/2 }^{(3/2) \pi} \cos(\zeta) d\zeta$
is finite, 
the incomplete trajectories 
$z_k ( {\tt t} ): (a, b) \subsetneq \RR \to \CC_z$ of $X$, 
having as images the 
real segments 
$(\pi/2 + k\pi,(3/2) \pi + k\pi ) \subset \RR$,
$k \in \ZZ$,
are located
at the poles $\{ (1/2) \pi + k \}$ of $X$.
See Figure \ref{album-afin-AAP}.b.
\end{example}


\subsection{Sporadic examples}
\label{sporadic-examples}
In this section we explore the limits of Theorem \ref{teo-logaritmicas-separada}
by considering examples of single and multivalued additively 
automorphic functions $\Psi_X$, emphasizing the geometrical 
richness of the vector field perspective. 
In particular, how the knowledge of the phase portrait of $\Re{X}$ helps in
determining and understanding the type of singularities of $\Psi_X^{-1}$.

\begin{example}[An infinite number of separate
singularities and no nonseparate singularities.]
\label{Ejemplo-ExpSin}
Let

\centerline{
$\Psi_X (z) = \e^{\sin(z)}$. 
}

\noindent
The associated vector field is

\centerline{
$X(z)=
\dfrac{1}{\Psi_X^\prime (z)} \del{}{z}=
\dfrac{\e^{-\sin(z)}}{\cos(z)} \del{}{z}$.
}

\noindent
See Figure \ref{figExpSin}.
The critical points of $\Psi_X$ are 
$\big\{ \frac{\pi}{2} (2k+1) \ \vert \ k\in\ZZ \big\}$ and its critical values are $\{ \e , \e^{-1} \}$.

\noindent
The asymptotic values of $\Psi_X$ are $0$, $\infty$. 
Clearly, $0$ and $\infty$ are isolated asymptotic values, so the transcendental singularities 
are logarithmic.

\noindent 
By examining the phase portrait\footnote{
Note that, since $\Psi_X(z)=\e^{\sin(z)}$ 
the phase portrait of 
$\Re{X}$ is the pullback via 
$\e^{w}$ of the phase portrait of
$\Re{\sec(z)\del{}{z}}$, 
see Example 
\ref{no-finite-asymptotic-values}. 
}
of $\Re{X}$, 
it is clear that 
there are an infinite number of logarithmic singularities.

\noindent
Let $k\in\ZZ$, the
asymptotic paths 
$\alpha_{a_{k\pm}}( {\tt t} )=(2k+1)\frac{\pi}{2} \pm i {\tt t} $, 
are associated to the asymptotic values 

\centerline{
$a_{k\pm}=\begin{cases}
  0_{k\pm}=0, \text{ for odd }k,\\
  \infty_{k\pm}=\infty, \text{ for even }k.
\end{cases}$
}

\noindent
Their neighbourhoods are

\centerline{
$U_{a_{k\pm}}(\rho)=\{ z\in\CC_z \ \vert\ \abs{\Re{z}- (2k+1)\frac{\pi}{2}} < \pi , 
\ \pm\Im{z}>R(\rho) \}$, 
}

\noindent
for appropriate $R(\rho)$. 
Note that the neighbourhoods $U_{0_{k\pm}}$ 
are hyperbolic tracts (coloured green in Figure \ref{figExpSin}), the 
plus sign indicating the ones on the top, the minus sign indicating the ones on the 
bottom. 
Similarly the neighbourhoods $U_{\infty_{k\pm}}$ are elliptic tracts 
(coloured blue in Figure \ref{figExpSin}).

\noindent 
Additionally note that
along the real axis $\Psi_X(z)$ does not converge as $z\to\pm\infty$, 
{\it i.e.} there is no asymptotic path (or value) along the real axis.

\noindent
There are no other singularities of $\Psi_X^{-1}$, even though
$\infty\in\CW_z$ is a nonisolated essential singularity of $X$.

\begin{figure}[h]
\begin{center}
\includegraphics[width=0.65\textwidth]{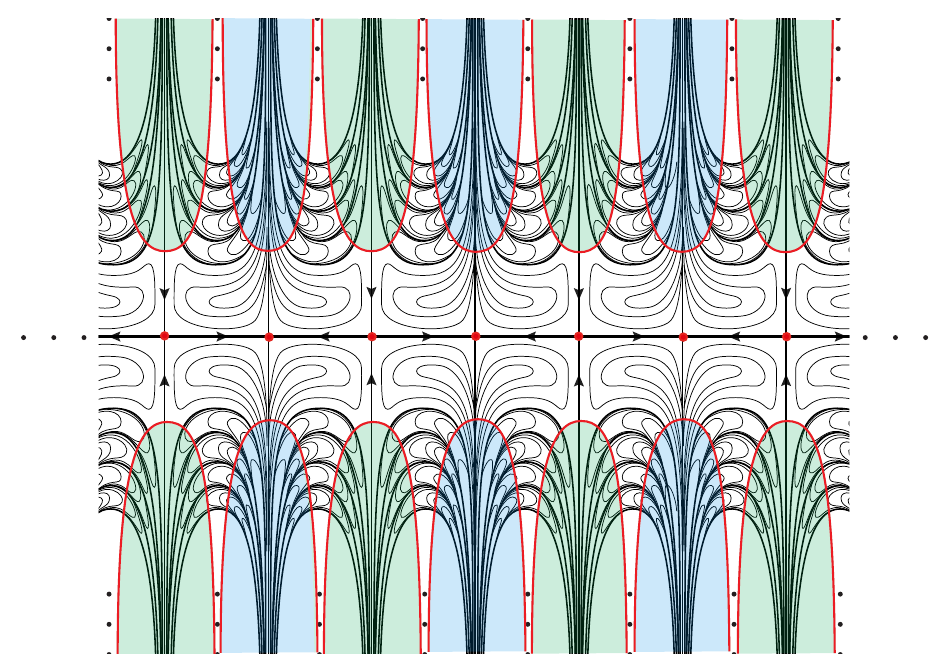}
\caption{
Example \ref{Ejemplo-ExpSin},
function $\Psi_X(z)= \e^{\sin(z)}$ 
and phase portrait of the corresponding vector field 
$X(z)=\sec(z) \e^{-\sin(z)} \del{}{z}$. 
There are infinite hyperbolic and 
elliptic tracts, coloured green and blue respectively.
}
\label{figExpSin}
\end{center}
\end{figure}
\end{example}


\begin{example}[Direct nonlogarithmic singularity of $\Psi^{-1}_X$]
\label{Ejemplo-ExpSin-z}
Consider the function

\centerline{
$\Psi_X (z) = \e^{\sin(z)-z}$,}

\noindent
studied in
\cite{Langley-2}.  
The associated vector field is

\centerline{ 
$X(z)=
\frac{1}{\Psi_X^\prime (z)}  \del{}{z}
= \dfrac{\e^{\sin(z)-z}}{\cos(z)-1} \del{}{z}$.} 

\noindent 
The critical points of $\Psi_X$  are 
$\{p_k\doteq 2\pi k  \ \vert \ k \in \ZZ \}$, 
with critical values 
$\{\widetilde{p}_k\doteq \e^{-2\pi k} \ \vert \ k \in \ZZ  \}$.
See Figure \ref{ExpSin-z}.
The asymptotic values of $\Psi_X$ are $0$ and $\infty$.
\emph{Note that they are nonisolated singular values}.
Since $\Psi_X$ is entire and these are omitted values, 
the corresponding transcendental singularities are direct.

\noindent
Once again, from the phase portrait of $\Re{X}$ there seems to be an infinite 
number of logarithmic singularities.
As in the previous example,
for each $k\in\ZZ$, 
the asymptotic paths 
$\alpha_{k\pm} ({\tt t} )=(2k+1)\frac{\pi}{2} \pm i {\tt t} $, 
where ${\tt t} >0$,
are associated to the asymptotic values 

\centerline{
$a_{k\pm}=\begin{cases}
0_{k\pm}=0, \ \ \ \text{ for odd }k,
\\
\infty_{k\pm}=\infty, \text{ for even }k.
\end{cases}$
}

\noindent
Their neighbourhoods are

\centerline{
$U_{a_{k\pm}}(\rho)=
\left\{ 
z\in\CC_z \ \vert\ \abs{\Re{z} - (2k+1)\frac{\pi}{2}} < \pi , \ \pm\Im{z}>R(\rho) \right\}$, 
}

\noindent
for appropriate $R(\rho)$.
As before, the neighbourhoods $U_{0_{k\pm}}(\rho)$ are coloured green 
and the neighbourhoods $U_{\infty_{k\pm}}(\rho)$ are coloured blue 
in Figure \ref{ExpSin-z}.

\begin{figure}[htbp]
\begin{center}
\includegraphics[width=0.65\textwidth]{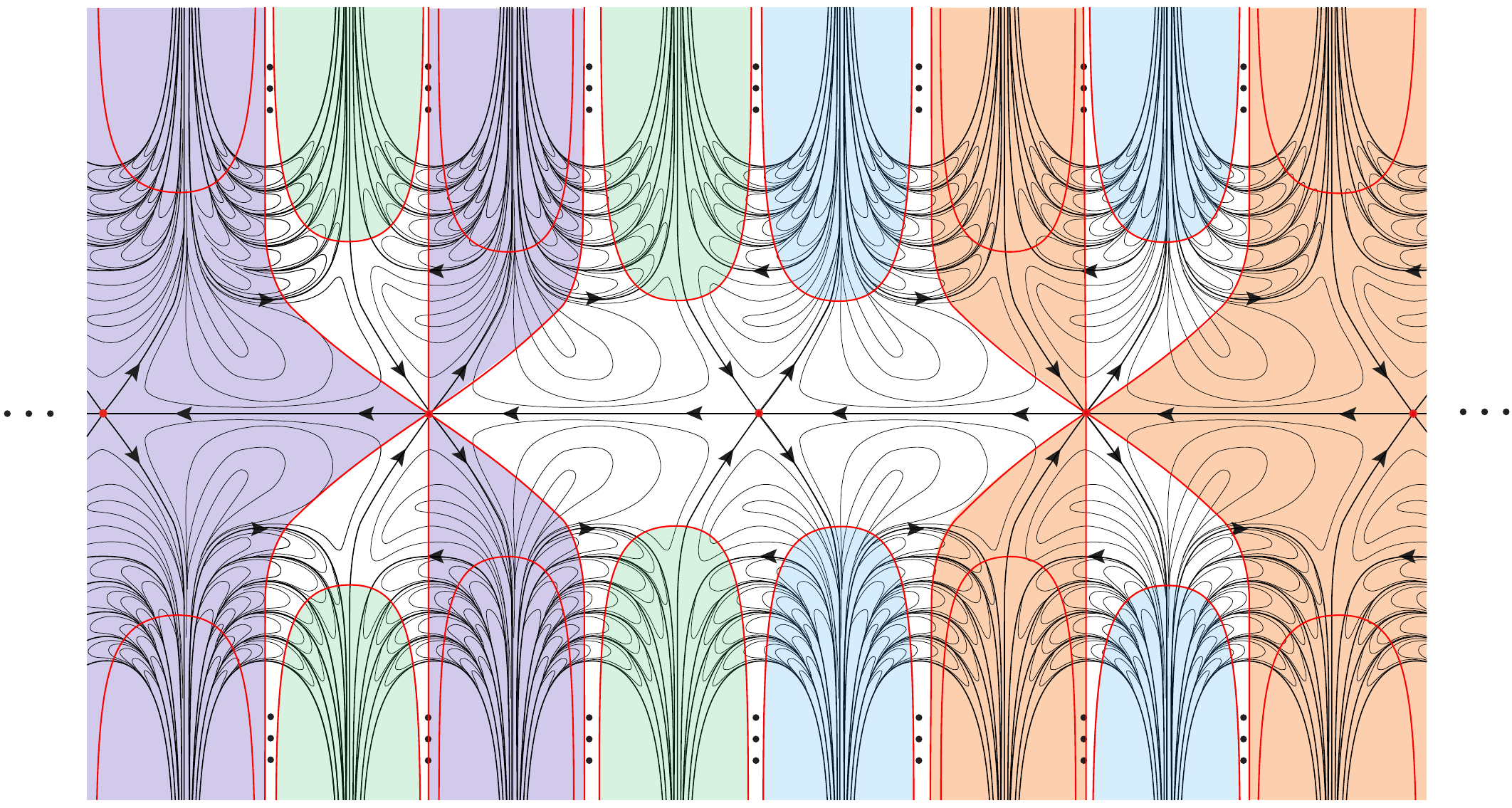}
\caption{
Example \ref{Ejemplo-ExpSin-z},
function $\Psi_X(z)= \e^{\sin(z)-z}$ and 
phase portrait of the corresponding vector field 
$X(z) = \frac{\e^{\sin(z)-z}}{\cos(z)-1} \del{}{z}$.
The neighbourhoods 
$U_{\infty,-}(\rho)$ and $U_{0,+}(\rho)$
of the nonseparate singularities are coloured purple and orange, 
respectively.
}
\label{ExpSin-z}
\end{center}
\end{figure}

\noindent 
Since these neighbourhoods are mutually disjoint, 
the singularities are separate, 
so by Theorem \ref{teo-logaritmicas-separada}, 
each $U_{a_{k\pm}}$ is logarithmic. 

\noindent 
However, in this example there are two more singularities of $\Psi_X^{-1}$:

\noindent
$\bigcdot$ 
The asymptotic value $\infty$, 
arising from the asymptotic path $\alpha_{-}( {\tt t} )$, 
having image $\RR^-$,
gives rise to a direct transcendental 
singularity of $\Psi^{-1}_X$, say 
$U_{\infty,-}$.  
The corresponding neighbourhoods 
$U_{\infty,-}(\rho)$ contain the regions

\centerline{
$\{ z\in\CC_z\ \vert\ -\frac{\pi}{2}
<\arg(\frac{1}{z})<\frac{\pi}{2},
\ \Re{z}<R(\rho) \}$,
}

\noindent
for suitable $R(\rho)$. 
These neighbourhoods are coloured purple in Figure 
\ref{ExpSin-z}.

\noindent 
$\bigcdot$ 
Similarly, the asymptotic value $0$ arising from the 
asymptotic path $\alpha_{+}(t)$, having image $\RR^+$,  
gives rise to a direct transcendental 
singularity of $\Psi^{-1}_X$, say 
$U_{0,+}$.
The corresponding neighbourhoods 
$U_{0,+}(\rho)$ contain the regions

\centerline{
$\{ z\in\CC_z\ \vert\ -\frac{\pi}{2}<\arg(\frac{1}{z})
<\frac{\pi}{2},
\ \Re{z}>R(\rho) \}$,
}

\noindent
for appropriate $R(\rho)$.
These neighbourhoods are coloured orange in Figure \ref{ExpSin-z}.

\noindent
The above implies that, 
for any given $\rho>0$, each neighbourhood 
$U_{\infty,-}(\rho)$ and $U_{0,+}(\rho)$
contains 
an infinite number of neighbourhoods
$U_{a_{k\pm}}(\rho)$;
thus are nonseparate. 

\noindent
By Theorem \ref{teo-logaritmicas-separada},
$U_{\infty,-}$ and $U_{0,+}$, are
direct nonlogarithmic singularities.
\end{example}


\begin{example}[Indirect transcendental singularity of $\Psi^{-1}_X$]
\label{Ejemplo-Sinc}
Let 

\centerline{
$\Psi_X (z) = \sin(z)/z$.
}

\noindent
The associated vector field is 

\centerline{
$X(z)=\frac{z^2}{z \cos (z)-\sin (z)}\del{}{z}$.
}

\noindent
The critical points of $\Psi_X$ are the unbounded set $\{ z\in\CC_z\ \vert\ z\cos(z)-\sin(z)=0 \}$, 
with critical values lying on the real axis and converging to $0$ as the critical points 
approach $\pm\infty$. 

\noindent
The asymptotic values of $\Psi_X$ are $0$ and $\infty$. 

\noindent
Since $\infty$ is an isolated 
asymptotic value, the singularities of $\Psi^{-1}_X$ over $\infty$ are logarithmic. 
In fact, 
there are two, say $U_{\infty\pm}$, arising from 
the asymptotic paths $\alpha_{\infty\pm}$ having 
images $i \RR^+$ and $i \RR^-$.
The corresponding (disjoint) neighbourhoods are

\centerline{
$U_{\infty\pm}(\rho)=\{z\in\CC_{z}\ \vert\ \pm\Im{z} > R(\rho) \}$, 
}

\noindent
for appropriate $R(\rho)>0$.

\noindent
The neighbourhoods $U_{\infty\pm}(\rho)$
are elliptic tracts.

\noindent 
On the other hand, since $\Psi_X$ assumes the value $0$ infinitely often along the 
real axis, the transcendental singularities of $\Psi^{-1}_X$ over $0$ are indirect.
In fact, there are two: $U_{0\pm}$ arising from
the asymptotic paths $\alpha_{0\pm}( {\tt t} )$ having images $\RR^+$ and $\RR^-$.
\end{example}

\begin{remark}[The topology of 
the vector field $\Re{X}$
does not determine the nature of 
the ideal points]
\label{remark-topologicamente-equivalente}
The previous example, shows that the vector fields 

\centerline{
$X_{1}(z)=\frac{z^2}{z \cos (z)-\sin (z)}\del{}{z}$
\  and  \ 
$X_{2}(z)=\sec(z)\del{}{z}$}

\noindent 
have the same topological phase portraits, 
see Figure \ref{album-afin-AAP}.b
and 
\cite{AlvarezMucino2} \S11 for accurate definitions. 
From the point of view of the
singularities of $\Psi^{-1}_X$, 
they have important differences:
the vector field 
$X_{1}$ has an indirect transcendental singularity, but $X_{2}$ does not.
Furthermore,
$\Psi_{X_1}$ has 4 asymptotic values 
$\{ 0,0,\infty,\infty \}$, 
but $\Psi_{X_2}$ only two
$\{\infty,\infty\}$.
\end{remark}


\begin{example}[Direct nonlogarithmic singularity without critical points]
\label{Ejemplo-IntExpExp}
Let 

\centerline{
$\Psi_X (z) = 
{\displaystyle \int_{0}^{z} } \e^{-\e^\zeta} d\zeta$,}

\noindent 
which is studied in
\cite{Herring}, \cite{Sixsmith}. 
The associated vector field is 

\centerline{
$X(z)=\e^{\e^{z}} \del{}{z}$.
}

\noindent
It is clear that the critical point set of $\Psi_X$ is empty. 
Let 

\centerline{
$a_0=\lim\limits_{\RR^{+}
\ni {\tt t}\to\infty}\Psi_{X}( {\tt t} )
=-
{\displaystyle \int_{-1}^{\infty } }
\dfrac{\e^{-t}}{t} \, dt \approx 0.219384$.}

\noindent 
There are an infinite number of finite asymptotic values of $\Psi_X$ given by

\centerline{
$\{a_k \doteq a_0 + i 2k\pi 
\ \vert \ k \in\ZZ 
\} \subset\CC_t$, }

\noindent
with asymptotic paths 

\centerline{
$\{ \alpha_{k} ({\tt t})= {\tt t} + i 2k\pi 
\ \vert \ 
k \in\ZZ \}$,  
\ \ \
for ${\tt t} \geq 0$,
}
\noindent
according to \cite{Herring} p.~271.

\noindent  
Since the finite asymptotic values are isolated, 
the corresponding transcendental singularities of $\Psi^{-1}_X$ are logarithmic
and their neighbourhoods $U_{a_k}(\rho)$ are hyperbolic tracts over $a_k$.

\noindent 
On the other hand, the asymptotic paths 

\centerline{
$\{ \beta_{k} ({\tt t} )= {\tt t} + i (2k+1)\pi 
\ \vert \ 
k\in\ZZ \}$, 
\ \ \
for ${\tt t} \geq 0$ 
}

\noindent
have the asymptotic value $\infty\in\CW_t$,
in accordance with \cite{Herring}, statement (8).
Note that $\infty$ is a nonisolated asymptotic value. 
The asymptotic paths 
$\{\beta_k\}$ correspond to 
neighbourhoods $U_{\infty,k}(\rho)$ 
that, 
for $\rho>0$ sufficiently small, are
disjoint from the neighbourhoods of other singularities 
of $\Psi^{-1}_X$; thus these transcendental singularities are separate.
Hence, by Theorem \ref{teo-logaritmicas-separada}, 
they are also logarithmic singularities of $\Psi^{-1}_X$.

\noindent 
From statements (9) and (10) of \cite{Herring}, $\infty\in\CW_t$ is an asymptotic
value for a\-symptotic paths arriving to $\infty\in\CW_t$ in an angular sector of angle
$2\pi$ that avoids the positive real line. 
We shall denote by $U_{\infty, -}$ the corresponding singularity.
For $\rho>0$, each neighbourhood $U_{\infty,-}(\rho)$ 
contains an infinite number of neighbourhoods $U_{a_k}(\rho)$ and $U_{\infty,k}(\rho)$, 
hence the singularity $U_{\infty,-}$ is nonseparate, thus direct nonlogarithmic.
See Figure \ref{figExpExp}.

\begin{figure}[htbp]
\begin{center}
\includegraphics[width=0.60\textwidth]{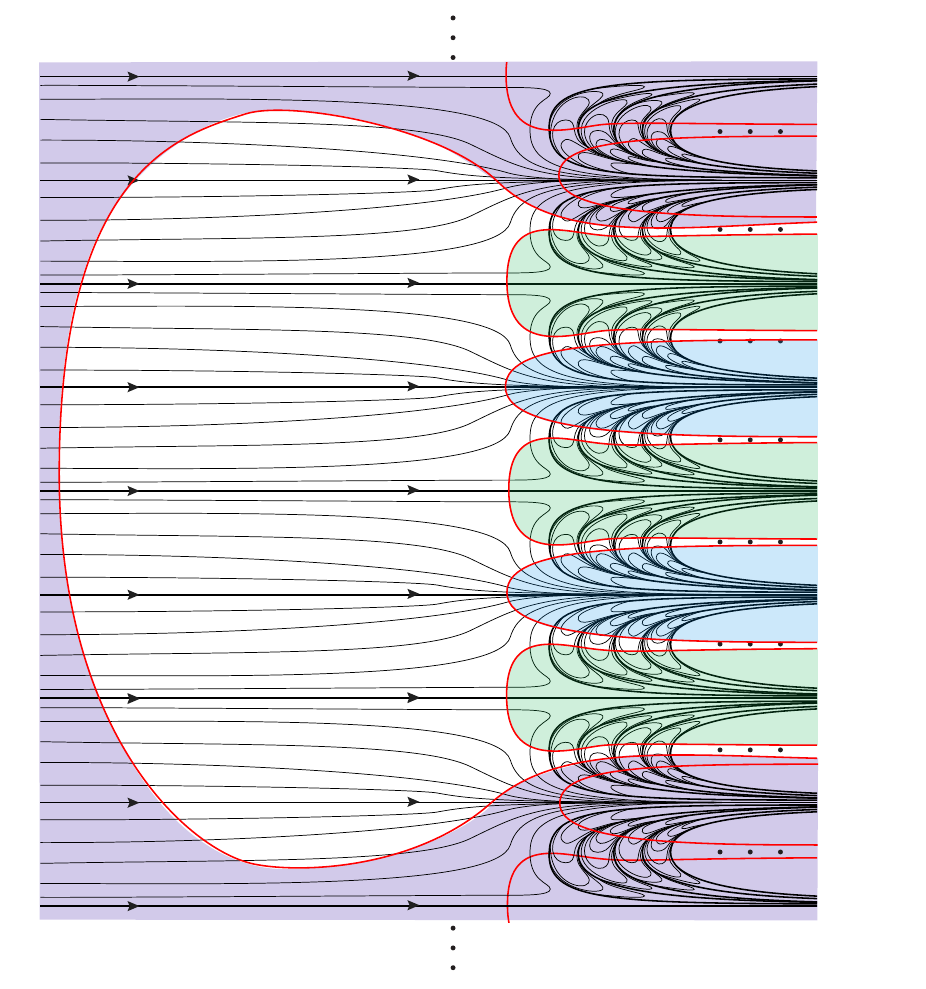}
\caption{
Example \ref{Ejemplo-IntExpExp},
function $\Psi_X(z)=\int_{0}^{z} \e^{-\e^\zeta} d\zeta$ and 
phase portrait of the corresponding vector field $X(z)=\e^{\e^z} \del{}{z}$.
The neighbourhood $U_{\infty, -}(\rho)$
of the nonseparate singularity is coloured purple. 
}
\label{figExpExp}
\end{center}
\end{figure}
\end{example}

\begin{example}[Direct nonlogarithmic singularity of $\Psi^{-1}_X$,
with an accumulation of critical values]
\label{Ejemplo-ExpSinExp}
Let 

\centerline{
$\Psi_X (z) = \e^z \sin(\e^z)$.
}

\noindent
The associated vector field is

\centerline{
$X(z)=\frac{1}{\e^z \sin \left(\e^z\right)+\e^{2 z} \cos \left(\e^z\right)} \del{}{z}$.
}

\noindent
The critical points of $\Psi_X$ are the unbounded set

\centerline{
$\left\{ z\in\CC_z\ \vert\ 
\e^z \big( \sin \left(\e^z\right)+\e^z \cos \left(\e^z\right) \big) = 0 \right\}$,
}

\noindent
which lie along the real lines of height $i k \pi$, $k\in\ZZ$ and whose real part is approximately 
given by $\{ \log((2j+1)\frac{\pi}{2}) \ \vert\ j\in\NN \}$. 
Thus in particular, 
the critical points lie to the right of $\Re{z}=\log(3\pi/2)\approx 1.55019$.
The corresponding critical values lie on the real axis and converge to $-\infty$ as the critical
points approach $\infty$.

\noindent
The asymptotic values of $\Psi_X$ are $0, \infty\in\CW_t$.

\noindent
Since $a=0$ is an isolated asymptotic value, there is a (direct) logarithmic singularity 
$U_0$ over it.
Its neighbourhoods $U_0 (\rho)$ are contained in half planes 

\centerline{
$U_0 (\rho) \subset \{z\in\CC_z \ \vert\ \Re{z}<-R(\rho) \}$,
}

\noindent
for appropriate $R(\rho)>0$.
The neighbourhoods $U_0 (\rho)$ are hyperbolic tracts over $0$ 
and are coloured green in Figure \ref{ExpSinExp}.
\begin{figure}[htbp]
\begin{center}
\includegraphics[width=0.65\textwidth]{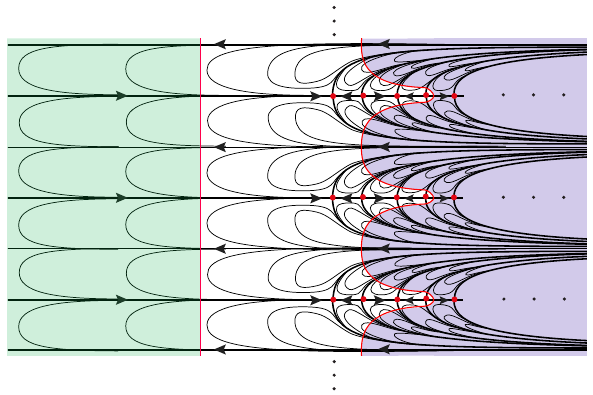}
\caption{
Example \ref{Ejemplo-ExpSinExp},
function $\Psi_X(z)= \e^z \sin(\e^z)$ and 
phase portrait of the corresponding vector field
$X(z)= \big(\e^z \sin \left(\e^z\right)+\e^{2 z} \cos \left(\e^z\right) \big)^{-1} \del{}{z}$.
The neighbourhood $U_{\infty}(\rho)$
of the nonseparate singularity is coloured purple. 
}
\label{ExpSinExp}
\end{center}
\end{figure}

\noindent
On the other hand, 
since $\Psi_X$ is entire, $\infty$ is an omitted value, 
and hence the singularity 
$U_{\infty}$ associated to the
asymptotic value $\infty$ is direct.

\noindent
Note that any neighbourhood $U_{\infty}(\rho)$, coloured 
purple in Figure \ref{ExpSinExp}, 
of this direct singularity contains 
a half plane $\{ \Re{z}>R(\rho) \}$, for
appropriate $R(\rho)$, and thus an infinite
number of critical points (algebraic singularities of $\Psi^{-1}_X$).
Therefore 
$U_{\infty}$ is nonseparate, {\it i.e.} it is a direct nonlogarithmic singularity over $\infty$.

\noindent
It is to be noted that this $\Psi_X$ only has two transcendental singularities of $\Psi^{-1}_X$:
a logarithmic singularity over $0$ and a direct nonlogarithmic singularity over $\infty$.
\end{example}

\begin{example}\label{ejemplo-sin}
We consider the vector field

\centerline{
$X(z)= i \sin(z)\del{}{z}.$
}

\noindent 
In Figure \ref{album-afin-AAP}.d.
is a sketch of  
the phase portrait of $\Re{X}$.
Since $\mathcal{S}_{R}=\{k\pi \ \vert\ k\in\ZZ \}$,
clearly $\Psi_X$ is multivalued
additively automorphic.
Let $\gamma_k({\tt t})=k\pi+{\tt t}\pi$ for ${\tt t}\in[0,1]$,
and $\Gamma=\{\gamma_k\}_{k\in\ZZ\backslash\{0\}}$ so
$\overline{\Gamma}=[-\infty,0]\cup [\pi,\infty]\subset\CW_z$
is a closed arc of a circle containing $\infty$.
It follows that a fundamental region is

\centerline{
$\Lambda = 
( \CW_z\backslash\overline{\Gamma} )
\bigcup\limits_{k\in\ZZ\backslash\{0\}} \gamma_{k+}$,
}

\noindent  
as in \S\ref{construccion-region-fundamental}.
The restriction of $\Psi_X$ to $\Lambda$, for 
$z_{\tt o}=\pi/2$, 

\centerline{$
\Psi_{X,\,\Lambda}(z)
=i
{\displaystyle \int_{z_{\tt o}}^z} \csc(\zeta) d\zeta
=i\Big(\log\big(\sin(z/2)\big) - 
\log\big(\cos(z/2)\big)\Big)
: \Lambda \longrightarrow \CW_t
$}

\noindent 
is single--valued. 
The fundamental region is

\centerline{$\Omega= \left\{\big(z,\Psi_X(z)\big) \ \vert\ 
z\in\Lambda \right\} \subset \R_X$,}

\noindent
see Figure \ref{cosRaiz-superficie}.a  
for a sketch of $\Omega$. 

\begin{figure}[h]
\begin{center}
\includegraphics[width=0.75\textwidth]{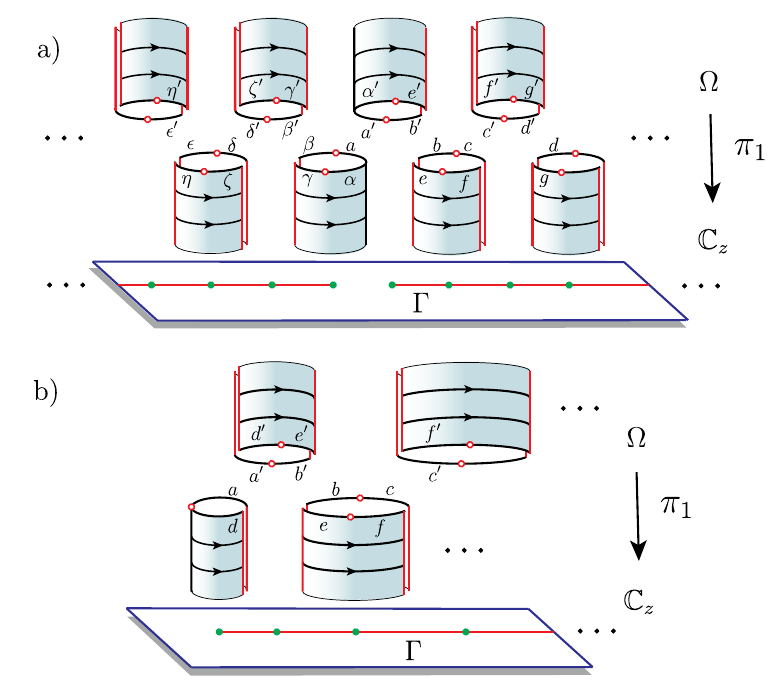}
\caption{
A sketch (using surgery) of the fundamental regions
$\Omega \subset \R_X$ of the vector fields: 
(a) $i \sin(z)\del{}{z}$ and
(b) $i \cos(\sqrt{z}) \del{}{z}$.
In both cases, the zeros are simple with
imaginary linear parts, hence they are isochronous
centers of $\Re{X}$ and determine half 
cylinders in the metric $(\CW_z, g_X)$ 
\emph{i.e.} of height $(0, \infty)$ or $(-\infty, 0)$.
Moreover, the upper and lower ends
of the cylinders, which are not identified
(coloured green),correspond to the zeros of the vector fields
(green points in $\CC_z$). 
The path $\Gamma$ (coloured red) is a
cut between the zeros of $X$, 
obtaining  flow boxes. 
The letters indicate the corresponding
identifications, that describe the connected
regions $\Omega$.
}
\label{cosRaiz-superficie}
\end{center}
\end{figure}

\noindent
One can observe 
a sequence of simple zeros accumulating
at $\infty\in\CW_z$:

\noindent 
$\bigcdot$ 
each simple zero of $X$, 
at $q_k=k\pi$, $k\in\ZZ$,
has asymptotic value $\infty$, 
and 

\noindent 
$\bigcdot$ 
the essential singularity at $\infty \in \CW_z$, 
has two finite asymptotic values 
$a_{\pm}=\mp \pi /2 $;
arising from asymptotics paths, 
$\alpha_\pm({\tt t})\subset\Lambda$, 
that start at $z_{\tt o}$ and 
arrive at $\infty\in\CW_z$ inside
of the upper or lower half planes $\HH_{+}$ or 
$\HH_{-}$ respectively.

\noindent
By using Diagram \ref{diagrama-dominio-fundamental}
and 
Definition \ref{definicion-singularidades-en-Lambda},
the singularities of 
$\Psi_{X,\,\Lambda}^{-1}$
are:

\noindent
$\bigcdot$ the  
$\star$--transcendental 
singularities 
$\{U_{\infty,k}\}_{k\in\ZZ}$ corresponding to the 
zeros $\{q_k\}$ of $X$, and

\noindent
$\bigcdot$ the two 
essential transcendental  
singularities $U_{a_\pm}$.

\noindent
Since the finite asymptotic values 
$a_\pm$ are isolated, 
then by Theorem \ref{teo-logaritmicas-separada},
the essential transcendental singularities $U_{a_\pm}$ 
are logarithmic transcendental singularities of
$\Psi_{X,\,\Lambda} ^{-1}$.
Moreover, since the asymptotic values $a_\pm$ 
are finite, 
the neighbourhoods $U_{a_\pm}(\rho)$ are 
hyperbolic tracts,
which
can be clearly observed on the 
phase portrait.

From the perspective of the universal cover
$\mathfrak{M}$ of
$\CW_z\backslash\{0,\infty\}$, 
Corollary \ref{cubierta-universal-singularidades}.3
applies.

\end{example}

\subsubsection{A family of vector fields with only one tract}
\label{un-solo-tract}
Let us now consider the family 

\centerline{
$\left\{ 
X(z)= \lambda z^{\ell} \cos^{r} (\sqrt{z}) \del{}{z} 
\ \big\vert \ 
r\in \ZZ^*, \ \hbox{ even } \ell \in \ZZ, 
\, \lambda \in \CC^* 
\right\}$
}

\noindent 
of singular complex analytic vector fields on 
$\CW_z$ with
an essential singularity at $\infty$.
In particular,
for $\ell \geq 0$ $X$ is holomorphic at $0$.
Using the complex quotient

\centerline{
$ \pi: \CW_z \longrightarrow 
\big( \CW_z / \pm id \big) 
= \CW_\mathfrak{z}, 
\ \ \
z \longmapsto[\pm z],$} 

\noindent 
here $[\ ]$ denotes the equivalence class,
it follows that $\cos(\mathfrak{z})$ is well defined.
The function
$\cos(\sqrt{z})$ is entire and nonvanishing at $0$. 

\begin{example}
\label{cos-raiz}
Case $\ell=0, \ r=1$.
\emph{A transcendental
singularity of $\Psi_X^{-1}$
with an accumulation of simple zeros.}
The vector field

\centerline{
$
X(z)= i \cos (\sqrt{z}) \del{}{z}
$} 

\noindent 
has a unidirectional sequence of simple
zeros (isochronous centers)
$\mathcal{Z}_{R}=\{ (k\pi + \pi/2 )^2 \ \vert \ k \in \NN\}
\subset \RR^+$
that accumulates to
the essential singularity at $\infty$.
The corresponding

\centerline{
$
\Psi_X(z) = -i 
{\displaystyle \int}_{\hspace{-4.5pt}0}^z \sec(\sqrt{\zeta}) d\zeta
=2 \left[
2 i \mathfrak{C} 
+ \text{Li}_2\left(-i e^{i \sqrt{z}}\right) 
- \text{Li}_2\left(i e^{i \sqrt{z}}\right)
- 2 \sqrt{z} \tan^{-1} \left(e^{i \sqrt{z}} \right)
\right],
$}

\noindent 
is a multivalued additively automorphic 
singular complex analytic function, 
where
$\mathfrak{C} \doteq \sum_{k=0}^{\infty } \frac{(-1)^k}{(2 k+1)^2} 
\simeq 0.91597$ 
is Catalan's constant and
$\text{Li}_2(z)
=-\int_0^z \frac{\log (1-u)}{u} du$
is the dilogarithm function,
see for instance \cite{Lewin}.
The phase space of $\Re{X}$ is sketched in
Figure \ref{album-afin-AAP}.e. 
Note that $\infty\in\CW_z$ is an accumulation (from the right)
of the simple zeros $\mathcal{Z}_R$ of $X$.
We construct a fundamental 
domain $\Lambda$ for $\Psi_X$, Subsection
\ref{construccion-region-fundamental}.
Let $\{ \gamma_k\}$ be real segments between two 
consecutive zeros of $X$, 
$\overline{\Gamma}$ is the closure of these segments, 
a fundamental region  

\centerline{
$\Lambda =
\big( \CW_z \backslash  \overline{\Gamma}  \big)
\cup \gamma_{k+}, 
\ \ \
\hbox{ where }
\overline{\Gamma} = [\pi/2, + \infty].$
}

\noindent 
The periods of the trajectories of 
$\Re{ i\cos( \sqrt{z}) \del{}{z}}$ are

\centerline{
$T_k= 2 \pi (r_k) =
(-1)^{k} i \pi^2  (4k+2)$
}

\noindent
where $r_k = 
Res(\omega_X, q_k) = 
\frac{1}{2\pi i} \int_\vartheta \omega_X$,
and as usual $\vartheta$ encloses the respective zero.  
Each zero $q_k$  
of $X$ determines a basin of periodic trajectories
of $\Re{X}$, 
say $\mathscr{C}_k$, 
which provided with the metric $g_X$
is isometric to a semi infinite flat cylinder
of perimeter $T_k$. 
Thus, the perimeters of the cylinders
tend towards $\infty$, as the zeros approach
the essential singularity $\infty$.
See Figure \ref{cosRaiz-superficie}.b.

\noindent
On the other hand, since $X(z)$ is entire, then 
$\Psi_X$ does not have any finite critical values.
Moreover, since $\Psi_{X,\,\Lambda}$ is single--valued 
on $\Lambda$, and there 
is only one homotopy class of paths approaching $\infty\in\CW_z$,
there is one finite asymptotic value
for $\Psi_{X,\,\Lambda}$ arising from the asymptotic path $\alpha({\tt t})=-{\tt t}$ for ${\tt t}\in\RR^+$,
that is

\centerline{$
a=
\lim\limits_{{\tt t}\to\infty} 
\Psi_{X,\,\Lambda}\big( \alpha({\tt t}) \big)
= 4 i \mathfrak{C} \simeq i 3.66388 \in\CC_t.%
$}

\noindent 
Thus the singular values are $\{ a,\infty\}$.
Once again, by Theorem \ref{teo-logaritmicas-separada},
the singularity $U_a$ associated to the asymptotic value $a$ is logarithmic.
Since the asymptotic value $a\in\CC_t$ is finite, 
$U_a(\rho)$ are hyperbolic tracts,
coloured green in Figure \ref{album-afin-AAP}.e.

\noindent
Thus, 
by using Diagram \ref{diagrama-dominio-fundamental}
and 
Definition \ref{definicion-singularidades-en-Lambda}, 
all the singularities of 
$\Psi_{X,\,\Lambda}^{-1}$ are:

\noindent
$\bigcdot$ the $\star$--transcendental 
singularities $\{ U_{\infty,k} \}_{k\in\NN} $ corresponding to the
zeros $q_k$ of $X$, and

\noindent
$\bigcdot$ the logarithmic singularity
$U_a$ 
over the finite asymptotic value
$a\in\CC_t$ as above.

From the perspective of the universal cover
$\mathfrak{M}$ of
$\CW_z\backslash\{0,\infty\}$, 
Corollary \ref{cubierta-universal-singularidades}.1--3
applies.
\end{example}

\begin{example}
\label{una-sola-singularidad-de-la-inversa}
Case $\ell=0, \ r=-1$.
\emph{A direct nonlogarithmic 
singularity of $\Psi_X^{-1}$
over $\infty$ with an accumulation of finite 
critical values.}
The vector field

\centerline{$
X(z)= \frac{1}{\cos (\sqrt{z})} \del{}{z}
$}

\noindent 
has an unidirectional sequence of poles 
$\mathcal{P}=\{ p_k\doteq (k \pi + \pi/2 )^2 \ \vert  \ k \in \NN\cup\{0\} \}
\subset \RR^+$
that accumulates to
the essential singularity al $\infty$.
The singular set of $X$ is 
$\mathcal{S}_{X}=
\overline{\mathcal{P}}=\mathcal{P}\cup\{\infty\}$.
The single--valued additively automorphic 
entire function

\centerline{
$
\Psi_X(z)= \int_0^z 
\cos(\sqrt{\zeta})
d\zeta
=
2 \big( \sqrt{z} \sin \left(\sqrt{z}\right)+ \cos \left(\sqrt{z}\right) -1 \big)
$,
}

\noindent 
has critical points at $\mathcal{P}$
with critical values 
$\{ \widetilde{p}_k\doteq
(-1)^{k} (2 k + 1)\pi -2 \ \vert  \ k \in \NN\cup\{0\} \}
\subset\CC_t$,
an alternating sequence centered about $0\in\CC_t$ with 
accumulation point $\infty\in\CW_t$.
A model for the Riemann surface $\R_X$ can be constructed by surgery as follows, 
see Figure \ref{raiz-coseno-trayectorias}:

\begin{figure}[h]
\begin{center}
\includegraphics[width=0.85\textwidth]{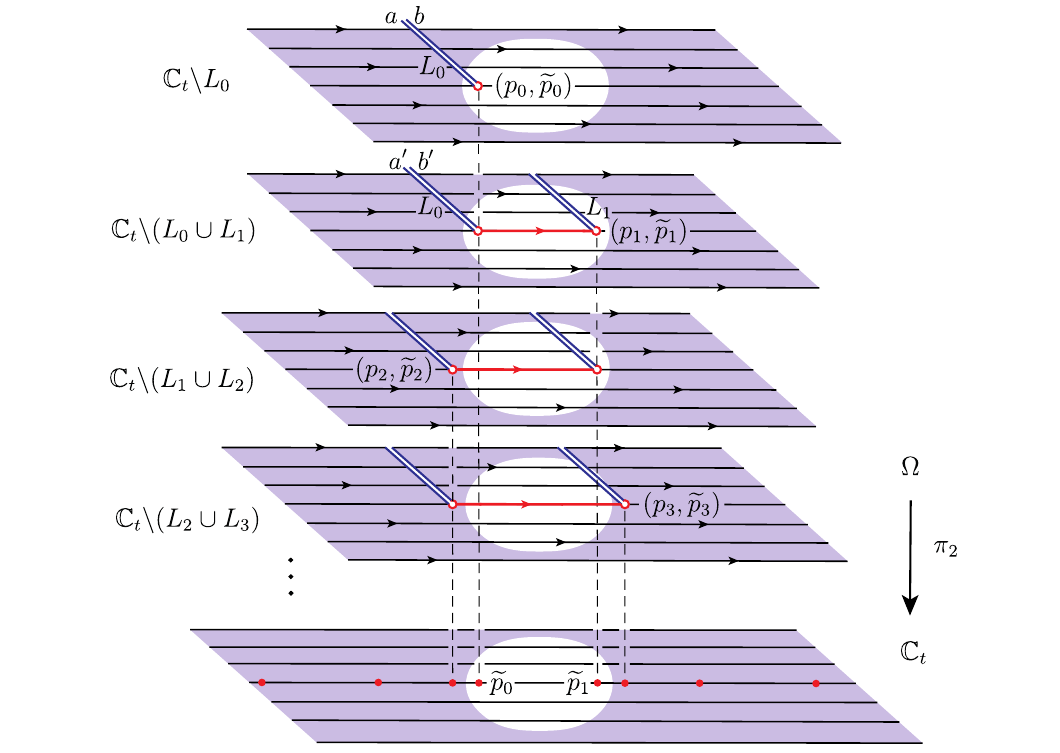}
\caption{
A sketch (using surgery) of the fundamental region
$\Omega \subset \R_X$ of the vector field 
$X(z)= \big(1/\cos (\sqrt{z})\big) \del{}{z}$.
The segments of $\pi_1^{-1}(\Gamma)$ are in red in $\Omega$.
Thus, there are infinite copies of $\CC_t$ in $\Omega$ with
auxiliary cross cuts $L_k $,
whose boundaries
are identified to produce branch points $(p_k, \widetilde{p}_k )$ of index two. The $L_k$
are orthogonal to the trajectories of $\Re{X}$. 
The critical values $\widetilde{p}_k$ accumulate to $\infty \in \CW_t$.
A neighbourhood $D(\infty, \rho) \subset \CC_t$ 
is coloured purple. 
The preimage $\pi_2^{-1} \big( D(\infty, \rho) \big)
\subset \Omega$ is one connected component, coloured purple,
and contains an infinite number of branch points
$(p_k, \widetilde{p}_k)$ corresponding to poles of $X$, 
represented as red dots.
}
\label{raiz-coseno-trayectorias}
\end{center}
\end{figure}  

\noindent
$\bigcdot$ 
As a first step
consider copies of $\CC_t$, say $\CC_t\backslash L_0$
and $\CC_t\backslash (L_0 \cup L_1)$,
where $L_0$ is a branch cut starting at the branch point 
$(p_0,\widetilde{p}_0)$ 
and $L_1$ is a branch cut starting at the branch point 
$(p_1,\widetilde{p}_1)$, both 
of ramification index 2.
As usual, the boundaries of the $L_0$'s 
are identified:
side $a$  with  side $b'$ and
side $b$  with  side $a'$.

\noindent
$\bigcdot$ 
Secondly, for each $k\in\NN$ consider $\CC_t\backslash(L_k\cup L_{k+1})$, 
here $L_{k}$ is a branch cut starting at the branch point
$(p_k,\widetilde{p}_k)$ of ramification index 2,
and $L_{k+1}$ is a branch cut starting at the branch point
$(p_{k+1},\widetilde{p}_{k+1})$ of ramification index 2.
The above copies of 
$\CC_t\backslash(L_{k}\cup L_{k+1})$ 
are glued along the corresponding branch cuts $L_k$, $k\in\NN\cup\{0\} $.

\noindent
Moreover, note that 
$\infty\in\CW_t$ is an asymptotic value with asymptotic path $\alpha_{\infty}({\tt t})$
in the angular sector $\{0<\arg{z}<2\pi\}$.
\begin{lemma} 
Let $
\Psi_X(z)= \int_0^z \cos(\sqrt{\zeta}) d\zeta$, 
the singularity $U_\infty$ of $\Psi_X^{-1}$ is 

\begin{enumerate}[label=\arabic*),leftmargin=*]
\item
non separate,

\item
direct and non logarithmic. 
\end{enumerate}
\end{lemma}
\begin{proof}
To prove (1), consider $\overline{D(0,R)}\subset\CW_t$, for $R>\abs{\widetilde{p}_1}$,
note that the complement is $D(\infty,1/R)$.
By considering Diagram \ref{diagramaRX},
it follows that 
$\pi_2^{-1}\big( D(\infty,1/R)\big)\subset\R_X$ 
contains an infinite number of 
branch points for $R>\abs{\widetilde{p}_1}$.
From this it immediately follows that 
$U_\infty(1/R)=\Psi_X^{-1}\big( D(\infty,1/R)\big)$ 
contains an infinite number of critical points $\{p_k\}$.
Thus $U_\infty$ is nonlogarithmic.

Finally, since $\Psi_X$ is entire, $\infty\notin U_\infty(\rho)$ for $\rho>0$, 
so $U_\infty$ is also direct.
\end{proof}

\noindent 
The singularities of $\Psi_{X}^{-1}$ are:

\noindent
$\bigcdot$ the algebraic
singularities corresponding to the
poles $p_k$ of $X$, and

\noindent
$\bigcdot$ 
exactly one
direct and non logarithmic singularity
$U_\infty$.

\noindent
Figure \ref{album-afin-AAP}.c illustrates 
the phase portrait of $\Re{X}$ at $\infty$.
Note that every neighbourhood $U_\infty (\rho)$
contains an
infinite number of poles of $X$.
Moreover, the phase space of $\Re{X}$ has at infinity a region 
which is a 
topological elliptic tract
(with an angular sector of angle $2\pi$),
however it is nonanalytically equivalent to 
the elliptic tract in 
Definition \ref{tract-eliptico-hiperbolico}.
\end{example}

\begin{example} 
\label{ejemplo-tan}
Consider the vector field 
  
\centerline{$X(z)= \tan(z)\del{}{z}.$}

\noindent  
A sketch of the phase portrait of $\Re{X}$ 
can be found in 
Figure \ref{album-afin-AAP}.f.
Once again the simple zeros of $X$ are 
$\mathcal{S}_{R}=\{ q_k \doteq k\pi \ \vert\ k\in\ZZ \}$,
whence $\Psi_X$ is multivalued.

\noindent
The poles of $X$ are 
$\mathcal{P}=\{ p_k \doteq\frac{\pi}{2} + k\pi \ \vert\ k\in\ZZ \}$, 
and since $\infty\in\CW_z$ is an accumulation point of zeros and poles, 
it follows that 
$\mathbb{E}=\{\infty\}$.

\noindent
We now choose a fundamental domain as in 
\S\ref{construccion-region-fundamental};
for this we note that because of Remark \ref{No-necesario-evitar-polos}.3
it is not necessary that $\Gamma$ avoid $\mathcal{P}$. 
Let 

\centerline{
$\gamma_0({\tt t}) = q_0 - {\tt t} \pi$, for  $t \in [0,1]$,
}

\centerline{
$\gamma_k({\tt t}) = q_k + \text{sign}(k)\, {\tt t} \pi$, for  $t \in [0,1]$, 
and $k\in\ZZ\backslash\{0\}$.
}

\noindent
Let $\Gamma_\varepsilon=\{ \gamma_k \}_{k\in\ZZ}$,
so
$\overline{\Gamma}\subset\CW_z$
is a simple path 
containing $\infty$.
It follows that a fundamental domain is
  
\centerline{
$\Lambda = 
( \CW_z\backslash\overline{\Gamma} )
\bigcup\limits_{k\in\ZZ} \gamma_{k+}$,
}  

\noindent  
as in \S\ref{construccion-region-fundamental}.
The restriction of $\Psi_X$ to $\Lambda$, for 
$z_{\tt o}=\pi/2$, 
  
\centerline{
$\Psi_{X,\,\Lambda}(z)
={\displaystyle  \int_{z_{\tt o}}^z }\cot(\zeta) d\zeta
=\log\big(\sin(z)\big)
: \Lambda \longrightarrow \CW_t$
}
  
\noindent 
is single--valued. 
The fundamental region is
  
\centerline{$\Omega= \left\{\big(z,\Psi_X(z)\big) \ \vert\ 
z\in\Lambda \right\} \subset \R_X.$}
  
\noindent
A surgery model for $\Lambda$ and $\Omega$, 
illustrating  Diagram \ref{diagrama-dominio-fundamental}, 
is shown in Figure \ref{tangente-proyecciones}.
Recalling that $k\in\ZZ$, the 
needed 
identifications are:

\centerline{
side $a_k$ is to be identified with side $b_{k+1}$,
side $d_k$ is to be identified with side $c_{k+1}$.}

\noindent
For simplicity of the drawing, the identifications are shown 
on two of the building blocks of $\Lambda$ 
and only on one of the building blocks of $\Omega$.

\noindent
In the same figure, on $\Lambda$ one can also observe (as red trajectories of $\Re{X}$) the segments $\{\gamma_k\}$ comprising $\Gamma$. 
The corresponding image on $\Omega$ is  
observed as the red trajectories of $\Re{X}$ that come from $\infty$ and land on the
branch points $(p_k,\widetilde{p}_k)$ that have ramification index 2.

\noindent 
Note that for the
sequence of simple zeros and simple poles accumulating
at $\infty\in\CW_z$,

\noindent 
$\bigcdot$ 
each simple zero $q_k$ of $X$, 
has asymptotic value $\infty$, 
  
\noindent
$\bigcdot$ 
each simple pole $p_k$ of $X$, 
has critical value 
$\widetilde{p}_k = i k\pi$;
that is
the critical values associated to the poles 
$\mathcal{P}$ are 
$\mathcal{CV}=\{ i k\pi\ \vert\ k\in\ZZ \}\subset\CC_t$,
and 
  
\noindent 
$\bigcdot$ 
the essential singularity at $\infty \in \CW_z$, 
has associated the asymptotic value $\infty$ 
with multiplicity two
arising from asymptotics paths, 
$\alpha_\pm({\tt t})\subset\Lambda$, 
that start at $z_{\tt o}$ and 
arrive at $\infty\in\CW_z$ inside
of the upper or lower half planes $\HH_{+}$ or 
$\HH_{-}$ respectively.

\begin{figure}[h]
\begin{center}
\includegraphics[width=1.0\textwidth]{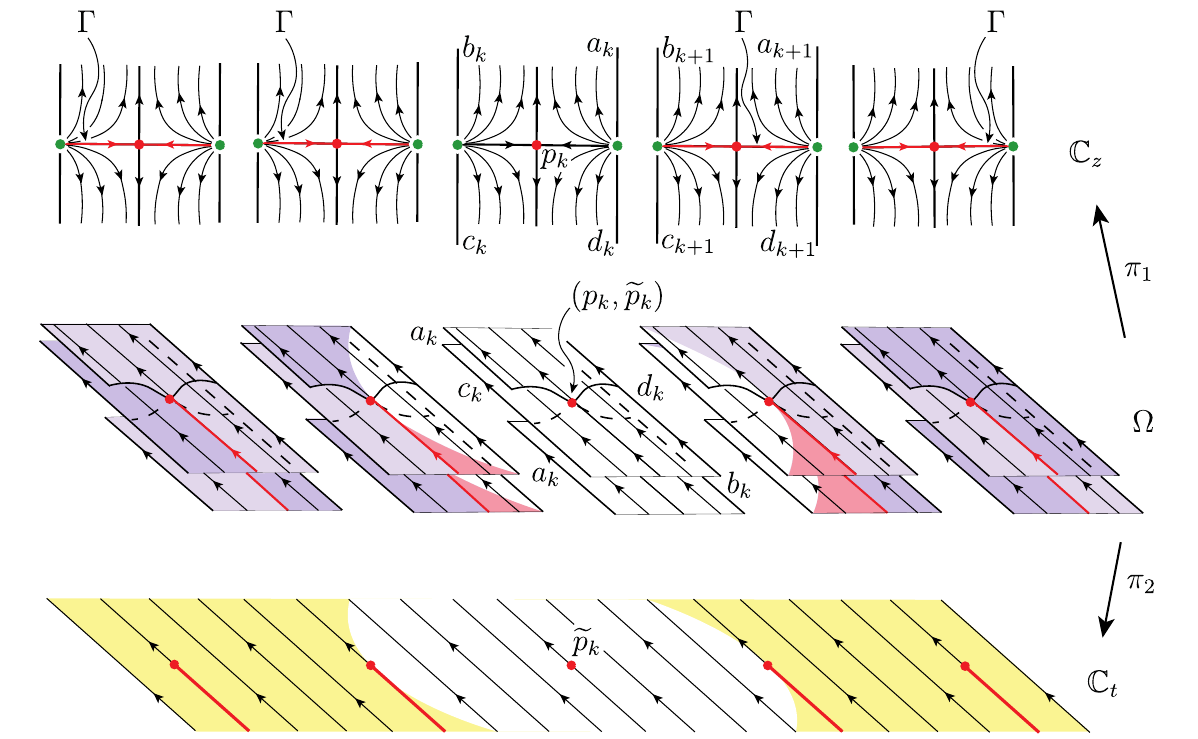}
\caption{
A sketch (using surgery) of the fundamental region
$\Omega \subset \R_X$ of the vector field 
$X(z)= \tan(z) \del{}{z}$.
Using the 
$\pi$--periodicity of $X$ in $\CC_z$,
we recognize that each vertical band in $\CC_z$
has a pole of 
$X$ (a red point), 
determining a branch point of index two in $\Omega$.
A neighbourhood $D(\infty, \rho) \subset \CC_t$ is
coloured yellow.  
The inverse image $\pi_2^{-1}\big( D(\infty,\rho)\big)$ 
lifts to $\Omega$ with several connected components.
The segments of $\pi_1^{-1} (\Gamma) \subset \Omega$ 
are in red, recall that they are cuts, 
hence the colours
that describe the connected components of 
$\pi_2^{-1}\big( D(\infty, \rho) \big)$
change along them.
Our interest lies in the
two connected components 
that contain an infinite number of branch points,
these connected components are coloured purple and dark 
purple.
The branch points 
$(p_k,\widetilde{p}_k)$ corresponding
to poles of $X$ are represented as red dots on
$\Omega$.
Note that the  critical values $\widetilde{p}_k$ are
$ ik \pi$, $k \in \ZZ$.
The zeros of $X$ are represented as 
green dots on $\CC_z$,
the corresponding branch points $(q_k,\infty)$
are not illustrated in $\Omega$.
}
\label{tangente-proyecciones}
\end{center}
\end{figure}  
  
\noindent
By using Diagram \ref{diagrama-dominio-fundamental}
and Definition \ref{definicion-singularidades-en-Lambda},
the singularities of 
$\Psi_{X,\,\Lambda}^{-1}$
are:
  
\noindent
$\bigcdot$ the 
$\star$--transcendental 
singularities 
$\{U_{\infty,k}\}_{k\in\ZZ}$ corresponding to the 
zeros $\{q_k\}$ of $X$, 
  
\noindent
$\bigcdot$ the algebraic singularities 
$U_{i k\pi}$ 
corresponding to the poles $p_k$,
and
  
\noindent
$\bigcdot$ the two 
essential transcendental singularities 
$U_{\infty+}$ and $U_{\infty-}$ corresponding to asymptotic paths $\alpha_{\pm}$.
  
\noindent
Since the critical values (arising from the poles of $X$) accumulate at 
$\infty \in \CW_t$, the asymptotic values $\infty$ 
are not isolated.
  
\noindent
For the separateness properties of the singularities $U_{\infty\pm}$, consider
the fundamental region $\Omega\subset\R_X$ 
(for the choice of fundamental domain $\Lambda$ as above). 
Moreover, 
recall that $\Psi_{X,\,\Lambda}^{-1}=\pi_{1}|_{\Omega}\circ\pi_{2}^{-1}$,
as in Diagram \ref{diagrama-dominio-fundamental}, 
thus in Figure \ref{tangente-proyecciones} 
the red segments 
$\pi_1^{-1}(\Gamma)\subset\Omega$ are cuts, 
and thus boundaries of $\Omega\subset\R_X$.
It follows that the purple and dark purple 
connected components of $\pi_2^{-1}\big(D(\infty,\rho)\big)$, 
correspond to 
the two essential transcendental 
singularities $U_{\infty\pm}$.
Note that they always intersect 
with an infinite number of neighbourhoods 
$V\big( (p_k,\widetilde{p}_k), \rho') \subset\Omega$ and 
$V\big( (q_k,\infty), \rho') \subset\Omega$, 
associated  to the poles and zeros of $X$. 
In other words the two neighbourhoods $U_{\infty \pm}(\rho)$
intersect an infinite number of neighbourhoods of the
branch points corresponding to the poles and zeros of $X$.

\noindent 
We conclude that the two essential transcendental singularities 
$U_{\infty\pm}$ are non separate.
  
\noindent
From the perspective of the universal cover
$\mathfrak{M}$ of
$\CW_z\backslash\{0,\infty\}$, 
Corollary \ref{cubierta-universal-singularidades}.1--3
applies. 
\end{example}

\section{Three applications}
\label{aplicaciones}
\subsection{Maximal domains for the flow: the description of $\R_X$}
\label{flujo-maximal}

Let $X$ be a singular complex analytic vector field
on a Riemann surface $M$. 
Our interest is in 
\emph{local nonstationary 
complex trajectory solutions} of $X$ 
with initial conditions $z_{\tt o} \in M \backslash \mathcal{S}$, \emph{i.e.}

\centerline{
$z(t) : D(0,\rho) \subset  \CC_{t}
\longrightarrow  M, 
\ \
t\longmapsto \Psi_X^{-1}(t),
$}

\noindent  
where $\Psi(z)= \int_{z_{\tt o}}^z \omega_X:
M \backslash \mathcal{S} \longrightarrow \CC_t$,
compare with Equation \eqref{parametro-local}.

\begin{definition}
1) A vector field $X$ is \emph{complete} when 
its  complex trajectory solutions $\{ z(t) \}$ are 
holomorphic for all complex time $t \in \CC_t$ and
all initial condition 
$z_{\tt o} \in M$. 
Otherwise $X$ is \emph{incomplete}.    

\noindent 
2)
A \emph{real incomplete trajectory}
$z({\tt t} ): (a, b) \subseteq \RR \longrightarrow M$ of $X$  
is such that its maximal domain is an strict subset 
$(a, b)$ of $\RR$. 
\end{definition}

The following result is well known,
an elementary proof 
is provided in \cite{LopezMucino}.

\begin{corollary}
\label{familias-campos-completos}
A singular complex analytic vector field $X$ on 
a Riemann surface $M$ 
is complete if and only if
belongs to one of the following families. 

\begin{enumerate}[label=\arabic*),leftmargin=*]
\item
$X$ is rational on $\CW$ with two zeros 
(counted with multiplicity).  

\item
$X$ is polynomial of degree zero or one on $\CC$.

\item
$X$ is polynomial of degree one on $\CC^*$
with zero at $0$.

\item
$X$ is holomorphic on a torus $\CC/ \Lambda$.
\hfill \qed
\end{enumerate}

\end{corollary}

For an incomplete $X$, the interesting phenomenon
is the following.

\begin{definition}
\label{dominio-maximal-de-univalencia}
A
\emph{maximal region of univalence $\mathscr{D}_X$ of $X$}
is the connected Riemann surface obtained by
analytic continuation of a local 
nonstationary complex trajectory solution $z(t)$,
along paths from $t=0$ in $\CC_t$.
\end{definition}

The surface $\mathscr{D}_X$ satisfies that 
$\pi_2 :\mathscr{D}_X \longrightarrow \CC_t$
is a Riemann domain (an unbranched cover).

\begin{example}[Meromorphic vector fields case]
Let $X$ be a meromorphic vector field on $M$, 
non necessarily compact.
The local analytic normal forms,
Proposition \ref{la-topologia-dice-la-singularidad},
show that 
each pole $p$ of $X$ of order/multiplicity $-k\leq -1$ provides a 
exactly $(2k+2)$ hyperbolic sectors, 
hence the same number of separatrices
which are incomplete trajectories 
$z({\tt t})$ having an $\alpha$ or $\omega$--limit
at the pole $p$.
We consider a cover  

\centerline{
$\pi_u: \widehat{M} \longrightarrow 
M \backslash\{ \mathcal{Z}_R\}
$}

\noindent 
that kills classes $[\beta] \in H_1(M, \ZZ)$
of the poles of $\omega_X$ with nonzero residue and the 
nonzero periods of $\omega_X$. 
Note that, since $\omega_X$ is meromorphic 
non zero classes $[\beta] \in H_1(M, \ZZ)$ 
with $\int_\beta \omega_X =0$ may exist; 
these classes are not killed by $\pi_u$. 
The maximal region of univalence of a non stationary
solution 
$z(t)$ of $X$ is the punctured surface

\centerline{$
\mathscr{D}_X = 
\widehat{M} \backslash  \{ \pi_u^{-1}(\mathcal{P} \cup
\mathcal{Z}_0)\}.
$}

\noindent
Moreover, we recognize that  

\centerline{$
\mathscr{D}_X
=
\{ (z,\Psi_X(z) ) \ \vert \
z \in M \backslash ( \mathcal{P} \cup \mathcal{Z} )
\}
= 
\R_X \backslash 
\pi_{1}^{-1}(\mathcal{P} \cup \mathcal{Z}_0). 
$}
\end{example}

The results outlined in
\S\ref{caso-widetilde-Psi}, 
particularly Corollary 
\ref{cubierta-universal-singularidades}
provides us with the following.

\begin{theorem}[Maximal univalence region for trajectory solutions]
\label{teo:cubierta-universal-RX}
Let $X$ be a singular complex analytic vector field
on $M$.
The maximal univalence region for a non stationary
complex solution $z(t)$ of $X$ is
\begin{equation*}
\mathscr{D}_X
=
\{ (z,\Psi_X(z) ) \ \vert \
z \in M \backslash \mathcal{S} \}.
\end{equation*}
 
\noindent 
Moreover, 
$\mathscr{D}_X$ is independent of 
the initial condition
$z_{\tt o} \in 
M \backslash \mathcal{S}$.
\end{theorem}

\begin{proof}
Let $\R_{X}\subset M\times\CW_{t}$ be the Riemann surface 
defined as the graph of 
$\Psi_{X} (z) = \int_{z_{\tt o}}^z \omega_X$ as in Equation 
\eqref{parametro-local}. 
The Riemann surface $\R_X$ is a leaf of the 
singular complex analytic vector field

\centerline{$
f(z) \del{}{z} + \del{}{t} \ \ \
\hbox{ on }  M \times \CC_{t},
$}

\noindent
where $f(z) \del{}{z}$ 
corresponds to $\Psi_X(z)$ due to the Dictionary \eqref{diagramacorresp}.
The singular complex analytic foliation in the 
two dimensional complex
manifold has as leaves:

\noindent 
$\bigcdot$
copies of $\R_X$ under translations
in the $\CC_t$ factor, and
 
\noindent 
$\bigcdot$
horizontal copies of $\{q\} \times \CC_t$
from each zero $q$ of $X$. 

\noindent 
Clearly, $\mathscr{D}_X$ is independent of the
initial condition $z_{\tt o}$.
Since we are considering holomorphic solutions, 
it is necessary to remove  the set
$\pi_1^{-1} (\mathcal{P} \cup \mathcal{Z}_0)$
from $\R_X$.
\end{proof}

\begin{example}
\label{ejemplo-3-puntos}
Consider the vector field

\centerline{
$X(z)=z(z-1)\e^{-z} \del{}{z}\in\E(2,0,1)$,
}

\noindent
with singular set 
$\mathcal{S}=\mathcal{S}_{R}
=\{0,1,\infty\}\subset\CW_z$.
Its associated multivalued additively automorphic 
singular complex analytic function is

\centerline{$\Psi_X(z)=
{\displaystyle  \int^z } 
\dfrac{\e^z}{\zeta (\zeta-1)}d\zeta.
$}

\noindent
The residues of the 1--form of time $\omega_X$ at 
$\mathcal{S}_{R}$ are 
$\{-1,\e, 1-\e\}$, respectively.

\noindent
A fundamental domain, as in \S\ref{construccion-region-fundamental},
can be chosen as follows.
Let $\gamma_1({\tt t})=1+ i {\tt t}$  and 
$\gamma_2({\tt t})=- i {\tt t}$, for ${\tt t}\in(0,\infty)$. 
Furthermore, let $\Gamma=\gamma_1 \cup \gamma_2$.
Thus, a fundamental domain is 

\centerline{
$\Lambda=
\Big(
\CW\backslash\overline{\Gamma}
\Big)
\cup (\gamma_{1+} \cup \gamma_{2+})$.
}

\noindent
Note that the singularities of $\Psi_{X,\,\Lambda}^{-1}$
are two $\star$--transcendental singularities over 
$\infty\in\CW_t$ corresponding to the two zeros of 
$X$, and two logarithmic singularities: 
$U_\infty$ over $\infty$
and $U_a$ over the finite asymptotic value 

\centerline{
$a=\lim\limits_{{\tt t}\to\infty} \int^{\alpha({\tt t})}
\omega_X = 0$.
}

\noindent
In other words, $\Lambda$ contains an elliptic tract $U_\infty (\rho)$
(corresponding to the logarithmic singularity 
over $\infty$) and a hyperbolic tract $U_a (\rho)$ (corresponding to 
the logaritmic singularity over the finite asymptotic
value $a$).
See Figure \ref{campos-Delta}.a.

\noindent
Since $M\backslash\overline{\mathcal{S}_{R}}
=\CW\backslash\{0,1,\infty\}$, 
then the universal cover
$\mathfrak{M}=\Delta$ is the unit disk.
Moreover $\mathfrak{M}$ is composed of infinite copies of 
an ideal hyperbolic quadrangle $\Lambda$
glued together at the two borders 
$\gamma_1$ and $\gamma_2$, see Figure 
\ref{campos-Delta}.b.
The function $\widetilde{\Psi_X} :\Delta \longrightarrow \CW$
is holomorphic. 
The singular points of $\widetilde{\Psi_X}$ 
form a countable and dense
set on the boundary $\partial\Delta$ of the disk $\Delta$;
they are precisely the ideal points of 
$\widetilde{\Psi_X}^{-1}$.
The reader can compare the present singularities of 
$\widetilde{\Psi_X}$ on the boundary $\partial \Delta$
with the classical theorem of 
A. I. Plessner for singularities
on the boundary, see \cite{Pommerenke} \S 6.4. 
Clearly, $\widetilde{\Psi_X}$ is invariant under a 
Fuchsian group (the group of deck transformations).

\begin{figure}[htbp]
\begin{center}
\includegraphics[width=0.9\textwidth]{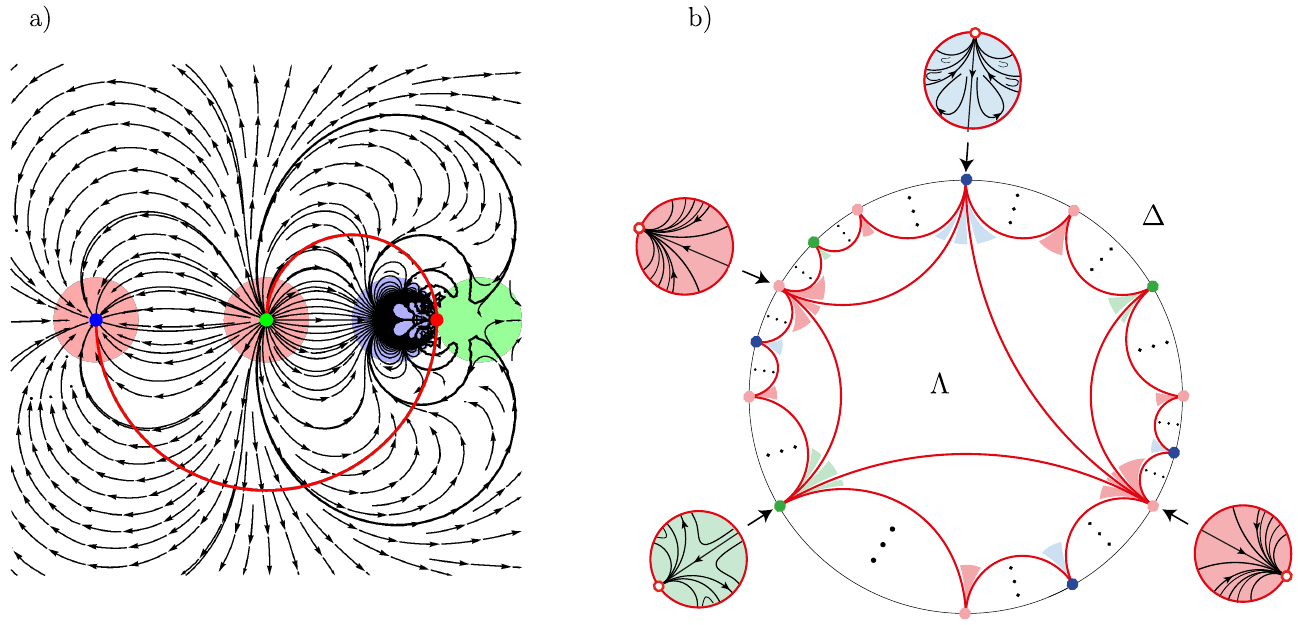}
\caption{
We regard
the vector field $X(z)=z(z-1)\e^{-z} \del{}{z}$,
Example \ref{ejemplo-3-puntos}.
(a) in order to 
better visualize the behaviour, we 
show the pullback
vector field $Y(w)=(T^* X)(w)$, with $T(w)=w/(-w+1)$; 
the blue and green points correspond to 
the zeros of $Y$, and the red point is the 
essential singularity of $Y$. 
The red arcs of a circle correspond to the inverse 
images of $\Gamma$.
The hyperbolic and elliptic tracts are shaded 
green and blue respectively, 
while the $\star$--transcendental singularities 
corresponding to the zeros are pink.
(b) Shows the universal cover $\mathfrak{M}\cong\Delta$ and 
some copies of the fundamental domain $\Lambda$.
Note that the ideal points of $\widetilde{\Psi_X}^{-1}$ 
form a countable dense set on the boundary 
$\partial\Delta$ of the disk $\Delta$. 
Each neighbourhood of these ideal points is composed
by an infinite number of angular sectors with angle 0.
Each angular sector is a tract  
of the ideal point of $\Psi_{X,\,\Lambda}^{-1}$ with
the same behaviour. 
In this example: hyperbolic tracts (green), elliptic tracts (blue)
and parabolic sectors (pink).
}
\label{campos-Delta}
\end{center}
\end{figure}


\end{example}

\subsection{Localizing incomplete trajectories}\label{localizingIncomplete}
In \cite{Guillot} A. Guillot explores 
relations between complex differential equations 
and the geometrical properties of their
(incomplete) trajectories. We recall facts.  

\begin{proposition}\label{prop:infinidad-trayectorias-incompletas}
Let $X$ be a
singular complex analytic vector field on  
a compact Riemann surface $M_\mathfrak{g}$. 

\begin{enumerate}[label=\arabic*),leftmargin=*]
\item
A vector field $X$ is rational and nonholomorphic on 
$M_\mathfrak{g}$ if and only if
$X$ has a finite (nonzero) number of incomplete trajectories.

\item
Every nonrational, 
singular complex analytic vector field $X$ on $M_\mathfrak{g}$, 
has an infinite number of incomplete trajectories. 
\end{enumerate}
\end{proposition}

\begin{proof}
Assertion (1) uses the normal form 
in Proposition \ref{la-topologia-dice-la-singularidad}.
For assertion (2), the argument is by contradiction, 
if the number of 
incomplete trajectories is
finite, then by (1) $X$ is rational.
\end{proof}

Note that the above proof is not constructive;
however,
the appearance of incomplete trajectories
in the vicinity of an essential singularity of $X$ 
is explained
in the next subsection.

\begin{remark}
\label{rem-tray-incompletas}
Let $X$ be a rational vector field on 
the Riemann sphere. 
There exists an
incomplete trajectory $z({\tt t} )$ of $X$ 
having $\alpha$ or $\omega$--limit at $p\in\CW_z$
if and only if 
$p$ is a pole of $X$, equivalently
$p$ is a critical point of $\Psi_X$ with
a finite critical value 
$\widetilde{p}=\Psi_{X}(p) \in\CC_t$.
In particular, if $\alpha_p ({\tt t} )$ is a path with $\lim\limits_{{\tt t} \to\infty} \alpha_p({\tt t})=p$,
then  

\centerline{
$\lim\limits_{{\tt t} \to\infty}\Psi_X(\alpha_p ({\tt t} ))=\widetilde{p}$.}
\end{remark}

With this in mind, 
the following is straightforward.

\begin{theorem}[Incomplete trajectories and finite singular values]
\label{teo:tray-asintoticas-e-incompletas}
Let $X$ be a singular complex analytic vector field on
$M$.
The following statements are equivalent. 
\begin{enumerate}[label=\arabic*),leftmargin=*]
\item
There exists an
incomplete trajectory $z({\tt t} )$ of $X$ 
having $\alpha$ or $\omega$--limit at 
$z_{\tt s} \in M$.

\item
There exists a finite singular value 
$a \in\CC_t$ of $\Psi_X$, whose asymptotic path 
$\alpha_a ({\tt t} )$ is a trajectory of $\Re{X}$ ending at $z_{\tt s} \in M$.
\end{enumerate}
\end{theorem}

\begin{proof}
The argument
follows directly from the definitions of asymptotic 
path, of a finite asymptotic value of $\Psi_X$ and 
of incomplete trajectories of $X$.
\end{proof}

\begin{remark} 
Theorem \ref{teo:tray-asintoticas-e-incompletas} 
is independent 
of whether $\Psi_X$ is single or multivalued.
\end{remark}
A natural question to ask is where these incomplete 
trajectories are localized in a vicinity of an essential 
singularity. The assertion is as follows.

\begin{theorem}[Localizing incomplete trajectories]
\label{Teo:ValAsintotico-TrayIncompleta}
Let $X$ be a singular complex analytic vector field on  
$M$ with an essential singularity at $z_{\tt s} \in M$.

\begin{enumerate}[label=\arabic*),leftmargin=*]
\item
Any neighbourhood $U_a(\rho)$, 
of an essential transcendental singularity $U_a$ of 
$\Psi_X^{-1}$ over a finite 
asymptotic value $a\in\CC_t$,  
contains an infinite number
of incomplete trajectories of $X$.

\item
If $\Psi_X$ has no finite asymptotic values 
at $z_{\tt s}$,
then $X$ has an infinite number of poles accumulating at 
$z_{\tt s} \in M$.
\end{enumerate}
\end{theorem}

\begin{proof}
For statement (1), 
first consider the case when $U_a$ is a logarithmic singularity of $\Psi_X^{-1}$.
Recalling Theorem \ref{singularidades-algebraicas-y-logaritmicas}.1, 
note that for $\rho>0$ small enough, 
the neighbourhood $U_{a}(\rho)$ 
of a logarithmic singularity $U_a$ over a finite asymptotic value $a$ is a hyperbolic tract.
It consists of an infinite number of hyperbolic sectors, 
and the separatrices of each hyperbolic sector are incomplete trajectories. 
Thus any neighbourhood $U_{a}(\rho)$ of the logarithmic singularity $U_a$ contains an 
infinite number of incomplete trajectories.

On the other hand, if the transcendental singularity $U_a$ of $\Psi_X^{-1}$ 
is nonlogarithmic,
by Theorem \ref{teo-logaritmicas-separada}, $U_a$ is nonseparate.
Thus
for any $\rho_a>0$, the neighbourhood
$U_a(\rho_a)$ contains an infinite number of neighbourhoods 
$U_{a_\sigma}(\rho_{\sigma})$, for appropriate
$\{ \rho_\sigma>0 \}$. 
Note that the collection $\{a_\sigma\}$ is bounded, \emph{i.e.} the $a_\sigma$ 
are all finite, and 
satisfy 
\begin{equation}
\label{contencion-de-vecindades}
U_{a_\sigma}(\rho_{\sigma}) \subset U_a(\rho_a).
\end{equation}

\noindent
If an infinite number of the $a_\sigma$ are critical values, we are done: 
these critical values have corresponding critical points 
that are poles of $X$.
Thus, 
by \eqref{contencion-de-vecindades},
any neighbourhood $U_{a}(\rho)$ of the nonlogarithmic singularity $U_a$ contains an 
infinite number of incomplete trajectories.

\noindent
Otherwise the collection $\{a_\sigma\}$ contains an infinite number of distinct (finite) 
a\-symp\-to\-tic values.
Without loss of generality,
we shall assume that the $\{a_\sigma\}$ are all asymptotic values and 
that they once again satisfy \eqref{contencion-de-vecindades}.
Now recall that the associated Riemann surface $\R_{X}$ has as its (ideal) boundary 
precisely the branch points corresponding to all the asymptotic values of $\Psi_X$. 

\noindent
Since the (ideal) boundary of  $\R_{X}$ is totally disconnected, 
then every single branch point corresponding
to the singularities $U_{a_\sigma}$ has a trajectory $\widetilde{\alpha}_\sigma ({\tt t} )\subset\R_X$
arriving to it. This trajectory projects,
via $\pi_1$, to an incomplete trajectory 
$\alpha_\sigma ({\tt t} )\subset U_{a_\sigma}(\rho_{\sigma}) 
\subset U_a(\rho_a)
\subset M $.

\medskip
The proof of statement (2) is by contradiction.
Assume that there is only a finite number of poles of $X$, the number of incomplete
trajectories is then finite. 
This contradicts Proposition 
\ref{prop:infinidad-trayectorias-incompletas}.
\end{proof}

The interested reader can compare the above results with theorems 1.2 and 1.3 of \cite{Sixsmith}.

\begin{remark}
Whenever there is an essential singularity of $X$, 
we have the dichotomy described below.

\noindent
$\bigcdot$ 
If $\Psi_X$ has no finite asymptotic values, 
then $X$ has an infinite number of poles accumulating at the essential singularity of $X$ at 
$z_{\tt s}\in M$.

\noindent
$\bigcdot$ 
If $X$ only has a finite number of poles, 
then $\Psi_X$ has (at least) one finite asymptotic
value.
\end{remark}

\subsubsection{What can be said about $X$ without an explicit knowledge of $\Psi_X$?}

\noindent
As a direct consequence of Theorem \ref{Teo:ValAsintotico-TrayIncompleta}, 
we can extend Langley's result (see \cite{Langley}) from
the case when $f^{-1}$ has a logarithmic singularity over $\widetilde{a}=\infty$, 
to the general case:

\begin{corollary}
\label{sing-trancendente-de-f}
Let $X=f(z)\del{}{z}$ be a singular complex
analytic vector field on 
$M$
with an essential singularity at 
$z_{\tt s}\in M$.
Any neighbourhood $U_{\widetilde{a}}(\rho)$ of a 
transcendental singularity
$U_{\widetilde{a}}$ of $f^{-1}$
over a nonzero asymptotic value $\widetilde{a}
\in\CW_t\backslash\{0\}$
contains an infinite number of incomplete trajectories of 
$X$.
\end{corollary}
\begin{proof}
By definition,  
$f: M \longrightarrow\CW_t$
is transcendental meromorphic.
Since $f$ has a nonzero asymptotic value 
$\widetilde{a}$, 
then it follows that 
there is an asymptotic path 
$\widetilde{\alpha}({\tt t} )$ 
of $f$ such that 
$\abs{ \frac{1}{f( \widetilde{\alpha}({\tt t} ) )} } 
\subset 
D\Big( \frac{1}{ \abs{\widetilde{a}} }, \varepsilon \Big)$ 
for small enough $\varepsilon>0$ and 
large enough ${\tt t} >0$. 
Thus, 

\centerline{
$\lim\limits_{{\tt t} \to\infty}\Psi_X\big( \widetilde{\alpha}({\tt t} ) \big)=a \in\CC_t$,}

\noindent  
{\it i.e.} $\Psi_X$ has $a$ as a finite asymptotic value.
By Theorem \ref{Teo:ValAsintotico-TrayIncompleta}, 
we are done.
\end{proof}

\begin{lemma}\label{flog-Psilog}
The following assertions are equivalent.

\begin{enumerate}[label=\arabic*),leftmargin=*]
\item
$f^{-1}$ has a logarithmic singularity over an asymptotic value 
$\widetilde{a} \in \CW$.

\item
$\Psi_X^{-1}$ has a logarithmic 
singularity over the corresponding asymptotic value $a \in\CW$ as
in \eqref{valasintPsi}.
\end{enumerate}
\end{lemma}
\begin{proof}
$(1)\Rightarrow(2)$. 
From the definition, a transcendental singularity 
$U_{\widetilde{a}}$ of $f^{-1}$
is a logarithmic singularity over $\widetilde{a}$ 
if $f:
U_{\widetilde{a}}(\rho)
\longrightarrow 
D(\widetilde{a},\rho)\backslash\{\widetilde{a}\}\subset\CW$ 
is a universal covering 
for some $\rho>0$. 
Hence, there exists a biholomorphism 
$\phi:D(0,r)\longrightarrow 
U_{\widetilde{a}}(\rho)$ 
such that $f(\phi(w))=\exp(w)$ for small enough $r>0$.
In other words, $\Psi_X^{-1}$ has a 
logarithmic singularity over $a$.

\noindent
$(1)\Leftarrow(2)$.
Since $\Psi_X$ is a universal cover, locally $\Psi_X(\phi(w))=\exp(w)$ 
so $f(\psi(w))= \frac{d}{dw} \exp(w) =\exp(w)$,
{\it i.e.} $f$ is a universal covering for some $\rho>0$.
\end{proof}

The following complements Corollary \ref{sing-trancendente-de-f}.
Compare with \cite{Langley} theorem 1.2.

\begin{proposition}\label{hipotesisf}
Let 
$f: M \longrightarrow\CW_t$
be a transcendental meromorphic function,
such that $f^{-1}$ has a logarithmic singularity $U_{\widetilde{a}}$ 
over $\widetilde{a}\in\CW_t$.

\begin{enumerate}[label=\arabic*),leftmargin=*]
\item
If the singularity $U_{\widetilde{a}}$ of $f^{-1}$ is over a nonzero asymptotic value
$\widetilde{a} \in \CC^*\cup\{\infty\}$,
then $X$, at $z_{\tt s}\in M$,
has an infinite number of hyperbolic sectors
and an infinite number of incomplete trajectories.

\item
If the singularity $U_{\widetilde{a}}$ of $f^{-1}$ is over the asymptotic value
$0=\widetilde{a} \in\CC_{t}$,
then $X$ at $z_{\tt s}\in M$
has an infinite number of elliptic sectors.
\end{enumerate}
\end{proposition}

\begin{proof}
Because of Lemma \ref{flog-Psilog}, 
it follows that $\Psi_X^{-1}$ has a logarithmic singularity over
$a$, recall Equation \eqref{valasintPsi}.
By Theorem \ref{teo-logaritmicas-separada}, 
$f$ has at most a finite number of zeros and poles 
in the exponential tract.
Thus 

\centerline{
$\lim\limits_{{\tt t} \to\infty}\Psi_X\big( \widetilde{\alpha}({\tt t} ) \big)
=
\begin{cases}
a \in\CC_t & 
\text{if } \widetilde{a} \in \CC^*\cup\{\infty\}, 
\\
a=\infty & \text{if } \widetilde{a} = 0,
\end{cases}$
}

\noindent
and
hence by Theorem \ref{singularidades-algebraicas-y-logaritmicas} we are done.
\end{proof}


\subsection{Riemann $\xi$--vector field}
\label{xi-campo-Broughan-subseccion}
Let 
\begin{equation}
\label{xi-campo}
X_\xi(z) =
\xi(z) \del{}{z}
\end{equation}
be the entire 
\emph{Riemann $\xi$--vector field}, 
arising by considering the Riemann $\xi$--function as in
K. Broughan \emph{et al.} \cite{Broughan}, 
\cite{Broughan-Barnett}.
The following features are related to the vector field \eqref{xi-campo}.

\noindent 
$\bigcdot$
The multivalued
additively automorphic 
function associated to $X_\xi$ is 

\centerline{
$\Psi_{X_\xi} (z) = 
{\displaystyle \int_{z_0}^z }
\dfrac{d\zeta}{ \xi(\zeta)}:
\CC \backslash \{ \xi(z)=0 \}
\longrightarrow \CC_t.$
}

\noindent 
$\bigcdot$ 
The real singular foliation of $\Re{ X_\xi }$
has a symmetry of reflection 
with
respect to the critical line $\{\Re{z}=1/2 \}$.
In particular, it implies that the simple zeros 
$\{ \frac{1}{2} + i \gamma_n \}$
of $X_\xi$ are isochronous centers.

\noindent 
$\bigcdot$
In \cite{Broughan-Barnett}, it is proved that
$\log{T_n} \geq \frac{\pi}{4} \gamma_n + O(\log{\gamma_n})$ for $n\in\NN$,
where $T_n$ is the absolute value of the periods
$\tau_{n}=2\pi i /\xi' \big( \frac{1}{2} +i \gamma_{n} \big)$
of the $n$--th isochronous center 
$\frac{1}{2} + i \gamma_{n}$ along the critical line.
This implies that $T_n$ is strictly increasing.

\noindent 
$\bigcdot$
In \cite{Broughan}, it 
is proved that there exists an infinite number of 
incomplete trajectories $\Gamma_j$, for $j\in\ZZ$, 
with $\alpha$ and $\omega$--limits at $\pm \infty$ 
that do not contain
any singularities of $X_\xi(z)$
(crossing separatrices in their terminology).
Moreover, these incomplete trajectories separate the zeros
on the critical line, \emph{i.e.} $\Gamma_{j-1}$ and $\Gamma_j$ 
are the boundaries of an unbounded band $B_j$.

\noindent
Since it is unknown whether all the zeros on the critical line are simple (centers), then 
the band containing the $n$--th isochronous center $\frac{1}{2} + i \gamma_n$ 
along the critical line is $B_{j(n)}$.
Note that there might be other zeros (not on the critical line) inside 
each band $B_j$.
However, it is well known that if there are zeros not on the critical line, they must lie inside
the critical strip: 
a vertical strip of width 1 centered at the critical line.

\noindent  
As is expected, 
we show that 
$X_\xi$ can not be as simple as a pullback of 
a periodic vector field, 
compare with theorem 6.1 of \cite{Broughan}.

Since $\Psi_{X_\xi}$ is a multivalued additively 
automorphic meromorphic function on $\CC_z$, 
we proceed 
to construct a fundamental domain $\Lambda$ 
as in \S\ref{caso-multivaluado-para-Psi}.1.

\noindent
Consider first a closed Jordan path $\widehat{\Gamma}$ 
that contains $\mathcal{Z}_{R}\cup \{\infty \}$
and the vertical segment 
$[\frac{1}{2}-i \gamma_1,\frac{1}{2}+i \gamma_1]$.
Now let $\Gamma=\widehat{\Gamma}\backslash[\frac{1}{2}-i \gamma_1,\frac{1}{2}+i \gamma_1]$ and the
fundamental domain for $\Psi_{X_\xi}$ be
$\Lambda=(\CW_z\backslash\Gamma)\cup\Gamma_+$.

\begin{proposition}\hfill
\label{prop:Xi-Riemann}
\begin{enumerate}[label=\arabic*),leftmargin=*]

\item 
The Riemann $\xi$--vector field  \eqref{xi-campo} 
is not holomorphically equivalent to a
pullback of a periodic vector field $Y$ with a finite 
number of distinct residues.

\item
Let $\Lambda$ be the fundamental domain described above.
The single--valued function $\Psi_{X_\xi,\,\Lambda}$ 
has: 
\begin{enumerate}[label=\alph*),leftmargin=*]
\item an infinite number of $\star$--trascendental singularities over $\infty$ corresponding to the
zeros with nonzero residue in the critical strip,

\item 
two logarithmic singularities $U_{a_{\alpha_\pm}}$ over the 
finite asymptotic values $a_{\alpha_\pm}$,

\item
two hyperbolic tracts: the left and right hand planes, 
$\{ \Re{z}<0 \}$ and $\{ \Re{z}>1 \}$.
\end{enumerate}

\end{enumerate}
\end{proposition}

\begin{proof}
For the first statement,
recall the fact that the residues of a 
vector field at its zeros 
are holomorphic invariants. 
However 
$X_\xi(z)$ has an infinite number of \emph{distinct} periods.

For the second statement, 
let $\mathcal{Z}_{R}$ denote
the zeros of $\xi(z)$ that determine a nonzero residue 
of the 1--form of time $dz/\xi(z)$,
in particular the simple zeros 
$\{ \frac{1}{2}\pm i\gamma_n \}$ 
of $\xi(z)$ are contained in $\mathcal{Z}_{R}$.
Each of them is associated to a $\star$--transcendental singularity over $\infty$.

\noindent
On the other hand, 
the function $\xi(z)$ is entire, 
thus $\Psi_{X_\xi}$ does not have
any finite critical values.
Moreover, since $\Psi_{X_\xi ,\, \Lambda}$ 
is single--valued on $\Lambda$, 
and there are only two homotopy classes of paths approaching $\infty\in\CW_z$,
there are two finite asymptotic values for 
$\Psi_{X_\xi ,\, \Lambda}$
as follows:
let 
$\alpha_\pm( {\tt t} )\subset
\CC_z\backslash\mathcal{Z}_{R}$
denote a simple path starting at 
$z_{\tt o}$ and ending at $\pm\infty\in\CW_z$ tangent to the 
real axis. 
The two finite asymptotic values are:

\centerline{
$a_{\alpha_\pm}\doteq \lim_{ {\tt t} \to\infty} 
{\displaystyle  \int_{\alpha_\pm( {\tt t})} } 
\dfrac{d\zeta }{ \xi(\zeta)}
\simeq \{13.0074, -10.9997\} \subset\CC_t.$
}

\noindent
Thus we have two essential transcendental singularities $U_{a_{\alpha_\pm}}$ over the two finite 
asymptotic values $a_{\alpha_\pm}$.
Moroever they are logarithmic since their 
asymptotic values are isolated.
Their neighbourhoods $U_{a_{\alpha_\pm}}(\rho)$ are 
hyperbolic tracts. 
Note that outside of the critical strip $\{ z\in\CC_z\ \vert\ 0<\Re{z}<1 \}$ 
there are no zeros of $X_\xi(z)$, hence the hyperbolic tracts are as stated.
\end{proof}

\section{Future work}
\label{Future-work-subseccion}
The use of vector fields $X$, in particular their phase portrait, allows us to observe 
the following new phenomena,
even for single--valued functions $\Psi_X$.

\smallskip
\noindent $\bigcdot$ 
In Example \ref{Ejemplo-ExpSin},
the real line is not an asymptotic path thus there is no
transcendental singularity associated to the real line, 
however any other path arriving to 
$\infty\in\CW_z$ corresponds to a logarithmic 
singularity of $\Psi_{X}^{-1}$.
Thus 
\emph{at $\infty\in\CW_z$ there are an infinite number 
of logarithmic singularities,
and two rays $\RR^\pm$
that do not correspond to transcendental singularities 
of $\Psi_{X}^{-1}$.}
This phenomena is not captured by Definition 
\ref{eremenko1}.
As an extreme situation, 
Example \ref{sing-campo-no-implica-sing-psi} 
shows that there are singularities of 
vector fields $X$ that do not allow any singularities of the inverse $\Psi_X^{-1}$.
A further characterization of these singularities might address this limitation.

\smallskip
\noindent $\bigcdot$
The extension of 
Theorem \ref{familias-P-r} 
to the case 
$\Psi_X(z)= f(g(z))$,
for a pair of entire functions $f$, $g$
with 
noncommensurable periods, remains open.
This kind of factorization technique is useful to 
study families of transcendental 
functions, 
see for instance \cite{Steinmetz-Fact}.

\smallskip 
\noindent $\bigcdot$  
A systematic study of non separate
singularities of $\Psi_X^{-1}$ is left for future work.
Particularly interesting is the case when the cardinality of 
$\mathbb{E}\cup\mathcal{Z}_{R}$ 
is at least 3: 
there is an
obvious relationship with Fuchsian groups and with the classical results of Plessner 
for singularities on the boundary of the disk 
(see \cite{Pommerenke} \S 6.4).

\smallskip
\noindent $\bigcdot$
Of course the complete study of Riemann 
$\xi$--function, from the perspective of vector fields, 
warrants further work. 
Two obvious 
perspectives present themselves: 
(a) to consider $X_{\xi}(z)=\xi(z)\del{}{z}$, 
or (b) to consider $\Psi(z)=\xi(z)$ and its 
corresponding vector field 
$\big(1/\Psi'(z)  \big)\del{}{z}$.

\smallskip
\noindent $\bigcdot$
If $\Phi$ is any
multivalued singular complex 
analytic function, then
the extension of Iversen's theory, 
(developed in \S\ref{caso-multivaluado-para-Psi}, 
mainly Definition \ref{Nueva-def-vecindades} and its consequences),
can be carried through
so long as one can find a fundamental 
domain $\Lambda$.
({\it i.e.} a maximal simply connected univalence region for 
$\Phi$).
However, a priori the properties of the singularities
of the inverse $\Phi^{-1}$ can depend on the choice of 
$\Lambda$.
A first step in this direction would entail examining the case where 
the multivalued singular complex 
analytic function $\Phi$ has an 
automorphy factor, which is not just a translation 
(as is the case for additively automorphic functions).


\begin{thebibliography}{10}

\bibitem{Iversen}
F.~Iversen.
\newblock {\em Recherches sur les fonctions inverses des fonctions
  m{\'e}romorphes}.
\newblock PhD thesis, Helsingfords, 1914.

\bibitem{BergweilerEremenko}
W.~Bergweiler and A.~Eremenko.
\newblock On the singularities of the inverse to a meromorphic function of
  finite order.
\newblock {\em Rev. Mat. Iberoamericana}, 11(2), 1995, 355--373.

\bibitem{EremenkoReview}
A.~Eremenko.
\newblock Singularities of inverse functions,
\newblock 2021.
\newblock \url{https://arxiv.org/abs/2110.06134}

\bibitem{Nevanlinna1}
R.~Nevanlinna.
\newblock {\em Analytic Functions}.
\newblock Springer--Verlag, 1970.

\bibitem{Hille}
E.~Hille.
\newblock On the zeros of the functions of the parabolic cylinder.
\newblock {\em Arkiv f\"ur Mathematik, Astronomy Och Physik}, 18(26), 1924, 1--56.

\bibitem{Taniguchi1}
M.~Taniguchi.
\newblock Explicit representations of structurally finite entire functions.
\newblock {\em Proc. Japan Acad. Ser. A Math. Sci.}, 77, 2001, 68--70.

\bibitem{Taniguchi2}
M.~Taniguchi.
\newblock Synthetic deformation space of an entire function.
\newblock {\em Contemp. Math.}, 303, 2002, 107--136.

\bibitem{AlvarezMucinoII}
A.~Alvarez--Parrilla and J.~Muci{\~n}o--Raymundo.
\newblock Dynamics of singular complex analytic vector fields with essential
  singularities {II}.
\newblock {\em J. Singul.}, 2022, 1--78.

\bibitem{Langley}
J.~K.~Langley.
\newblock Trajectories escaping to infinity in finite time.
\newblock {\em Proc. Amer. Math. Soc.}, 145(5), May 2017, 2107--2117.

\bibitem{Langley-2}
J.~K.~Langley.
\newblock Transcendental singularities for meromorphic functions with
  logarithmic derivative of finite lower order.
\newblock {\em Comput. Methods Funct. Theory}, 19, 2019, 117--133.

\bibitem{Langley-3}
J.~K.~Langley.
\newblock Complex flows, escape to infinity and a question of {R}ubel.
\newblock {\em Ann. Fenn. Math.}, 47(2), 2022, 885--894.

\bibitem{Broughan}
K.~A.~Broughan.
\newblock The holomorphic flow of {R}iemann's function $\xi(z)$.
\newblock {\em Nonlinearity}, 18, 2005, 1269--1294.

\bibitem{Berenstein-Gay}
C.~A.~Berenstein and A.~Gay.
\newblock {\em Complex Variables an introduction}, volume 125 of {\em Graduate
  Texts in Mathematics}.
\newblock Springer, Berlin, 2nd edition, 1997.

\bibitem{MR}
J.~Muci{\~n}o--Raymundo.
\newblock Complex structures adapted to smooth vector fields.
\newblock {\em Math. Ann.}, 322, 2002, 229--265.

\bibitem{AlvarezMucino}
A.~Alvarez--Parrilla and J.~Muci{\~n}o--Raymundo.
\newblock Dynamics of singular complex analytic vector fields with essential
  singularities {I}.
\newblock {\em Conform. Geom. Dyn.}, 21, 2017, 126--224.

\bibitem{Gyllstrom}
G.~Gyllstr{\"o}m.
\newblock Solutions graphiques d'{\'e}quations diff{\'e}rentielles du premier
  ordre.
\newblock {\em Meddelanden fran Statens Meteorologisk--Hydrographiska Anstalt},
  4(9), 1929.

\bibitem{Ahlfors-Sario}
L.~V.~Ahlfors and L.~Sario.
\newblock {\em Riemann Surfaces}.
\newblock Number~26 in Princeton Math. Series. Princeton University Press,
  Princeton, N.J., 1960.

\bibitem{Richards}
I.~Richards.
\newblock On the classification of noncompact surfaces.
\newblock {\em Trans. Amer. Math. Soc.}, 106, 1963, 259--269.

\bibitem{AlvarezMucino2}
A.~Alvarez--Parrilla and J.~Muci{\~n}o--Raymundo.
\newblock Symmetries of complex analytic vector fields with an essential
  singularity on the {R}iemann sphere.
\newblock {\em Adv. Geom.}, 21(4), 2021, 483--504.

\bibitem{Kilmes-Rousseau}
M.~Klime\v{s} and Ch.~Rousseau.
\newblock Remarks on rational vector fields on $\mathbb{CP}^1$.
\newblock {\em J. of Dynamical and Control Systems}, 27, 2021, 293--320.

\bibitem{Dias-Garijo}
K.~Dias and A.~Garijo.
\newblock On the separatrix graph of a rational vector field on the {R}iemann
  sphere.
\newblock {\em J. Differential Equations}, 288, 2021, 541--565.

\bibitem{Jenkins}
J.~Jenkins.
\newblock {\em Univalent Functions and Conformal Mapping}.
\newblock {\em Ergebnisse der Mathematik und ihrer Grenzgebiete}.
Springer--Verlag, Berlin, 1958.

\bibitem{Strebel}
K.~Strebel.
\newblock {\em Quadratic Differentials}, volume~5 of {\em Ergebnisse der
  Mathematik und ihrer Grenzgebiete}.
\newblock Springer--Verlag, Berlin, 1984.

\bibitem{Ahlfors}
L.~V.~Ahlfors.
\newblock {\em Conformal invariants: topics in geometric function theory},
\newblock McGraw--Hill, New York, 1973.

\bibitem{Hajek-1}
0.~H{\'a}jek.
\newblock Notes on meromorphic dynamical systems {I}.
\newblock {\em Czech. Math. J.}, 16(91), 1966, 14--27.

\bibitem{Hajek-2}
0.~H{\'a}jek.
\newblock Notes on meromorphic dynamical systems {II}.
\newblock {\em Czech. Math. J.}, 16(91), 1966, 28--35.

\bibitem{Hajek-3}
0.~H{\'a}jek.
\newblock Notes on meromorphic dynamical systems {III}.
\newblock {\em Czech. Math. J.}, 16(91), 1966, 36--40.

\bibitem{Brickman-Thomas}
L.~Brickman and E.~S.~Thomas.
\newblock Conformal equivalence of analytic flows.
\newblock {\em J. Differential Equations}, 25(3), 1977, 310--324.

\bibitem{Needhan-King}
D.~J.~Needham and A.~C.~King.
\newblock On meromorphic complex differential equations.
\newblock {\em Dynam. Stability Systems}, 9(2), 1994, 99--122.

\bibitem{Garijo-Gasull-Jarque}
A.~Garijo, A.~Gasull, and X.~Jarque.
\newblock Normal forms for singularities of one dimensional holomorphic vector
  fields.
\newblock {\em Electron. J. Differential Equations}, (122), 2004, 7.

\bibitem{Segal}
S.~L.~Segal.
\newblock {\em Nine Introductions to Complex Analysis}, volume 208 of {\em
  North-Holland Mathematics Studies}.
\newblock North--Holland, Amsterdam, rev. ed. edition, 2008.

\bibitem{Gross}
W.~Gross.
\newblock {\"U}ber die {S}ingularit{\"a}ten analytischer {F}unktionen.
\newblock {\em Mh. Math. Phys.}, 29(1), 1918, 3--47.

\bibitem{Zheng}
J.~Zheng.
\newblock {\em Value distribution of meromorphic functions}.
\newblock Tsinghua University Press, Beijing; Springer, Heidelberg, 2010.

\bibitem{Devaney}
R.~L.~Devaney.
\newblock {\em Complex exponential dynamics}.
\newblock {In H. W. Broer et al., editor, {\em Handbook of Dynamical Systems}
volume 3, 125--223,}
\newblock North Holland, Amsterdam, 2010. 

\bibitem{Steinmetz}
N.~Steinmetz.
\newblock {\em Nevanlinna Theory, Normal Families and Algebraic Differential
  Equations}.
\newblock Universitext. Springer, Cham, 2017.

\bibitem{HockettRamamurti}
K.~Hockett and S.~Ramamurti.
\newblock Dynamics near the essential singularity of a class of entire vector
  fields.
\newblock {\em Trans. Amer. Math. Soc.}, 345(2), 1994, 693--703.

\bibitem{Oberhettinger}
F.~Oberhettinger.
\newblock Hypergeometric functions.
\newblock In M.~Abramowitz \emph{et al.}, editor, {\em Handbook of Mathematical
  Functions with Formulas, Graphs, and Mathematical Tables}, volume 9th
  printing, chapter~15, 555--566. Dover Publications, Inc., New York,
  1972.

\bibitem{Alvarez-Mucino-Solorza-Yee}
A.~Alvarez--Parrilla, J.~Muci{\~n}o--Raymundo, S.~Solorza--Calder{\'o}n, and
  C.~Yee--Romero.
\newblock On the geometry, flows and visualization of singular complex analytic
  vector fields on Riemann surfaces.
\newblock {\em Proceedings of the 2018 Workshop in Holomorphic Dynamics, C.
  Cabrera et al. Eds., Instituto de Matem{\'a}ticas, UNAM, M{\'e}xico, Serie
  Papirhos, Actas 1}, 2019, 21--109.

\bibitem{Herring}
M.~E.~Herring.
\newblock Mapping properties of {F}atou components.
\newblock {\em Ann. Acad. Sci. Fenn. Math.}, 23, 1998, 263--274.

\bibitem{Sixsmith}
D.~J.~Sixsmith.
\newblock A new characterization of the {E}remenko--{L}yubich class.
\newblock {\em J. Anal. Math.}, 123, 2014, 95--105.

\bibitem{Lewin}
L.~Lewin.
\newblock {\em Dilogarithms and associated functions}.
\newblock Macdonald, London, 1958.

\bibitem{LopezMucino}
J.~L.~L{\'o}pez and J.~Muci{\~n}o--Raymundo.
\newblock On the problem of deciding whether a holomorphic vector field is
  complete.
\newblock In L.~Ram{\'\i}rez de~Arellano~et al., editor, {\em Complex analysis
  and related topics (Cuernavaca, 1996)}, volume 114 of {\em Oper. Theory Adv.
  Appl.}, 171--195. Birkh{\"a}user, Basel, 2000.

\bibitem{Pommerenke}
Ch.~Pommerenke.
\newblock {\em Boundary Behavior of Conformal Maps}.
\newblock{Grundlehren Math. Wiss. 299,}
\newblock Springer--Verlag, Berlin, 1992.

\bibitem{Guillot}
A.~Guillot.
\newblock Complex differential equations and geometric structures in curves.
\newblock In L.~Hern{\'a}ndez--Lamoneda et~al., editor, {\em Geometrical
Themes Inspired by the $N$--body Problem}, volume 2204 of {\em Lect. Notes
  Math.}, 1--47. Springer, 2018.

\bibitem{Broughan-Barnett}
K.~A.~Broughan and A.~R.~Barnett.
\newblock Linear law for the logarithms of the {R}iemann periods at simple
  critical zeta zeros.
\newblock {\em Math. Comp.}, 75(254), 2005, 891--902.

\bibitem{Steinmetz-Fact}
N.~Steinmetz.
\newblock On factorization of the solutions of the {S}chwarzian differential
  equation $\{w, z\}=q(z)$.
\newblock {\em Funkcialaj Ekvacioj}, 24, 1981, 307--315.

\end{thebibliography}
\end{document}